\documentclass[reqno, 11pt]{article}

\usepackage{fullpage}
\usepackage{newlfont}
\usepackage{amsmath}
\usepackage{amsfonts}
\usepackage{amssymb}
\usepackage{amsthm}
\usepackage{bbm}
\usepackage{dsfont}
\usepackage{algorithm,algorithmic}
\usepackage{graphicx}
\usepackage{enumerate}
\usepackage{subfigure}
\usepackage{hyperref}

\theoremstyle{plain}

\newtheorem{theorem}{Theorem}[section]
\newtheorem{lemma}[theorem]{Lemma}
\newtheorem{proposition}[theorem]{Proposition}

\newtheorem{corollary}[theorem]{Corollary}

\theoremstyle{remark}
\newtheorem{definition}[theorem]{Definition}

\newtheorem{remark}[theorem]{Remark}

\renewcommand{\b}[1]{\boldsymbol{\mathrm{#1}}} 

\newcommand{\wt}{\widetilde}

\newcommand*\di{\mathop{}\!\mathrm{d}} 
\newcommand*\dd{\mathop{}\!\mathrm{d}} 
\newcommand*\ii{\mathop{}\!\mathrm{i}} 
\newcommand{\st}{such that }
\renewcommand{\P}{\mathbb{P}}
\newcommand{\E}{\mathbb{E}}
\newcommand{\R}{\mathbb{R}}
\newcommand{\C}{\mathbb{C}}
\newcommand{\N}{\mathbb{N}}

\newcommand{\e}{\mathrm{e}}

\newcommand{\caN}{{\mathcal N}}

\newcommand{\bbE}{{\mathbb E}}

\newcommand{\bsv}{{\boldsymbol v}}
\newcommand{\bsw}{{\boldsymbol w}}

\newcommand{\bsH}{{\boldsymbol H}}

\newcommand{\p}[1]{(#1)} 
\newcommand{\pb}[1]{\bigl(#1\bigr)}
\newcommand{\pB}[1]{\Bigl(#1\Bigr)}
\newcommand{\pbb}[1]{\biggl(#1\biggr)}

\newcommand{\pa}[1]{\left(#1\right)}

\newcommand{\q}[1]{[#1]} 
\newcommand{\qb}[1]{\bigl[#1\bigr]}
\newcommand{\qB}[1]{\Bigl[#1\Bigr]}

\newcommand{\qa}[1]{\left[#1\right]}

\newcommand{\h}[1]{\{#1\}} 
\newcommand{\hb}[1]{\bigl\{#1\bigr\}}

\newcommand{\abs}[1]{\lvert #1 \rvert} 
\newcommand{\absb}[1]{\bigl\lvert #1 \bigr\rvert}
\newcommand{\absB}[1]{\Bigl\lvert #1 \Bigr\rvert}
\newcommand{\absbb}[1]{\biggl\lvert #1 \biggr\rvert}
\newcommand{\absBB}[1]{\Biggl\lvert #1 \Biggr\rvert}
\newcommand{\absa}[1]{\left\lvert #1 \right\rvert}

\newcommand{\norm}[1]{\lVert #1 \rVert} 

\newcommand{\scalar}[2]{\langle{#1} \mspace{2mu}, {#2}\rangle}

\DeclareMathOperator{\diag}{diag}
\DeclareMathOperator{\tr}{Tr}
\DeclareMathOperator{\Tr}{Tr}

\DeclareMathOperator{\im}{Im}
\newcommand{\beq}{ \begin{equation} }
\newcommand{\eeq}{ \end{equation} }

\newcommand{\expect}[1]{\mathbb{E} \left[ #1 \right] }
\newcommand{\G}[1]{G_{#1}}
\newcommand{\rvline}{\hspace*{-\arraycolsep}\vline\hspace*{-\arraycolsep}}

\renewcommand{\t}{t}

\numberwithin{equation}{section} 
\numberwithin{theorem}{section}

\title{Spectral Properties and Weak Detection in Stochastic Block Models}

\author{%
	Yoochan Han \footnote{Department of Mathematical Sciences, KAIST, Daejeon, 34141, Korea
		\newline email: \texttt{happycuki71@kaist.ac.kr}}, \;
	Ji Oon Lee \footnote{Department of Mathematical Sciences, KAIST, Daejeon, 34141, Korea
		\newline email: \texttt{jioon.lee@kaist.edu}},	\;
	and Wooseok Yang \footnote{Department of Mathematical Sciences, KAIST, Daejeon, 34141, Korea
		\newline email: \texttt{ws.yang@kaist.ac.kr}}
}

\begin{document}

\maketitle

\begin{abstract}
We consider the spectral properties of balanced stochastic block models of which the average degree grows slower than the number of nodes (sparse regime) or proportional to it (dense regime). For both regimes, we prove a phase transition of the extreme eigenvalues of SBM at the Kesten--Stigum threshold. We also prove the central limit theorem for the linear spectral statistics for both regimes. We propose a hypothesis test for determining the presence of communities of the graph, based on the central limit theorem for the linear spectral statistics.
\end{abstract}

\section{Introduction}\label{sec:intro}

We consider the stochastic block model (SBM), one of the most fundamental models for the networks with community structure. We focus on the spectral properties of balanced SBMs in which the difference between the intra-community probability and the inter-community probability is significantly smaller than their average. 

\vskip10pt

\textbf{Stochastic Block Model:} 
A stochastic block model we consider is a graph with $N$ nodes, partitioned into disjoint subsets, called the communities, $C_1, \dots, C_K$ of equal sizes, where the number of the communities $K$ is independent of $N$. Its adjacency matrix $\wt M$ is a symmetric $N \times N$ matrix whose entries are Bernoulli random variables satisfying
\begin{align}\label{def:wt M}
	\P( \wt M_{ij} =1 )=
	\begin{cases}
		p_s & (i \sim j) \\
		p_d & (i\not\sim j)
	\end{cases}, \qquad
	\P( \wt M_{ij} =0 )=
	\begin{cases}
		1-p_s & (i \sim j) \\
		1-p_d & (i\not\sim j)
	\end{cases},
\end{align}
where $i \sim j$ means that $i$ and $j$ are within the same community. We assume that the SBM is balanced, i.e., the communities are with the same size.

For the spectral analysis, it is easier to rescale the adjacency matrix so that the typical size of the eigenvalues is of order one. 
We rescale $\wt M$ via the average edge probability $p_a$ defined as
\beq \label{eq:p_a}
	p_a:=\frac{p_s+(K-1)p_d}{K},
\eeq
which can be obtained from given data. We introduce the rescaled matrix $M$ defined by
\beq \label{eq:M}
	M_{ij} = \frac{\wt M_{ij} - p_a}{\sigma}, \qquad \sigma:=\sqrt{N\cdot\frac{p_s(1-p_s)+(K-1)p_d(1-p_d)}{K}}.
\eeq
With the rescaling, the variance of the entries $M_{ij}$ is $\Theta(N^{-1})$ and it can be checked that the most of the eigenvalues of $M$ are contained in $[-2, 2]$. Note that the entries of $M$ are not centered due to the difference between $p_s$ and $p_d$.

\vskip10pt

\textbf{Spiked Wigner Matrix:}
A spiked Wigner matrix is a random matrix of the form $\lambda X X^T + H$, where the spike $X$ is an $N \times K$ matrix whose column vectors are with the unit norm and $H$ is an $N \times N$ Wigner matrix. The parameter $\lambda$ corresponds to the signal-to-noise ratio (SNR). In this model, with the normalization $\E H_{ij}^2 = N^{-1}$ for $i \neq j$, the largest eigenvalue of $\lambda X X^T + H$ converges to $\lambda + \lambda^{-1}$ if $\lambda > 1$ and to $2$ if $\lambda < 1$. This phase transition is called the Baik--Ben Arous--P\'ech\'e (BBP) transition, after the seminal work of \cite{BBP05} for the phase transition of the largest eigenvalue of a spiked Wishart matrix.

The BBP-transition result suggests that the standard principal component analysis (PCA) can be applied to the detection problem for spiked Wigner matrices in case $\lambda > 1$, which guarantees reliable detection. On the other hand, in case $\lambda < 1$, it is known that reliable detection is impossible if the noise is Gaussian and the spike is rank-$1$ \cite{PWBM18}. In this case, one can consider the weak detection, which is a hypothesis test about the presence of the signal. The likelihood ratio (LR) test is optimal as can be checked from Neyman--Pearson lemma, but one can also construct an optimal test based on the behavior of the linear spectral statistics (LSS) of the eigenvalues \cite{CL19,JCL20}, which is a linear functional defined as
\beq \label{def:LSS}
	L_M(f) :=  \sum_{i=1}^N f(\mu_i(M))
\eeq
for a given function $f$, where $\mu_1(M),\cdots\mu_N(M)$ are the eigenvalues of the matrix $M$. 

The rescaled adjacency matrix $M$ can be viewed as a generalized spiked Wigner matrix with a spike of rank-$(K-1)$ as follows: We first decompose $M$ into $M = \E M + H$. The expectation $\E M$ is a deterministic matrix whose rank is $(K-1)$ and its only non-zero eigenvalue is $\frac{N(p_s-p_d)}{K\sigma}$ with multiplicity $(K-1)$. The noise $H := M - \E M$, which we call a centered SBM, is a random matrix. It can be easily computed that the SNR of the SBM is given by
\begin{align}\label{def:SNR}
	\frac{N(p_s-p_d)^2}{K(p_s+(K-1)p_d)}.
\end{align}
From the fact that $1$ is the threshold for the SNR in the BBP-transition, it can be deduced that it is possible to reliably detect the communities in an SBM if the SNR in \eqref{def:SNR} is larger than $1$. The threshold is called the Kesten--Stigum (KS) threshold, which first appeared in \cite{KS66}. 

\vskip10pt

\textbf{Main Problem:}
The spectral properties, including the location of the largest eigenvalues, of SBMs are largely unknown. The main difficulty in the spectral analysis of the SBM is that the variances of the entries of the noise $H$ are not identical. Furthermore, in the sparse regime where $p_a = o(1)$, the analysis is more involved due to its singular nature that the entries of $M$ are highly concentrated at a single value $-p_a/\sigma$. Our goal is to prove spectral properties of SBM, including the BBP-type transition and the CLT of the LSS, and apply it to the detection problem.

\vskip10pt

\textbf{Main Contribution:}
Our main contributions in this work are as follows:
\begin{enumerate}[(1)]
	\item We prove the eigenvalue phase transition in the dense regime and the sparse regime. (Theorems~\ref{thm:eigval bbp} and~\ref{thm:eigval_bbp_sparse}).
	\item We prove the central limit theorem (CLT) of the LSS in both the dense regime and the sparse regime. (Theorems~\ref{thm:CLT} and~\ref{thm:converge gaussian}).
	\item We propose a test based on the CLT of the LSS. (Theorem~\ref{thm:CLT test}).
  \item We prove the local law for the sparse centered generalized stochastic block model. (Lemma~\ref{lemma:local law cgSBM})
\end{enumerate}

In our work, in terms of the average edge probability $p_a$ in \eqref{eq:p_a}, the dense regime means $p_a = \Theta(1)$ and the sparse regime means $p_a = N^{-c}$ for some constant $c \in (0, 1)$. See also Definitions \ref{def:Csbm H} and \ref{def:Dsbm M}.

Our first main result is the phase transition of the largest eigenvalues of the SBM in both the dense regime and the sparse regime. More precisely, we prove that the largest eigenvalues pop up from the bulk of the spectrum if and only if the SNR is above the KS-threshold. For the proof, we adapt the strategy of \cite{Rao11}, based on an estimate on the resolvent of a random matrix, known as the isotropic local law in random matrix theory. While we can directly apply the isotropic local law for generalized Wigner matrices in the dense regime, the corresponding result was not known in the sparse regime. In Lemma \ref{lemma:local law cgSBM}, we prove a weaker version of the isotropic local law in the sparse regime, which is enough for the proof of the eigenvalue phase transition. The local law we proved in this paper is not simply a tool for the proof of the eigenvalue transition but of great importance per se, since the result itself and also the idea of the proof for it can be used in many other problems on sparse random matrices.

Our second main result is the CLT for the LSS of SBM with general ranks. For Wigner matrices, the proof of the CLT is based on the analysis of the resolvent in Cauchy's integral formula \cite{BY05} or the analysis of the characteristic function \cite{Lytova2009}, and the proof can be extended to more general models by the interpolation with a reference matrix \cite{Lytova2009,CL19}. For the SBM in the dense regime, the reference matrix for the interpolation is a generalized Wigner matrix for which the CLT for the LSS was proved in \cite{LX20}. In the sparse regime, however, the corresponding result is not known and thus as the first step we introduce a centered SBM and prove the CLT for the LSS for it. The proof in the first step requires the ideas from the both methods, the analysis of Cauchy's integral formula and the analysis of the characteristic function. We finish the proof by applying the interpolation method.

With the CLT of the LSS, we follow the ideas in \cite{CL19,JCL20} to propose a hypothesis test between the hypotheses on the number of communities $K$,
\[
	\bsH_1 : K=K_1, \qquad \bsH_2 : K=K_2,
\]
for non-negative integer $K_1 < K_2$, independent of $N$. The test is computationally easy and the idea can also be used for the estimation of the rank of the spike. We prove the limiting error of the proposed test and numerically check its performance.

While the main motivation of the current work lies in the study of the spectral behavior of the SBM, our results naturally extend to more general models. See Definitions \ref{def:Csbm H} and \ref{def:Dsbm M}.

\vskip10pt

\textbf{Related Works:}
The stochastic block model was introduced in the study of social networks \cite{HLL83}. It provides a basic yet fundamental model in various fields of study, most notably in the research for the community detection (and recovery) problem. Several methods have been proposed for the problem, including spectral clustering \cite{KMM13,SB15,Jin15}, maximization of modularities \cite{BC09}, semi-definite programming \cite{GV16}, and penalized ML detection with optimal misclassification proportion \cite{GMZZ17}. See \cite{Abbe17} and references therein for the history and more recent developments.

The spectral properties of spiked Wigner matrices, especially the behavior of the extremal eigenvalues, have been extensively studied in random matrix theory, e.g., \cite{P06,Rao11,BV16}. Such results have been applied to the detection problem for spiked Wigner models \cite{BS16,Lei16}. The limits of detection in this model have been considered in statistical learning theory \cite{MRZ17,PWBM18,EKJ20,CL19}. 

The study of sparse random matrices is also of great importance in random matrix theory. The analysis of sparse random matrices based on the Stieltjes-transform method was initiated in \cite{EKYY12,EKYY13}, where the main objects of study was Erd\H{o}--R\'enyi graphs in the sparse regime. The behavior of the extreme eigenvalues of sparse Erd\H{o}--R\'enyi graphs were considered in various works \cite{LS18,HLY20,HK20,lee2021higher,huang2022edge}. Other related sparse random matrix models were also studied, including the sparse sample covariance matrices \cite{Hwang2018} and sparse SBM \cite{HLY20}.

\vskip10pt

\textbf{Definition of the Model:}
Here, we precisely define the model we consider in this paper.

\begin{definition}[Centered Generalized SBM, cgSBM] \label{def:cgSBM_main} \label{def:Csbm H}
	Fix any $0 < \phi \leq 1/2$. We assume that $H=(H_{ij})$ is a real $N \times N$ block random matrix with $K$ balanced communities with $1\leq K \leq N$, whose entries are independently distributed random variables, up to symmetry constraint $H_{ij}=H_{ji}$. We suppose that each $H_{ij}$ satisfies the moment conditions
	\begin{align}\label{eq:moments}
		\E H_{ij}=0, \qquad \E|H_{ij}|^2=\sigma_{ij}^2,\qquad\E|H_{ij}|^k \leq \frac{(Ck)^{ck}}{Nq^{k-2}},\qquad (k\geq 2),
	\end{align}
	with sparsity parameter $q$ satisfying
	\begin{align} \label{eq:sparsity}
		q = C\cdot N^\phi
	\end{align}
 for some constant $C$.
	Here, we further assume the normalization condition $\sum_{i}\sigma_{ij}^2 =1.$
\end{definition}
Note that the condition (\ref{eq:sparsity}) can be extended to $C_1\cdot N^\phi \leq q \leq C_2 \cdot N^{1/2}$ for some constant $C_1$ and $C_2$.

\begin{definition}[Deformed cgSBM]\label{def:Dsbm M}
	Let $H$ be a centered SBM given in Definition \ref{def:Csbm H}, $K \in \N$ be fixed, $V$ be a deterministic $N \times K$ matrix satisfying $V^T V = I$, and $d_1, \dots, d_K$ be (possibly $N$-dependent) deterministic constants such that $d_1\ge\cdots\ge d_K >0$.
	A rank-$K$ deformed SBM is a matrix $M$ of the form
	\begin{equation}\label{eq:deformation}
		M= H + 	V D V^T, 
	\end{equation}
	where $V D V^T = \sum_{i = 1}^K d_i \b v^{(i)} (\b v^{(i)})^T$ and $D=\diag(d_1, \dots, d_K)$ with $V = [\b v^{(1)}, \dots, \b v^{(K)}]$. Also, each column $v^{(i)}$ of $V$ should have block structure, which means that $v^{(i)}$ can be partitioned into finite number of blocks with the same dimension and the entries in each block have same value.
\end{definition}

Let
\beq \label{eq:gamma}
\gamma_N := \frac{N(p_s-p_d)}{\sigma K}=\sqrt{\frac{N(p_s-p_d)^2}{K(p_s(1-p_s)+(K-1)p_d(1-p_d))}}.
\eeq
Then, the (rescaled) SBM in \eqref{eq:M}is a deformed cgSBM with $d_1=\dots=d_{K-1}= \gamma_N$ and its $H$ is cgSBM satisfying $q^2 = Np_a$. We mostly focus on the case where $\gamma := \lim_{N \to \infty} \gamma_N \in (0, \infty)$, which happens when $|p_s - p_d|$ is sufficiently small; in the dense regime, for example, $|p_s - p_d| = O(N^{-1/2})$.

\textbf{Organization of the Paper:}
The rest of the paper is organized as follows: In Section \ref{sec:main result dense}, we present our main results for the dense regime, including a BBP-like transition for the extreme eigenvalues and the central limit theorem for the linear statistics, and we also propose an algorithm for a hypothesis test for the weak detection, based on the linear spectral statistics. In Section \ref{sec:main result sparse}, we present our main results for the sparse regime, including a BBP-like transition for the extreme eigenvalues and the central limit theorem for the linear statistics. In Section \ref{sec:proof of thm eig bbp}, we prove the local law of the sparse centered generalized stochastic block model and use it to prove the transition for the extreme eigenvalues. In Section \ref{sec:CLT}, we explain the main ideas of our proof of the central limit theorems. We conclude the paper in Section \ref{sec:conclusion} with the summary of our works and possible future research directions. Some results from numerical experiments and the technical details of the proofs can be found in Appendices.

\section{Main Results - Dense Regime}\label{sec:main result dense}
In this section, we consider the cgSBM in the dense regime, satisfying $\phi = 1/2$ so that $q = N^{1/2}$ and in terms of the (rescaled) SBM in \eqref{eq:M},
\[
	p_a = \frac{p_s+(K-1)p_d}{K} = \Theta(1).
\]
\subsection{Eigenvalue phase transition}\label{subsec:bbp}
Our first main result is the following phase transition for the largest eigenvalues, which basically coincides with the BBP-transition.
\begin{theorem}[Eigenvalue phase transition]\label{thm:eigval bbp}
	Let $M$ be a deformed cgSBM that satisfies Definition \ref{def:Dsbm M} with cgSBM $H$ has $\phi=1/2$. The block structure condition of $M$ can be omitted. Denote the ordered eigenvalues of $ M$ by $\lambda_1( M)\ge\cdots \ge \lambda_N( M)$.
	Then, for each $1\le  i \leq K$,
	$$\lambda_i(M) \rightarrow \begin{cases} d_i+d_i^{-1}& \textrm{ if } d_i>1,\\ 
		2 &\textrm{ otherwise,}\end{cases}$$
	as $N \to \infty$. Moreover, for each fixed $i>K$, $\lambda_i(M)\rightarrow 2$ almost surely as $N \to \infty$.
\end{theorem}
We prove Theorem \ref{thm:eigval bbp} in Section \ref{subsec:proof of thm eig bbp dense}. We remark that similar results hold for the smallest eigenvalues $\lambda_{N-i}(M)$ when $d_{K-i} < -1$.

By Theorem \ref{thm:eigval bbp}, if $\gamma = \lim_{N \to \infty} \gamma_N >1$, the number of communities in \eqref{eq:M} can be reliably checked by PCA. 
Results of numerical experiments regarding Theorem \ref{thm:eigval bbp} are provided in Appendix \ref{sec:example BBP}.

\subsection{CLT with perturbation}\label{subsec:clt perturbation}
Our second result in this section is the CLT for the LSS of deformed SBMs. For a precise statement, we introduce the Chebyshev polynomial (of the first kind).
\begin{definition}[Chebyshev polynomial] The $n$-th Chebyshev polynomials of the first kind $T_n$ are obtained from the recurrence relation $T_0(x) = 1$, $T_1(x)=x$ and 
	$$
	T_{n+1}(x) = 2x T_n (x) - T_{n-1}(x).
	$$
\end{definition}

\begin{theorem}[CLT for deformed SBM]\label{thm:CLT}
	Let $M$ be a deformed cgSBM that satisfies Definition \ref{def:Dsbm M} with $d_i =\gamma_N$ for all $i = 1, 2, \dots, K$ with $\gamma_N < 1$ with cgSBM $H$ having $\phi=1/2$. The block structure condition of $M$ can be omitted. Then, for any function $f$ analytic on an open interval containing $[-2, 2]$,
	\begin{align} \label{eq:CLT dSBM}
	 \pB{L_M(f) - N \int_{-2}^2 \frac{\sqrt{4-z^2}}{2\pi} f(z) \di z}  \Rightarrow \mathcal{N}\left(m_K(f), V_0(f)\right)\,.
	\end{align}
	The mean and the variance of the limiting Gaussian distribution are given by
	\begin{align} 
	&m_K(f) = \frac{1}{4} \left( f(2) + f(-2) \right) -\frac{1}{2} \tau_0(f)  - \tau_2(f) + k_4 \tau_4(f) + K \sum_{\ell=1}^{\infty} \gamma_{N}^{\ell} \tau_{\ell}(f),\label{eq:m_k(f)} \\
	&V_0(f) = -\tau_1(f)^2 + 2k_4 \tau_2(f)^2 
		+ 2\sum_{\ell=1}^{\infty} \ell \tau_{\ell}(f)^2\,,\label{def:v_0(f)}
	\end{align}
	where we let
	\[
	\tau_{\ell}(f) = \frac{1}{\pi} \int_{-2}^2 T_{\ell} \left( \frac{x}{2} \right) \frac{f(x)}{\sqrt{4-x^2}} \di x
	\]
	with $T_{\ell}$ be the $\ell$-th Chebyshev polynomial. The parameter $k_4$ is defined by
	\beq \label{eq:k_4}
		k_4 := \frac{1-7p +12p^2 -6p^3}{p(1-p)^2}.
	\eeq
        where p is the limitig value of $q^2/N$
\end{theorem}
The parameter $k_4$ is approximately the sum of the fourth cumulants of $H_{ij}$. Note that the variance $V_0(f)$ of the limiting Gaussian does not depend on $K$. We prove Theorem \ref{thm:CLT} in Appendix \ref{sec:proof of thm CLT}. 

\subsection{Detection}\label{subsec:detection}
Recall that $\bsH_1$ and $\bsH_2$ are the hypotheses such that
\[
	\bsH_1 : K=K_1, \qquad \bsH_2 : K=K_2,
\]
for non-negative integer $K_1 < K_2$, independent of $N$. Note that a hypothesis test between $\bsH_1$ and $\bsH_2$ corresponds to the weak detection for the presence of the community structure if $K_1 = 0$.

Suppose that the value $\gamma_N < 1$ is known and our task is to detect whether the community structure is present from a given data matrix $M$ given in \eqref{eq:M}. If we construct a hypothesis test based on the LSS, it is clear that we need to maximize
\begin{align}
	\absa{ \frac{m_{K_1}(f)-m_{K_2}(f)}{\sqrt{V_0(f)}} }
\end{align}
Following the proof of Theorem 6 in \cite{CL19}, it can be proved that optimal $f$ is of the form $f=C_1\phi_\lambda+C_2$ for some constant $C_1$ and $C_2$, where 
\begin{align}
	\phi_{\gamma_N}(x) := \log \pB{\frac{1}{1-\gamma_N x+\gamma_N^2}}+\gamma_N x +\gamma_N^2 \pB{ \frac{1}{k_4+2} - \frac{1}{2} }x^2.
\end{align}
We thus use a test statistic $L_{\lambda}$ for the hypothesis test, defined as
\begin{align}\label{def:test stat}
	L_{\gamma_N} &:=L_M(\phi_{\gamma_N}) - N \int_{-2}^2 \frac{\sqrt{4-z^2}}{2\pi} \phi_{\gamma_N}(z)  \di z \notag \\
	&= -\log \det \pB{(1+\gamma_N^2)I-\gamma_N M} + \frac{\gamma_N^2 N}{2}  + \gamma_N\tr M + \gamma_N^2 \pB{\frac{1}{k_4+2} -\frac{1}{2}}(\tr M^2 -N).
\end{align} 

For $L_{\lambda}$, we have the following CLT result as a direct consequence of \eqref{eq:CLT dSBM}.

\begin{theorem}\label{thm:CLT test}
			Let $M$ be a deformed cgSBM that satisfies Definition \ref{def:Dsbm M} with $d_i =\gamma_N$  for all $i$'s and cgSBM $H$ has $\phi=1/2$. The block structure condition of $M$ can be omitted. Then,
	\begin{equation*}
		L_{\gamma_N} \Rightarrow \mathcal{N}(m_K,V_0),
	\end{equation*}
	where the mean $m_K$ is given by 
	\begin{align}\label{def:m_K}
		m_K &= m_0 + K\qB{-\log(1-\gamma_N^2) + \gamma_N^2 + \pB{\frac{1}{k_4+2}-\frac{1}{2}}\gamma_N^4 }
	\end{align}
	with
	\begin{align}\label{def:m_0}
		m_0 = -\frac{1}{2}\log (1- \gamma_N^2) -\frac{1}{2}\gamma_N^2 + \frac{k_4\gamma_N^4}{4},
	\end{align}
	and 
	\begin{align}\label{def:V_0}
		V_0 = -2 \log(1-\gamma_N^2) +2\gamma_N^2 + \pB{ \frac{2}{k_4+2} -1}  \gamma_N^4.
	\end{align}
\end{theorem}

We now propose a hypothesis test based on the CLT for the LSS. In this test, described in Algorithm \ref{alg:algorithm1}, for a given (rescaled) adjacency matrix $M$, we compute $L_\lambda$ and compare it with the average of $m_{K_1}$ and $m_{K_2}$,
\begin{align} \label{def:hypho critical value}
	m_{c} :=& \frac{m_{K_1}+m_{K_2}}{2} \notag \\
	=& -\frac{K_1 + K_2 + 1}{2} \log(1-\gamma_N^2) + \left( \frac{K_1 + K_2 - 1}{2} \right) \gamma_N^2
	+ \left( \frac{k_4 - K_1 - K_2}{4} + \frac{K_1 + K_2}{2(k_4+2)} \right) \gamma_N^4\,.
\end{align}
We accept $\bsH_1$ if $L_{\gamma_N} \leq m_c$ and reject it otherwise. The sum of the type-I and type-II errors of the proposed test can be computed as in Section 3 of \cite{EKJ20}, and we state it as a corollary here.

\begin{corollary}\label{coro:alg1 error}
	The error of the test in Algorithm \ref{alg:algorithm1} converges to 
	\begin{equation}\label{eq:alg1 error}
		\emph{erfc}\pbb{   \frac{K_2-K_1}{4}\sqrt{  -\log(1-\gamma_N^2) + \gamma_N^2 + \pB{\frac{1}{k_4+2}-\frac{1}{2}}\gamma_N^4    }  }.
	\end{equation}
	where $\emph{erfc}(\cdot)$ is the complementary error function defined as $\emph{erfc}(x)=\frac{2}{\sqrt{\pi}}\int_{x}^{\infty}e^{-t^2}\di t$.
\end{corollary}

\begin{algorithm}[t]
	\caption{:Hypothesis test}\label{alg:algorithm1}
	\begin{algorithmic}[1]	
		\STATE Data: $\wt M$, parameter $\gamma_N$
		\STATE $M$ $\gets$ matrix given by \eqref{eq:M}, $L_{\gamma_N} \gets$ test statistic in \eqref{def:test stat}, $m_c \gets$ critical value in \eqref{def:hypho critical value}
		\IF{$L_{\gamma_N} \leq m_c$} 
		\STATE Accept $\boldsymbol{H_1}$
		\ELSE 
		\STATE Reject $\boldsymbol{H_1}$
		\ENDIF
	\end{algorithmic}
\end{algorithm}

For a numerical experiment, we generated $1200 \times 1200$ rescaled adjacency matrices with the average probability $p_a = 0.1$ and the SNR $\gamma_N = \sqrt{0.7}$. In Figure \ref{fig:simulation}(a), we plot the histograms of the test statistic $L_{\gamma_N}$ for 10,000 independent samples with $K=0, 1, 2, 3, 4$, respectively. It can be seen from the figure that there is a deterministic shift in the histograms as $K$ increases, which is predicted by \eqref{def:m_K}. 

We also performed the test illustrated in Algorithm \ref{alg:algorithm1} and compare the error from the numerical simulation and the theoretical error of the proposed test in Corollary \ref{coro:alg1 error} for 10,000 independent samples with $K_1=0$ and $K_2 = 1, 2, 3, 4$, respectively, with $p_a = 0.1$ and varying $\lambda$ from $0$ to $\sqrt{0.7}$. The numerical errors of the test closely match the theoretical errors, which are depicted in Figure \ref{fig:simulation}(b).

Theorem \ref{thm:CLT test} can also be used in the estimation of $K$. Since the distance of the means $|m_{K+1} - m_K|$ does not depend on $K$, for a given test statistic $L_{\gamma_N}$, the best candidate for the rank $K$ is the minimizer of $|L_{\gamma_N} - m_K|$. This procedure of the estimation is equivalent to find the nearest non-negative integer of the value
	\begin{align}\label{def:alg2 critical value}
		\kappa' := \frac{L_{\gamma_N}-m_0}{-\log(1-\gamma_N^2) + \gamma_N^2 + \pB{\frac{1}{k_4+2}-\frac{1}{2}}\gamma_N^4 }.
	\end{align}

\begin{figure}[t]
	\centering
	\subfigure[]{
	\includegraphics[width=0.48\textwidth,height=0.2\textheight]{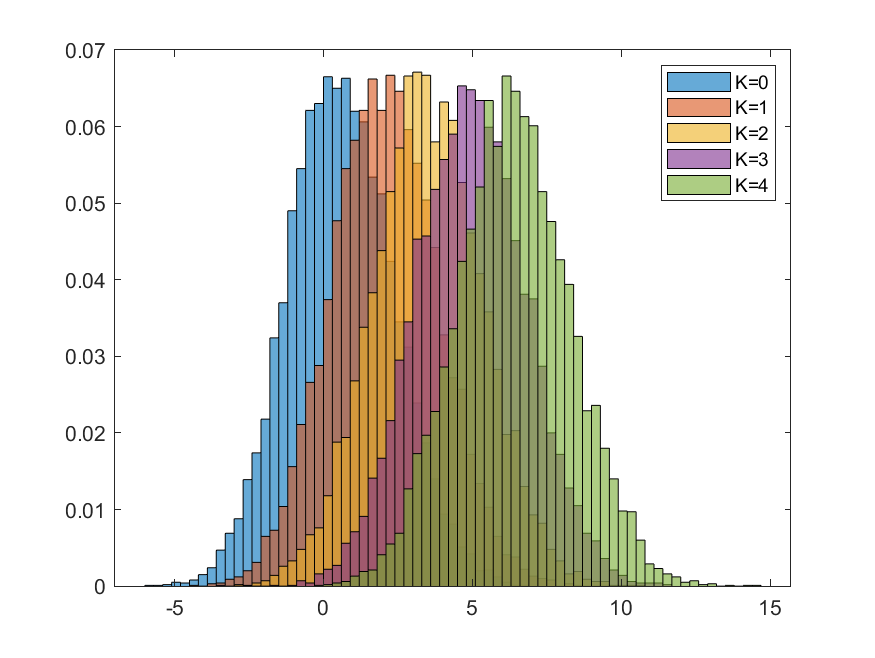}
}
	\centering
	\subfigure[]{
	\includegraphics[width=0.48\textwidth,height=0.2\textheight]{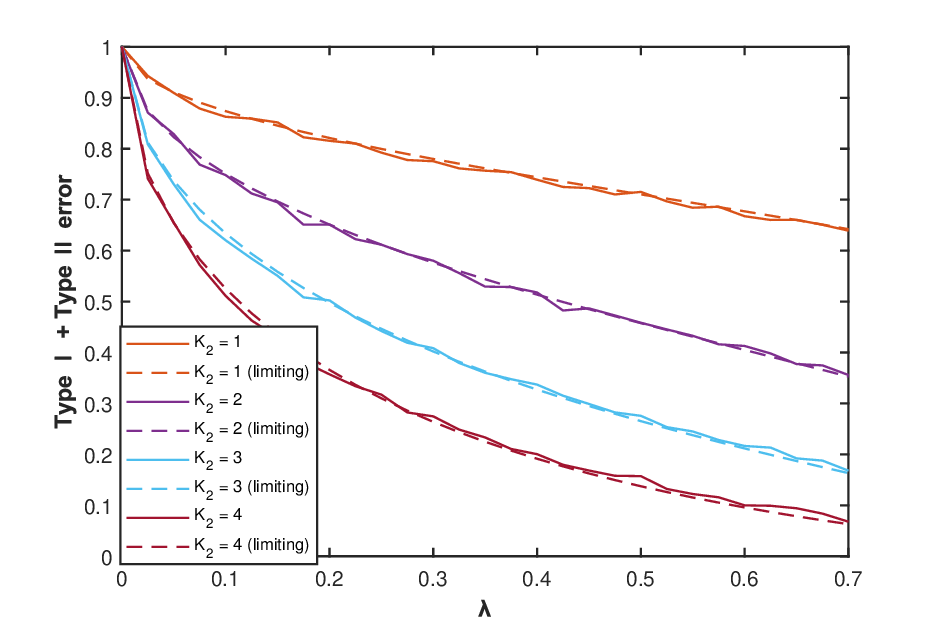}
}
	\caption{ Under the setting in Section \ref{subsec:detection} with $N=1200$ and fixed $p_a=0.1$, (a) the histograms of the test statistic $L_{\gamma_N}$ for $K=0,1,2,3,4$, with $\gamma_N=\sqrt{0.7}$, and (b) the errors from the simulation with Algorithm 1 (solid) versus the limiting errors in \eqref{eq:alg1 error} (dashed) with $K_2=1,2,3,4$.}\label{fig:simulation}
\end{figure}

\section{Main Results - Sparse Regime}\label{sec:main result sparse}
In this section, we consider the cgSBM in the sparse regime, satisfying
\[
0<\phi<1/2
\]
so that $q$ can be less than $N^{1/2}$ and in terms of the (rescaled) SBM in \eqref{eq:M},
\[
	p_a = \frac{p_s+(K-1)p_d}{K} = o(1).
\]
\subsection{Eigenvalue phase transition}
Our first main result in the section is the following phase transition for the largest eigenvalues.  We can get the following phase transition for the largest eigenvalues of the matrix with some sparsity condition.
\begin{theorem}[Eigenvalue phase transition]\label{thm:eigval_bbp_sparse_dcgSBM}
	Let $M$ be a deformed cgSBM that satisfies Definition \ref{def:Dsbm M} with cgSBM $H$ has $1/8<\phi<1/2$. Denote the ordered eigenvalues of $ M$ by $\lambda_1( M)\ge\cdots \ge \lambda_N( M)$.
	Then, for each $1\le  i \leq K$,
	$$\lambda_i(M) \rightarrow \begin{cases} d_i+d_i^{-1}& \textrm{ if } d_i>1,\\ 
		2 &\textrm{ otherwise,}\end{cases}$$
	as $N \to \infty$. Moreover, for each fixed $i>K$, $\lambda_i(M)\rightarrow 2$ almost surely as $N \to \infty$.
\end{theorem}
We prove Theorem \ref{thm:eigval_bbp_sparse_dcgSBM} in section \ref{subsec:proof of thm eig bbp sparse}. Recall the rescaled stochastic block matrix in (\ref{eq:M}). We can get the following corollary for the eigenvalues of $M$ in (\ref{eq:M}).
\begin{corollary} \label{thm:eigval_bbp_sparse}
    Consider the $N\times N$ matrix $M$ in (\ref{eq:M}) with $N^{-3/4}\ll p_a \ll 1$. Denote the ordered eigenvalues of $M$ by $\lambda_1(M)\geq\lambda_2(M)\geq\cdots\geq\lambda_N(M)$. Recall the definition of $\gamma_N$ in \eqref{eq:gamma} and we define the constant $\gamma$ as $\gamma = \lim_{N \to \infty} \gamma_N$. Then, as $N\to\infty$,
    \begin{enumerate}[(a)]
    \item if $0\leq \gamma \leq 1$, $\lambda_i(M)\rightarrow 2$ for each $1\leq i\leq N$. \\
    \item if $1<\gamma<\infty$, $\lambda_i(M)\rightarrow \begin{cases}
            \gamma+\gamma^{-1} & \text{ for } 1\leq  i \leq K-1 \\
            2 & \text{ for } i>K-1 
        \end{cases}$
    \item if $\gamma=\infty$, $\begin{cases}
            \lambda_i(M)-(\gamma_N + \gamma_N^{-1} )\rightarrow 0 & \text{ for } 1\leq  i \leq K-1 \\
            \lambda_i(M)\rightarrow 2 & \text{ for } i>K-1 
        \end{cases}$
    \end{enumerate}
\end{corollary}
We prove Theorem \ref{thm:eigval_bbp_sparse} in Section \ref{subsec:proof of thm eig bbp sparse}. Note that the condition $N^{-3/4}\ll p_a \ll 1$ is from $1/8<\phi<1/2$ since $q^2 = Np_a$.
As in the dense regime, we find that the number of communities can be reliably checked by PCA if $\gamma >1$. 

We remark that in the sparse regime the difference between $\sigma$ in \eqref{eq:M} and its approximation
\[
	\hat{\sigma}=\sqrt{Np_a(1-p_a)}
\]
is negligible in the sense that $|\sigma - \hat{\sigma}| = o(1)$. Thus, in case it is easier to find $p_a$ than $p_s$ and $p_d$, it is possible to use $\hat{\sigma}$ instead of $\sigma$ for the PCA, since the change of the extreme eigenvalues is $o(1)$.

\subsection{CLT with perturbation}\label{subsec:CLT_sparse}
Heuristically, if we assume that Theorem \ref{eq:CLT dSBM} holds also in the sparse regime, then the mean $m_K(f)$ and the variance $V_0(f)$ in Theorem \ref{thm:CLT} will be dominated by the terms containing $k_4$ as a coefficient, since $k_4 \sim p^{-1} \gg 1$ while all other terms are $O(1)$. This in particular suggests that $m_K(f), V_0(f) \sim p^{-1}$ and we need to rescale the LSS by the factor $p^{1/2}$ to observe its fluctuation.

Our main result in this section is the following theorem.
\begin{theorem}\label{thm:converge gaussian}
Suppose that $M$ is a deformed cgSBM satisfying Definition \ref{def:Dsbm M} where the cgSBM $H$ has $1/6<\phi<1/2$. Let $f$ be an analytic function on an open interval containing $[-2,2]$ such that $\tau_2(f) = \Theta(1)$. Then 
\[
    \frac{q}{\sqrt{2N}} \cdot \frac{L_M(f) -\E\q{L_M(f)} }{|\tau_2(f)|} \Rightarrow \caN(0, 1),
\]
where 
the right side is a standard Gaussian random variable. 
\end{theorem}

We prove Theorem \ref{thm:converge gaussian} in Section \ref{sec:proof of prop CLT sparse}.


For the mean $L_M(f)$ in Theorem \ref{thm:converge gaussian}, we have the following expansion formula.
\begin{proposition}[Expectation of LSS]\label{prop:mean of LSS}
Suppose that the assumptions in Theorem \ref{thm:converge gaussian} hold. Then, the expectation $L_M(f)$ satisfies
	\begin{equation}
		\E\qB{ \frac{q}{\sqrt{N}}\pB{ L_M(f) - N \int_{-2}^2 \frac{\sqrt{4-z^2}}{2\pi} f(z) \di z -\frac{q^2}{N}\tau_4(f) }  } \to 0.
	\end{equation}
\end{proposition}
We prove Proposition \ref{prop:mean of LSS} in Appendix \ref{sec:proof of mean}.

Note that Theorem~\ref{thm:converge gaussian} and Proposition~\ref{prop:mean of LSS} can be used to the matrix $M$ in (\ref{eq:M}) with $N^{-3/4}\ll p_a \ll 1$ by the condition $1/6 < \phi < 1/2$. We remark that the condition $1/6<\phi<1/2$ is assumed due to a technical reason and we believe that it can be relaxed to for any $\phi<1/2$.

\section{Phase transition of the largest eigenvalue}\label{sec:proof of thm eig bbp}
\subsection{Phase transition of the largest eigenvalue in the dense regime}\label{subsec:proof of thm eig bbp dense}
In this section, we prove Theorem \ref{thm:eigval bbp} by studying the spectrum of the rank-$K$ deformed SBM in \eqref{eq:deformation}, defined by
\begin{equation*}
	M \;=\; H + V D V^T\,.
\end{equation*}
We first introduce a result for the location of the outlier eigenvalues, which is a special case of Lemma 6.1 in \cite{Rao11}. 
\begin{lemma}\label{lem:phase transition}
	Fix a positive integer $K$, a family $d_1,\ldots, d_K$ of pairwise distinct nonzero real numbers.
	Let us define, for $z\in \C\backslash[-2,2]$, the $K\times K$ matrix 
	\begin{equation}
		M_G(z)=\diag( 1+d_1m_{sc}(z),\ldots, 1+d_Km_{sc}(z)),		
	\end{equation} and denote by $z_1>\cdots > z_p$  the $z$'s such that $M_G(z)$ is singular, where $p\in\{0,\ldots, K\}$ is identically equal to the number of $i$'s such that $-1<1/d_i<1$.
	
	Let us also consider a $K\times K$ Hermitian matrix $M_{n}(z)$, defined on $z \in \C\backslash[a_n, b_n]$ \st the  entries of $M_n(z) $ are analytic functions of $z$ and $[a_n, b_n] \rightarrow [-2,2]$. We suppose that
$M_{n}(z)$ converges, to the function $M_G(z)$, uniformly on $\mathcal{D}:=\{z\in \C : \operatorname{dist}(z, [-2,2])\geq \eta \}$, for all $\eta >0$ as $n \rightarrow \infty$.
	Then	
	\begin{itemize}
		\item There exists $p$ real sequences $z_{n,1}> \ldots > z_{n,p}$ converging respectively to $z_1, \ldots, z_p$   such that for any (small) $\epsilon>0$  , for (large)  $n$  ,   the $z$'s in $\R\backslash[-2-\epsilon, 2+\epsilon]$ such that $M_n(z)$ is singular are exactly $z_{n,1}, \ldots, z_{n,p}$,
	\end{itemize}
\end{lemma}

The following result has been used in several works on finite-rank deformations of random matrices.
\begin{lemma} \label{lem: det identity}
	If $\mu \in \R \setminus \sigma(H)$ and $\det(D) \neq 0$ then $\mu \in \sigma(M)$ if and only if
	\begin{equation*}
		\det \pb{V^T G(\mu) V + D^{-1}}\;=\;0\,.
	\end{equation*}
\end{lemma}
We omit the proof of Lemma \ref{lem: det identity} as it can be checked by an elementary matrix algebra.

We now prove Theorem \ref{thm:eigval bbp}. From Weyl's interlacing inequality, for all $1\le i\le N$,
\begin{equation}\label{eq:weyl}
	 \lambda_{i-K}(H) \leq \lambda_i(M) \leq\lambda_{i}(H),
\end{equation} 
where we use the convention that $\lambda_k(H)=-\infty$ if  $k\le 0$.
Since the empirical spectral distribution of $H$ converges to $\mu_{sc}$, it follows that the empirical spectral distribution of $M$ does as well. Note that for any $N$-independent $i\ge 1$, $\lambda_i(H)\rightarrow 2$.
By \eqref{eq:weyl}, we deduce the that $\liminf_{n\to\infty} \lambda_i(M)\ge 2$ for any $N$-independent $i>1$ and also $\lambda_i(M) \to 2$ if $i>K$ in addition.

Let us now consider the eigenvalues of $M$ outside the spectrum of $H$.
By Lemma \ref{lem: det identity}, these are precisely those values $z$ outside the spectrum of $H$ \st the $K\times K$ matrix
\begin{equation}\label{eq:Mn z}
	M_{N}(z):=V^T G(z) V + D^{-1}
\end{equation}
is singular. 

We first consider the case where all $d_i's$ are pairwise distinct.
From the isotropic local law for generalized Wigner matrices (see Lemma \ref{lem:iso local law} in Appendix), the $(i,j)$-entry of $M_N$ satisfies an estimate
\begin{equation} \label{eq:eig bbp dense local law}
	(M_N)_{i,j} = \scalar{\b v_i}{G(z) \b v_j}+\delta_{ij}d_i^{-1} = m_{sc}(z) \scalar{\b v_i }{\b v_j}+\delta_{ij}d_i^{-1} + O(N^{-1/2+\delta})
\end{equation}
for any $\delta>0$, with overwhelming probability. Thus we have,
\begin{equation}
	(M_N)_{i, j} \to \delta_{ij} \left(m_{sc}(z)+ \frac{1}{d_i} \right).
\end{equation}
Note that this convergence is uniform on $\mathcal{D}$. 

We now apply Lemma \ref{lem:phase transition} to find that 
\begin{enumerate}
\item for $i=1, 2, \dots, p$ with $d_i >1$, the eigenvalues $\lambda_1(M) > \lambda_2(M) > \dots \lambda_p(M)$ are outside $[-2-\epsilon, 2+\epsilon]$ for some $\epsilon>0$ with overwhelming probability and $\lambda_i(M) \to z_i$ where $z_i$ satisfies
\[
	m_{sc}(z_i) + \frac{1}{d_i} =0,
\]
and
\item for $i=p+1, \dots, K$ with $d_i \leq 1$, $\lambda_i(M) \to 2$.
\end{enumerate}

From the identity $m_{sc}^2(z) + zm_{sc}(z)+1 =0$, we easily find that $m_{sc}(z)+d^{-1}=0$ if and only if $d>1$ and $z=d+d^{-1}$. This concludes the proof of Theorem \ref{thm:eigval bbp} in the case $d_i$'s are pairwise distinct. 

Lastly, we consider the case where the $d_i$'s are not necessarily pairwise distinct. If $d_i = d_{i+1}$, for any (small) $\epsilon>0$, using the continuity of $\rho(x) := x + x^{-1}$ for $x>0$, we can choose distinct numbers $d'_i$ and $d'_{i+1}$ such that $\abs{\rho(d_i)-\rho(d'_i)}, \abs{\rho(d_{i+1})-\rho(d'_{i+1})} \leq \epsilon$. We then use the inequality by Hoffman and Wielandt, Corollary 6.3.8 of \cite{HJ85}, to control the change of the eigenvalue of $M$ by the differences $|d_i-d'_i|$ and $|d_{i+1}-d'_{i+1}|$. Putting these results together with Theorem \ref{thm:eigval bbp} for pairwise distinct $d_i$ case, we can show that the Theorem \ref{thm:eigval bbp} holds true for the case where the $d_i$'s are not necessarily pairwise distinct.

\subsection{Phase transition of the largest eigenvalue in the sparse regime}\label{subsec:proof of thm eig bbp sparse}
In this section, we prove Theorem \ref{thm:eigval_bbp_sparse_dcgSBM} and Corollary \ref{thm:eigval_bbp_sparse}. The proof of Theorem \ref{thm:eigval_bbp_sparse_dcgSBM} follows the proof of Theorem \ref{thm:eigval bbp} which presented in Section \ref{subsec:proof of thm eig bbp dense} except the (\ref{eq:eig bbp dense local law}) since Lemma \ref{lem:iso local law} in Appendix can be used only in the dense regime. To prove same equation in the sparse regime, we prove the local law for the sparse centered generalized SBM which is defined in Definition \ref{def:cgSBM_main} with some sparsity condition.

\begin{proposition}[Local law for the sparse cgSBM] \label{lemma:local law cgSBM}
    For centered generalized stochastic block model $H$ defined in Definition \ref{def:cgSBM_main} with $1/8<\phi<1/2$ and its Green function $G$ which is defined as $G(z)=(H-zI)^{-1}$, define the function $s(z)$ by
    \begin{align}
        s(z) = \frac{1}{N}\sum_{i,j}\G{ij}(z)
    \end{align}
    Then, the function $s(z)$ follows
    \begin{align}
        \left\lvert s(z)-\frac{-1}{z+m_{sc}} \right \rvert = \abs{s(z)-m_{sc}(z)} \prec \frac{\sqrt{N}}{q^4}+\frac{1}{q}
    \end{align}
\end{proposition}

This Proposition \ref{lemma:local law cgSBM} will be proved in Appendix~\ref{sec:proof of local law}. Using this Lemma \ref{lemma:local law cgSBM}, we can prove a lemma about the sparse version of (\ref{eq:eig bbp dense local law}) that can finalize the proof of Theorem \ref{thm:eigval_bbp_sparse}. 

\begin{lemma}\label{lemma:innerpro}
    Let $M$ be a deformed cgSBM that satisfies Definition \ref{def:Dsbm M} with cgSBM $H$ has $1/8<\phi<1/2$. Then for all $i,j=1,\cdots,K$, the green function $G(z) = (H-zI)^{-1}$ satisfies
    \begin{align}
        \scalar{\mathbf{v}_i}{G(z)\mathbf{v}_j} = \langle \mathbf{v}_i, \mathbf{v}_j \rangle \left(m_{sc}(z)+O\left(\frac{\sqrt{N}}{q^4}+\frac{1}{q^2}\right)\right)
    \end{align}
    while $\mathbf{v}_i, \mathbf{v}_j\ (i,j=1,\cdots,K)$ are column vectors of $V$.
\end{lemma}

This Lemma \ref{lemma:innerpro} will be proved in Appendix \ref{sec:EM} and it finalizes the proof of Theorem \ref{thm:eigval_bbp_sparse_dcgSBM}. For the proof of Corollary \ref{thm:eigval_bbp_sparse}, recall the $N\times N$ matrix $M$ in (\ref{eq:M}). The proof starts with proving that $M$ is rank-($K-1$) deformed cgSBM in Definition \ref{def:Dsbm M}. The following lemma will be proved in Appendix \ref{sec:EM}.

\begin{lemma} \label{lemma:EM}
    For the expectation matrix $\mathbb{E}M$ of the rescaled SBM $M$ in (\ref{eq:M}), the rank of $\mathbb{E}M$ is $K-1$ and the nonzero eigenvalues of $\mathbb{E}M$ is $\gamma_N$ with multiplicity $K-1$. Consider the orthonormal eigenvectors $\mathbf{v}_i \ (i=1,\cdots,K-1)$ of $\mathbb{E}M$ with respect to eigenvalue $\gamma_N$. Then, $\mathbf{v}_i$ have $K$ blocks with each blocks have $N/K$ elements. Also, each blocks are consisted of same number. In other word, each $\mathbf{v}_i$ have block structure.
\end{lemma}

Using this Lemma \ref{lemma:EM}, $N\times N$ matrix $\bbE M$ can be decomposed as $\bbE M = VDV^T$ where $VDV^T = \sum_{i=1}^{K-1} \gamma_N \mathbf{v}_i \mathbf{v}_i^T$ with $V^TV=I$, $V$ is $N\times (K-1)$ matrix with columns $\mathbf{v}_i$ for $i=1,\cdots,K-1$ and $D = \gamma_N \cdot I_{K-1}$. Also, it can be easily checked that $H=M-\bbE M$ satisfies the Definition \ref{def:cgSBM_main} with $q^2:=Np_a$. Since $N^{-3/4} \ll p_a \ll 1$ implies that $1/8<q<1/2$, the $M$ satisfies the condition for Theorem \ref{thm:eigval_bbp_sparse_dcgSBM}. Directly applying $M$ into Theorem \ref{thm:eigval_bbp_sparse_dcgSBM}, Corollary \ref{thm:eigval_bbp_sparse} can be obtained.

\section{Central Limit Theorems for Stochastic Block Models}\label{sec:CLT}
\subsection{Sketch of Proof of theorem \ref{thm:CLT}}\label{subsec:sketch clt}

Following \cite{CL19,JCL20}, we introduce a family of interpolating matrices
\begin{equation}\label{eq:M(theta)}
	M(\theta) = H + \theta \gamma_N VV^T
\end{equation}
for $\theta \in [0, 1]$ and denote the corresponding eigenvalues of $M(\theta)$ by $\{\lambda_i(\theta)\}_{i=1}^{N}$. 
We choose constants $a \in (2, 3)$ and $v_0 \in (0, 1)$ so that the function $f$ is analytic on the rectangular contour $\Gamma$ with vertices $(\pm a \pm \ii v_0)$. By Theorem \ref{thm:eigval bbp}, we may assume that all eigenvalues of $M$ are inside $\Gamma$. From Cauchy's integral formula,
\beq \begin{split} \label{eq:Cauchy int}
		&\sum_{i=1}^N f(\lambda_i(1)) - N \int_{-2}^2 \frac{\sqrt{4-z^2}}{2\pi} f(z) \di z  \\
		&=-\frac{1}{2\pi \ii} \oint_{\Gamma} f(z) \big( \Tr (M(1) - zI)^{-1} - \Tr (H - zI)^{-1} \big) \dd z \\
		& \qquad +\frac{1}{2\pi \ii} \oint_{\Gamma} f(z) \big( \Tr (H - zI)^{-1} - Nm_{sc}(z) \big) \dd z
\end{split} \eeq
where	we let $m_{sc}(z):=\frac{-z + \sqrt{z^2 -4}}{2}$ be the Stieltjes transform of the Wigner semicircle measure.
Note that the second integral in the right side of \eqref{eq:Cauchy int} converges to a Gaussian, as proved by \cite{LX20}. (See Theorem \ref{thm:general CLT} in Appendix for more detail.)
Define the resolvent $G(z)$ of $H$ and its normalized trace $m^H$ by
\begin{align}\label{eq:resolvent}
	G^H(z)\equiv G(z) := (H-zI)^{-1}, \qquad  m^H(z)\equiv m(z) := \frac{1}{N}\tr G^H(z), 
\end{align}
where $z = E + \ii \eta \in \C^+$. 
From an elementary calculation involving the resolvents,
\begin{equation}
(H-zI)^{-1} - (M(\theta) - zI)^{-1} = \sum_{m=1}^K\theta\sqrt{\lambda} (M(\theta) - zI)^{-1} \b v^{(m)} ( \b v^{(m)})^T (H-zI)^{-1}.
\end{equation}
We can apply the isotropic local law for the inner products $\langle\b v^{(m)}, \b v^{(m)} \rangle$, which asserts that with isotropic local law, Lemma \ref{lem:iso local law}, we obtain that with overwhelming probability,
\begin{align}
	\langle\b v^{(m)},(M(\theta) - zI)^{-1} \b v^{(m)} \rangle = \frac{ m_{sc}(z)}{1+\theta \gamma_N m_{sc}(z)} + O(N^{-1/2 +\delta})
\end{align}
with overwhelming probability. (See Lemma \ref{lem:iso local law} in Appendix for the isotropic local law. See Definition \ref{def:overwhelming probability} in Appendix for the precise definition of overwhelming probability events.)
Using matrix differentiation identities and properties of resolvents, we find that
\begin{align}
	\frac{\partial}{\partial \theta} \Tr (M(\theta) - zI)^{-1} &= -\sum_{m=1}^K \gamma_N \frac{\partial}{\partial z} \left( ({\bold{v}^{(m)}})^T (M(\theta) - zI)^{-1} {\bold{v}^{(m)}} \right) \notag \\
	&= -\frac{K \gamma_N m_{sc}'(z)}{(1+\theta \gamma_N m_{sc}(z))^2} + o(1),
\end{align}
with overwhelming probability.
Integrating over $\theta$ from 0 to 1 and applying Cauchy's integral formula again, we then find that the difference between the LSS of $M$ and that of $H$ is
\[
K \sum_{\ell=1}^{\infty} \gamma_N^\ell \tau_{\ell}(f),
\]
and this proves the desired theorem. See Appendix \ref{sec:proof of thm CLT} for the detailed proof.

\subsection{Proof of Theorem \ref{thm:converge gaussian}}\label{sec:proof of prop CLT sparse}

Our proof of Theorem \ref{thm:converge gaussian} is based on the following proposition about the characteristic function of the (rescaled) LSS of a centered SBM $H$.

\begin{proposition}\label{prop:CLT sparse}
	Suppose that $H$ satisfies conditions in Definition \ref{def:Csbm H} with $1/6<\phi<1/2$. Let $f$ be an analytic function on an open interval containing $[-2,2]$ and define
	\begin{equation}
		\phi(\t) := \E[\exp\hb{\ii \t q^2(L_H(f) -\E\q{L_H(f)}   )/N }]. \qquad (\t \in \R)
	\end{equation}
	Then, with overwhelming probability, the characteristic function $\phi$ satisfies 
	$$
	\phi'(\t)=-2\t \phi(\t) \tau_2(f)^2+ o(1).
	$$
\end{proposition}

 Applying the Arzel\`a-Ascoli theorem and L\'evy's continuity theorem with Proposition \ref{prop:CLT sparse}, we can easily see that the CLT result in Theorem \ref{thm:converge gaussian} holds. To prove Theorem \ref{thm:converge gaussian} for deformed SBM $M$ that replaced by $H$, we follow the proof of Theorem \ref{thm:CLT} to show that
\[
	\frac{q}{\sqrt{N}}\pB{\frac{\partial}{\partial \theta} \Tr (M(\theta) - zI)^{-1} } = o(1)
\]
with overwhelming probability, which would imply that the desired theorem holds for $M$.

We now prove Proposition \ref{prop:CLT sparse}. 
From Cauchy's integral formula,
\[ \begin{split}
\frac{q}{\sqrt{N}}\pb{ L_H(f) - \E\q{L_H(f)} } &= -\frac{q\sqrt{N}}{2\pi\ii} \oint_\Gamma f(z)  \qb{m(z)-\E m(z)} \di z\\
	&= -\frac{q}{2\pi\ii\sqrt{N}} \oint_\Gamma f(z)  \qb{\tr G(z)-\E\tr G(z) } \di z.
\end{split} \]
Recall that vertices of the contour $\Gamma$ are $(\pm a \pm \ii v_0)$ for constants $a \in (2, 3)$ and $v_0 \in (0, 1)$. We rewrite the characteristic function $\phi(\lambda)$ as
\beq \begin{split} \label{eq:phi_t}
	\phi(\t)  &:= \E \qb{\exp\hb{\ii\t \frac{q}{\sqrt{N}}(L_H(f) -\E\q{L_H(f)}   ) }}\\
	&=\E  \qb{\exp\hb{-\frac{\t q}{2\pi\sqrt{N}} \oint_\Gamma f(z)(\tr G(z)-\E\tr G(z)) \di z    }}. \qquad (\t \in \R)
\end{split} \eeq
If we decompose the contour $\Gamma$ into
\begin{equation}\label{eq:contour Gamma1,2}
	\Gamma_1 :=  \{ z=E+\ii \eta \in \Gamma : \eta \leq N^{-5}    \}, \qquad 	\Gamma_2 :=  \{ z=E+\ii \eta \in \Gamma : \eta > N^{-5}    \},
\end{equation} 
it can be easily checked that the obvious that in the integral appearing in the right-side of \eqref{eq:phi_t}, the contribution from the integral on $\Gamma_1$ is negligible, i.e., $o(1)$, in $\phi'(\t)$ and $\phi(\t)$. We thus consider the integral on $\Gamma_2$ only.

Differentiating $\phi(\t)$, we get 
\beq \label{eq:char fct diff}
	\phi'(\t)=-\frac{\t q}{2\pi\sqrt{N}} \int_{\Gamma_2} f(z)\E \qb{ e(\t)(\tr G(z)-\E\tr G(z))  }\di z + o(1),
\eeq
where we let
\beq
	e(\t) := \exp\hb{-\frac{\t q}{2\pi\sqrt{N}} \int_{\Gamma_2} f(z)(\tr G(z)-\E\tr G(z)) \di z    }.
\eeq
We thus focus on $\E \qb{ e(\t)\cdot(\tr G(z)-\E\tr G(z)) }$ for which we have the following estimate.

\begin{lemma} \label{lem:E e Tr G} 
Assume the conditions in Proposition \ref{prop:CLT sparse}. Then, uniformly on $z \in \Gamma_2$,
\beq \begin{split}
		&\E \qb{ e(\t)\cdot(\tr G(z)-\E\tr G(z)) } \\ 
		&= -\E [e(\t)]\frac{\t\sqrt{N}}{\pi q} m_{sc}(z)m_{sc}'(z)\oint_{\Gamma}f(z')m_{sc}(z')m_{sc}'(z') \di z'  + o(\frac{\sqrt{N}}{q}),
\end{split} \eeq
	with overwhelming probability.
\end{lemma}
Lemma \ref{lem:E e Tr G} is proved in Appendix \ref{sec:lem E e Tr G}. Combining Lemma \ref{lem:E e Tr G} with \eqref{eq:char fct diff}, we get 
\beq
	\phi'(\t) = \frac{\t}{2\pi^2}\phi(\t) \pB{\oint_{\Gamma}f(z)m_{sc}(z)m_{sc}'(z)dz}^2 + o(1) = -2\t \phi(\t)\tau_2(f)^2 + o(1)
\eeq
	with overwhelming probability and this concludes the proof of the Proposition \ref{prop:CLT sparse}.

\section{Conclusion and Future Works}\label{sec:conclusion}

In this paper, we considered spectral properties of $N \times N$ balanced stochastic block models with general number of communities. We first proved a BBP-like phase transition of extreme eigenvalues with the threshold equal to the Kesten--Stigum threshold for both the dense regime and the sparse regime. We then proved the central limit theorem for the linear spectral statistics for both the dense regime and the sparse regime, where in the dense regime the variance of the limiting Gaussian distribution does not depend on the number of communities while the mean depends on it. Exploiting this property, we proposed a hypothesis test based on the linear spectral statistics to determine the number of communities. We also provided the theoretical error of the proposed test and numerically checked the accuracy of the test. For the proof of the BBP-like phase transition of extreme eigenvalues for the sparse regime, we prove the local law for the sparse centered generalized stochastic block model. 

A possible future research direction is to extend our results in the dense regime to the sparse regime, especially the phase transition of the extremal eigenvalues and the optimal test statistic for a hypothesis test. We also hope to generalize the results in this paper to non-balanced stochastic block models or non-symmetric models such as directed graphs or bipartite graphs with community structure.


\appendix

\section{Simulations for Theorem \ref{thm:eigval bbp}}\label{sec:example BBP}
In Appendix \ref{sec:example BBP}, we provide the results from the numerical simulation on the outlier eigenvalues. We fix $N=8000$, $K=4$ and $p_a = 0.1$. In Figure \ref{fig:simulation BBP}, we compare the eigenvalue distributions of the matrices with $\lambda=1.5$ and with $\lambda=0.5$. As can be seen from the figure, it can be seen that even if all other conditions are the same, the presence or absence of an outlier is determined according to the value of $\lambda$. As $\lambda$ increases, the gap between the outlier and the support of the semicircle distribution increases, and as the $\lambda$ approaches 1, the gap gradually decreases. When the $\lambda$ is less than 1, the two distributions are mixed and cannot be distinguished.
\begin{figure}[t]
	\centering
	\subfigure[]{
		\includegraphics[width=0.45\textwidth,height=0.15\textheight]{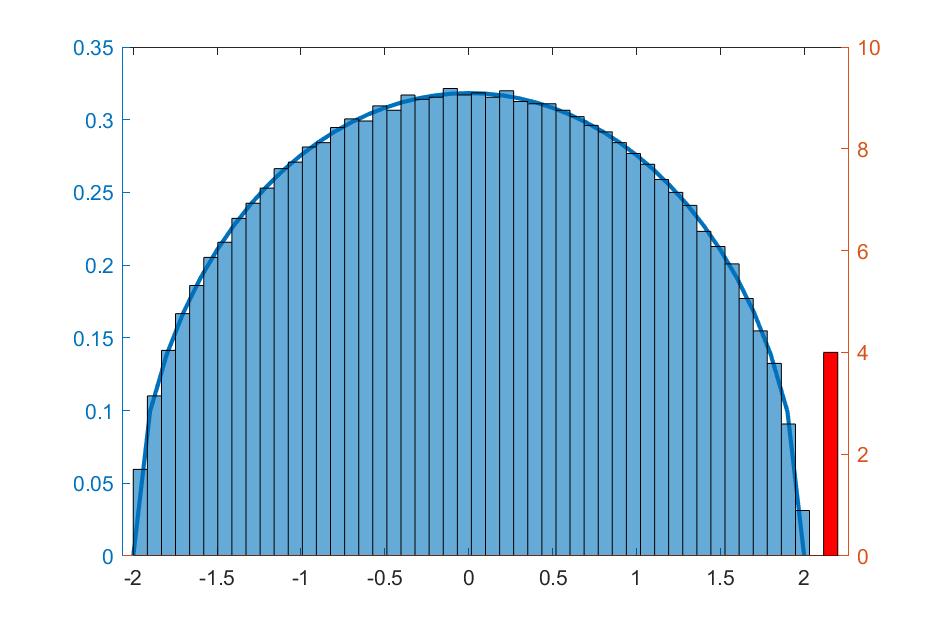}
	}
	\centering
	\subfigure[]{
		\includegraphics[width=0.45\textwidth,height=0.15\textheight]{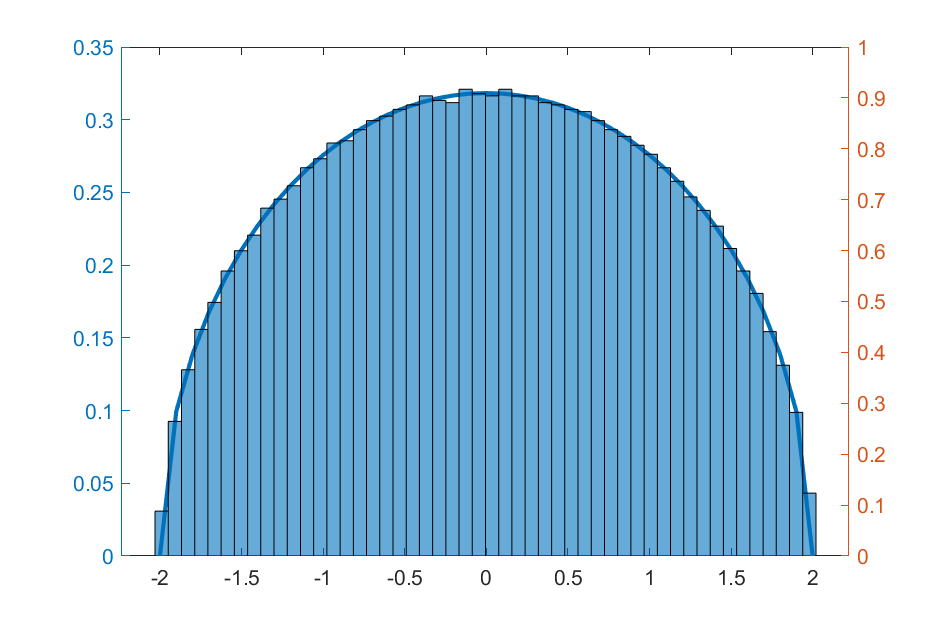}
	}
	\caption{ Under the setting in Theorem \ref{thm:eigval bbp} with fixed $N=8000$, $K=4$ and $p_a = 0.1$, the empirical eigenvalue distribution (a) with $\lambda=1.5$ and (b) with $\lambda=0.5$. Red bar	shows the number of outliers. }\label{fig:simulation BBP}
\end{figure}

\section{Preliminaries}\label{sec:prelim}
In Appendix \ref{sec:prelim}, we explain the basic definitions and theorems for our results. We first introduce more general class of matrix model that generalize Definition \ref{def:Csbm H}.
\begin{definition}[Centered Generalized SBM, cgSBM] \label{def:cgSBM} 
	Fix any $\phi < 1/2$. We assume that $H=(H_{ij})$ is a real $N \times N$ block random matrix with $K$ balanced communities with $1\leq K \leq N$, whose entries are independently distributed random variables, up to symmetry constraint $H_{ij}=H_{ji}$. We suppose that each $H_{ij}$ satisfies the moment conditions
	\begin{align}
		\E H_{ij}=0, \qquad \E|H_{ij}|^2=\sigma_{ij}^2,\qquad\E|H_{ij}|^k \leq \frac{(Ck)^{ck}}{Nq^{k-2}},\qquad (k\geq 2),
	\end{align}
	with sparsity parameter $q$ satisfying
	\begin{align}
		N^{\phi}\leq q \leq N^{1/2}.
	\end{align}
	Here, we further assume the normalization condition $\sum_{i}\sigma_{ij}^2 =1.$
\end{definition}
We  also shall study the spectrum of the deformed cgSBM that generalize Definition \ref{def:Dsbm M} in the sense that entries do not have to be Bernoulli distribution. Note that the matrix $M$ defined in \eqref{eq:M} is the special case of rank-$(K-1)$ deformation of cgSBM with $q^2:=Np_a$.

\begin{definition}[Deformed generalized stochastic block model, dgSBM]\label{def:dgSBM}
	Let $H$ be a centered SBM given in Definition \ref{def:cgSBM}, $K \in \N$ be fixed, $V$ be a deterministic $N \times K$ matrix satisfying $V^T V = I$, and $d_1, \dots, d_K \in \R \setminus \{0\}$ be deterministic numbers such that  $d_1\ge\cdots\ge d_K >0$. We also use the notation $V = [\b v^{(1)}, \dots, \b v^{(K)}]$, where $\b v^{(1)}, \dots, \b v^{(K)} \in \R^N$ are orthonormal.
	Then we define the rank-$K$ deformed SBM by
	\begin{equation}
		M= H + 	V D V^T, 
	\end{equation}
	where $V D V^T = \sum_{i = 1}^K d_i \b v^{(i)} (\b v^{(i)})^T$ and $D=\diag(d_1, \dots, d_K).$
\end{definition}

We denote by $\kappa_{ij}^{(k)}$ the $k$-th cumulant of $H_{ij}$. Under the moment condition in Definition \ref{def:cgSBM},
\begin{align}
	\kappa_{ij}^{(1)}=0, \qquad	|\kappa_{ij}^{(k)}| = O_k \pB{  \frac{1}{Nq^{k-2}}} ,\quad\qquad (k\geq 2) .
\end{align}
For our model with the block structure, we abbreviate $\kappa_{ij}^{(k)}$ as
\begin{align} 
	\kappa_{ij}^{(k)}=
	\begin{cases}
		\kappa_{s}^{(k)} & \text{ if } i \text{ and } j \text{ are within same community}, \\
		\kappa_{d}^{(k)} & \text{ otherwise }.
	\end{cases}
\end{align}
We will also use the normalized cumulants, $s^{(k)}$, by setting
\begin{align}
	s_{(\b{\cdot})}^{(1)}:=0, \quad s_{(\b{\cdot})}^{(k)}:=Nq^{k-2}\kappa_{(\b{\cdot})}^{(k)}, \quad (k\geq 2).
\end{align}
For convenience, we define the parameters $\zeta$ and $\xi^{(4)}$ as  
\begin{equation}\label{eq:xi,zeta}
	\zeta : = \frac{s_s^{(2)} - s_d^{(2)}}{K}=\frac{N(\kappa_s^{(2)} - \kappa_d^{(2)})}{K} , \qquad \xi^{(4)} :=  \frac{s_s^{(4)}+(K-1)s_d^{(4)}}{K}.
\end{equation}
We introduce some notations of basic definitions.
\begin{definition}[Overwhelming probability events]\label{def:overwhelming probability}
	We say that an $N$-dependent event $\Omega \equiv \Omega^{(N)}$ holds with overwhelming probability if for any (large) $D>0$,
	\[\P(\Omega^c) \leq N^{-D}, \]
	for $N\geq N_0(D)$ sufficiently large.
\end{definition}

\begin{definition}[Stochastic domination] Let $X \equiv X^{(N)}, Y \equiv Y^{(N)}$ be $N$-dependent non-negative random variables. We say that $X$ stochastically dominates $Y$ if, for all small $\epsilon>0$ and large $D>0$,
	\begin{align}
		\mathbb{P}(X^{(N)}>N^\epsilon Y^{(N)})\leq N^{-D},
	\end{align}
	for sufficiently large $N\geq N_0(\epsilon,D)$, and we write $X\prec Y$. When $X^{(N)}$ and $Y^{(N)}$ depend on a parameter $u \in U$, then we say $X(u) \prec Y(u)$ uniformly in $u\in U$ if the threshold $N_0(\epsilon,D)$ can be chosen independently of $u$. We also use the notation $X = O_\prec(Y)$ if $|X| \prec |Y|$ and $X = o_\prec(Y)$ if $|X| \prec |YN^{-\epsilon'}|$ for some sufficiently small fixed constant $\epsilon'$. 
\end{definition}
Throughout rest of this paper, we choose $\epsilon>0$ sufficiently small. More precisely, it is smaller than $(1/2-\phi)/20$, where $\phi$ is the fixed parameter in Definition \ref{def:cgSBM} above.

\begin{definition}[Stieltjes transform]
	For given a probability measure $\nu$, we define the Stieltjes transforms of $\nu$ as
	\[m_\nu(z):= \int \frac{\nu(\di x)}{x-z}, \qquad (z\in \C^+)  \] 
\end{definition}
For example, the  Stieltjes transform of the $\emph{semicircle measure}$,
\[ \rho_{sc}(\di x):=\frac{1}{2\pi}\sqrt{(4-x^2)_+} \di x, \]
is given by
\[ m_{sc}(z) = \int \frac{\rho_{sc}(\di x)}{x-z} = \frac{-z + \sqrt{z^2 -4}}{2}, \]
where the square root $\sqrt{z^2-4}$ is chosen so that $m_{sc}(z) \in \C^+$ for $z \in \C^+$ and $\sqrt{z^2-4} \sim z$ as $z \to \infty $.
Clearly, we have 
\[ m_{sc}(z) + \frac{1}{m_{sc}(z)}+z =0.  \] 
\begin{definition}[Green function(Resolvent)]
	Given a real symmetric matrix $H$ we define its resolvent or Green function, $G(z)$, and the normalized trace of its Green function, $m^H$, by
	\begin{align}\label{eq:greenfunction}
		G^H(z)\equiv G(z) := (H-zI)^{-1}, \qquad  m^H(z)\equiv m(z) := \frac{1}{N}\tr G^H(z), 
	\end{align}
	where $z = E + \ii \eta \in \C^+$ and $I$ is the $N\times N$ identity matrix. 
\end{definition}
Denoting by $\lambda_1 \geq \lambda_2 \geq \cdots \geq \lambda_N$ the ordered eigenvalues of $H$, we note that $m^H$ is the Stieltjes transform of the empirical eigenvalue measure of $H$, $\mu^H$, defined as 
\[
\mu^H := \frac{1}{N}\sum_{i=1}^{N}\delta_{\lambda_i}.
\]
In the rest of this section, we state the lemmas and theorems used to prove our results.
\begin{lemma}[Cumulant expansion, generalized Stein's lemma]\label{lemma:Stein}
	Let $F\in C^{\ell +1} (\mathbb{R}; \mathbb{C}^{+})$ and fix $\ell\in\mathbb{N}$. Let $Y$ be a centered random variable with finite moments to order $\ell +2$. Then,
	\begin{equation}
		\mathbb{E}[YF(Y)] = \sum_{r=1}^{\ell} \frac{\kappa ^{(r+1)}(Y)}{r!} \mathbb{E}[F^{(r)}(Y)] + \mathbb{E}[\Omega_{\ell}(YF(Y))],
	\end{equation}
	where $\mathbb{E}$ denotes the expectation with respect to $Y$, $\kappa^{(r+1)}(Y)$ denotes the $(r+1)$-st cumulant of~$Y$ and $F^{(r)}$ denotes the $r$-th derivative of the function~$F$. The error term $ \Omega_{\ell}(YF(Y))$ satisfies
	\begin{align}
		\mathbb{E}[\Omega_{\ell}(YF(Y))]& \leq C_{\ell} \mathbb{E}[|Y|^{\ell +2}] \sup_{|t|\leq Q} |F^{(\ell +1)}(t)| + C_{\ell}\mathbb{E}[|Y|^{\ell +2} \mathds{1}(|Y|>Q) \sup_{t\in \mathbb{R}}|F^{(\ell +1)}(t)|,
	\end{align}
	where $Q>0$ is an arbitrary fixed cutoff and $C_{\ell}$ satisfies $C_{\ell}\leq(C\ell)^\ell / \ell !$ for some constant $C$.
\end{lemma}

\begin{theorem}[Local law, \cite{HLY20}] \label{thm:locallaw}
	Let $H$ satisfy Definition \ref{def:cgSBM} with $\phi>0$. Then, there exist an algebraic function $\wt{m}:\mathbb{C}^+ \rightarrow \mathbb{C}^+$ such that the following hold:
	\begin{align}\label{thm:stronglaw}
		|m(z) -\wt{m}(z)| &\prec \frac{1}{q^2}+\frac{1}{N\eta},\\
		\max_{i,j}|G_{ij}(z)-\delta_{ij}\wt{m}| &\prec \frac{1}{q} + \frac{1}{\sqrt{N\eta}},
	\end{align}
	uniformly on the domain $\mathcal{D}_\ell:=\{ z=E+\ii \eta \in \C^+ : |E|<3, N^{-1+\ell} <\eta \leq 3\} $ ,where $\ell$ is a small positive constant.
\end{theorem}
\begin{theorem}[\cite{HLY20}]\label{thm:matrix norm}
	Suppose that $H$ satisfies Definition \ref{def:cgSBM} with $\phi >0$. Then,
	\begin{equation}
		\abs{\norm{H} -L} \prec \frac{1}{q^4} + \frac{1}{N^{2/3}}.
	\end{equation}
 where  $L:= 2+\frac{\xi^{(4)}}{q^2} +O(q^{-4})$.
\end{theorem}

Combining Theorem  \ref{thm:locallaw} and \ref{thm:matrix norm}, we can prove the local law on the outside of the spectrum. Since the proof is very similar to \cite{EKYY13_2} or \cite{BEKYY14}, we omit the details.
\begin{theorem}[Local law outside the spectrum]\label{thm:local law outside}
	Suppose that $H$ satisfies Definition \ref{def:cgSBM} with $\phi >1/6$. Then,
	\begin{align}
		|m(z) -\wt{m}(z)| &\prec \frac{1}{q^2}+\frac{1}{N\sqrt{(\kappa+\eta)}},\\
		\max_{i,j}|G_{ij}(z)-\delta_{ij}\wt{m}| &\prec \frac{1}{q} + \frac{1}{\sqrt{N}(\kappa+\eta)^{1/4}},
	\end{align}
	uniformly on the domain $\mathcal{D}_\tau:=\{ z=E+\ii \eta \in \mathcal{D}_\ell : |E-L|\geq N^{-2/3+\tau}\}$.
\end{theorem}
Note that $\wt{m}$ therein Theorem \ref{thm:locallaw} and Theorem \ref{thm:local law outside} may be replaced by $m_{sc}$ without changing the error bound.
We also have the following averaging fluctuation results for the monomials in the Green function entries.
\begin{theorem}[Theorem 4.8, Proposition 3.3 in \cite{EKY13}]\label{thm:local law averaging}
	Under Definition \ref{def:cgSBM}, the following estimate hold for $z \in \mathcal{D}_\ell$ uniformly:
	\begin{equation}
		\absB{ \sum_{i=1}^{N}\kappa_{ij}^{(2)}G_{ii}(z) - m_{sc}(z)  } \prec \rho \Psi^2(z),
	\end{equation}
	where $\rho = O(\frac{1}{\sqrt{\kappa+\eta}})$ and $\max_{i,j}\abs{ G_{ij}(z) - \delta_{ij}m_{sc}(z) } \prec \Psi(z).$
\end{theorem}
On the contour $\Gamma_2$ defined in \eqref{eq:contour Gamma1,2}, Theorem \ref{thm:local law outside} and Theorem \ref{thm:local law averaging} can be expressed simply as following:
\begin{proposition}\label{prop:local law on contour} Under Definition \ref{def:cgSBM}, for $z \in \Gamma_2$ given in \eqref{eq:contour Gamma1,2} uniformly,
	\begin{align}
		&|m(z) -{m}_{sc}(z)| \prec \frac{1}{q^2}, \qquad \max_{i,j}|G_{ij}(z)-\delta_{ij}m_{sc}| \prec \frac{1}{q}  , \\
	    &\absB{ \sum_{i=1}^{N}\kappa_{ij}^{(2)}G_{ii}(z) - m_{sc}(z)  } \prec \frac{1}{q^2}.
	\end{align}
\end{proposition}

\begin{lemma}[Basic properties of $m_{sc}$]\label{lem:basic properties of m} Define the distance to the spectral edge
	\begin{equation}
		\kappa \equiv \kappa(E) := \big| |E| -2 \big|.
	\end{equation}
	Then for $z \in D_\ell$ we have
	\begin{equation}\label{eq:m_z}
		|m_{sc}(z)| \sim 1, \qquad \qquad |1-m_{sc}^2| \sim \sqrt{\kappa+\eta}
	\end{equation}
	and 
	\begin{equation}\label{eq:imm_z}
		\emph{Im }  m_{sc}(z) \sim \begin{cases}
			\sqrt{\kappa+\eta} &\text{if $E \leq 2$}\\
			\frac{\eta}{\sqrt{\kappa+\eta}} &\text{if $E \geq 2$}.\\
		\end{cases}
	\end{equation}
	Moreover,
	\begin{equation}
		m_{sc}'(z)= -\frac{m_{sc}(z)}{z+2m_{sc}(z)}=\frac{m_{sc}^2(z)}{1-m_{sc}^2(z)}.
	\end{equation}
\end{lemma}
\begin{proof}
	The proof is an elementary calculation; see Lemma 4.2 in \cite{EYY11}.
\end{proof}

\section{Proof of Theorem \ref{thm:CLT}}\label{sec:proof of thm CLT}
We first express the left-hand side of \eqref{eq:CLT dSBM} by using a contour integral via Cauchy's integration formula. The integral is then written in terms of the Stieltjes transforms of the empirical spectral measure and the semicircle measure. Since the Stieltjes transform of the empirical spectral measure converges weakly to a Gaussian process, we find that the linear eigenvalue statistic also converges to a Gaussian random variable.
We introduce the CLT for the centered stochastic block models. 
\begin{theorem}[\cite{LX20}] \label{thm:general CLT}
	Let $H$ be a  cgSBM defined in defintion \ref{def:cgSBM_main}. Then, for any function $f$ analytic on an open interval containing $[-2, 2]$,
	\[
	\pB{L_H(f) - N \int_{-2}^2 \frac{\sqrt{4-z^2}}{2\pi} f(z) \di z} \Rightarrow \mathcal{N}\left(m_0(f), V_0(f)\right)\,.
	\]
	where the mean $m_0(f)$ is given by \eqref{eq:m_k(f)} with $K=0$ and the variance $V_0(f)$ is given in \eqref{def:v_0(f)}.
\end{theorem}
Recall the definitions of  $M(\theta)$ , 	 and contour $\Gamma$ from section \ref{subsec:sketch clt}. 
Define the resolvent $R(\theta, z)$ and the normalized trace of the resolvent $m_N(\theta,z)$ by
\begin{equation}
	R(\theta, z):= (M(\theta) - zI)^{-1}, \qquad m_N(\theta,z):=\frac{1}{N}\Tr R(\theta,z)=\frac{1}{N}\sum_{i = 1}^N\frac{1}{\lambda_{i}(\theta)-z}
\end{equation}
Then with Cauchy's integral formula, we have
\beq \begin{split} \label{eq:decouple2}
	\sum_{i=1}^N f(\lambda_i(1)) - N\langle\mu_{sc},f\rangle  &= -\frac{N}{2\pi \ii} \oint_{\Gamma} f(z) \big( m_N(1,z) - m_{sc}(z) \big) \dd z \\
	&= -\frac{1}{2\pi \ii} \oint_{\Gamma} f(z) \big( \Tr R(1,z) - \Tr R(0,z) \big) \dd z\\
	&\qquad  -\frac{1}{2\pi \ii} \oint_{\Gamma} f(z) \big( \Tr R(0,z) - Nm_{sc}(z) \big) \dd z.
\end{split} \eeq 
where the fluctuation result for \eqref{eq:decouple2} is already given by Theorem \ref{thm:general CLT}.
To analyze \eqref{eq:decouple2}, we use the results from the random matrix theory like isotropic local law given by Lemma \ref{lem:iso local law} below.

Hence, our strategy of the proof became to show that the limiting distribution of $ \Tr R(\theta,z)$ has deterministic shift when we integrate with respect to $\theta$. More precisely, we claim that
\begin{equation} \label{eq:fundamental}
	\frac{\partial}{\partial\theta} \Tr R(\theta,z) =  -\frac{K \sqrt{\lambda} m_{sc}'(z)}{(1+\theta \sqrt{\lambda} m_{sc}(z))^2} +O_\prec(N^{-\frac{1}{2}})
\end{equation}
uniformly on $z \in \Gamma$. Once we prove the claim, we can use the lattice argument to prove Theorem \ref{thm:CLT} as follows: Choose points $z_1, z_2, \dots, z_{16N} \in \Gamma$ so that $|z_i - z_{i+1}| \leq N^{-1}$ for $i=1, 2, \dots, 16N$ (with the convention $z_{16N+1} = z_1$). For each $z_i$, the claim \eqref{eq:fundamental} shows that
\begin{equation}
	\Tr R(1,z_i) - \Tr R(0,z_i) =   -\frac{K \sqrt{\lambda} m_{sc}'(z)}{1+ \sqrt{\lambda} m_{sc}(z)}  + O_\prec(N^{-\frac{1}{2}}).
\end{equation}
For any $z \in \Gamma$, if $z_i$ is the nearest lattice point from $z$, then $|z-z_i| \leq N^{-1}$. From the Lipschitz continuity of $\Tr R$, we then find $|\Tr R(\theta,z) - \Tr R(\theta,z_i)| = O_\prec(N^{-1})$ uniformly on $z$ and $z_i$. Hence, by triangular inequality, we can show that 
\beq \begin{split}
	&|\xi_N(1, z) - \xi_N(0, z)  +\frac{K \sqrt{\lambda} m_{sc}'(z)}{1+ \sqrt{\lambda} m_{sc}(z)} | \notag \\
	&\leq |\xi_N(1, z) - \xi_N(1, z_i)| + |\xi_N(1, z_i) - \xi_N(0, z_i) +\frac{K \sqrt{\lambda} m_{sc}'(z_i)}{1+ \sqrt{\lambda} m_{sc}(z_i)} | + |\xi_N(0, z_i) - \xi_N(0, z)| \notag \\
	&= O_\prec(N^{-\frac{1}{2}}).
\end{split} \eeq
Now, integrating over $\Gamma$, we get
\begin{equation} \label{eq:lattice}
	\frac{-1}{2\pi \ii} \oint_{\Gamma} f(z) \pB{ \xi_N(1, z) -\xi_N(0, z)} \dd z = \frac{K}{2\pi \ii} \oint_{\Gamma} f(z)\frac{ \sqrt{\lambda} m_{sc}'(z)}{1+ \sqrt{\lambda} m_{sc}(z)}\dd z + O_\prec(N^{-\frac{1}{2}}).
\end{equation}
Note that, for $\ell=0,1,2,\ldots,$
\begin{equation}
	\tau_{\ell}(f) = \frac{1}{2\pi}\int_{-\pi}^{\pi}f(2\cos\theta)\cos(\ell\theta)\dd\theta = \frac{(-1)^\ell}{2\pi \ii}\oint_{|s|=1}f\pa{-s-\frac{1}{s}} s^{\ell-1}\dd s
\end{equation}
where we set $s=-\e^{\ii\theta}$ for the second inequality.

Converting the integral on $z$ to the integral on $m_{sc}$ by $z=-m_{sc}-\frac{1}{m_{sc}}$, we conclude that the the difference between the LSS of $M$ and that of $H$ is
\[
K \sum_{\ell=1}^{\infty} \sqrt{\lambda^{\ell}} \tau_{\ell}(f).
\]

We now prove the claim \eqref{eq:fundamental}. For the ease of notation, we omit the $z$-dependence in some occasions. Using the formula
\begin{equation} \label{eq:matrixderivative}
	\frac{\partial R_{jj}(\theta)}{\partial M_{ab}(\theta)} =
	\begin{cases}
		-R_{ja}(\theta) R_{bj}(\theta) -R_{jb}(\theta) R_{aj}(\theta) & \text{ if } a \neq b, \\
		-R_{ja}(\theta) R_{aj}(\theta) & \text{ if } a=b,
	\end{cases}
\end{equation}
and the fact that $M$ and $R(\theta)$ are symmetric, it is straightforward to check that
\begin{equation} \label{eq:theta_derivative}
	\frac{\partial}{\partial \theta} \Tr R(\theta,z) = -\sum_{m=1}^{K}\sqrt{\lambda}\left( {\bold{v}^{(m)}}^T R(\theta, z)^2 {\bold{v}^{(m)}} \right)= -\sum_{m=1}^{K}\sqrt{\lambda}\frac{\partial}{\partial z} \left( {\bold{v}^{(m)}}^T R(\theta, z) {\bold{v}^{(m)}} \right).
\end{equation}
where 
\begin{equation}
	R(\theta, z)^2 = (M(\theta)-zI)^{-2} = \frac{\partial}{\partial z} (M(\theta)-zI)^{-1} = \frac{\partial}{\partial z} R(\theta, z),
\end{equation}
which can be checked from the definition of the resolvent.

Set $S(z) := R(0,z) = (H-zI)^{-1}$ for convenience. We have from the definition of the resolvents that
\begin{equation} \label{eq:resolvent_diff1}
	R(\theta, z)^{-1} - S(z)^{-1} = \theta\sqrt{\lambda}VV^T S(z)
\end{equation}
and after multiplying $S(z)$ from the right and $R(\theta, z)$ from the left, we find that
\begin{align} \label{eq:resolvent_diff2}
	S(z) - R(\theta, z) &= \theta \sqrt{\lambda} R(\theta, z) VV^T S(z)= \sum_{i=1}^{K}\theta\sqrt{\lambda}R(\theta,z)\b v^{(i)}(\b v^{(i)})^T S(z).
\end{align}
Thus,
\begin{equation} \begin{split} \label{eq:resolvent_expansion}
		\langle \b v^{(m)}, S(z) \b v^{(m)} \rangle &= \langle \b v^{(m)}, R(\theta, z) \b v^{(m)} \rangle + \theta\sqrt{\lambda}\sum_{i=1}^{K} \langle \b v^{(m)}, R(\theta, z) \b v^{(i)} (\b v^{(i)})^T S(z) \b v^{(m)} \rangle \\
	&= \langle \b v^{(m)}, R(\theta, z) \b v^{(m)} \rangle + \theta\sqrt{\lambda}\sum_{i=1}^{K} \langle \b v^{(m)}, R(\theta, z) \b v^{(i)} \rangle \langle \b v^{(i)}, S(z) \b v^{(m)} \rangle 
\end{split} \end{equation}

For the resolvents of the centered SBM, we have the following lemma.
\begin{lemma}[Isotropic local law] \label{lem:iso local law}
	For an $N$-independent constant $\epsilon > 0$, let $\Gamma^{\epsilon}$ be the $\epsilon$-neighborhood of $\Gamma$, i.e.,
	\[
	\Gamma^{\epsilon} = \{ z \in \C : \min_{w \in \Gamma} |z-w| \leq \epsilon \}.
	\]
	Choose $\epsilon$ small so that the distance between $\Gamma^{\epsilon}$ and $[-2, 2]$ is larger than $2\epsilon$, i.e., 
	\begin{equation}
		\min_{w \in \Gamma^{\epsilon}, x \in [-2, 2]} |x-w| > 2\epsilon.
	\end{equation}
	Assume that the matrix $H$ satisfies Definition \ref{def:cgSBM} with $q=c\sqrt{N}$ for some constant $0<c<1$.
	Then, for any deterministic $\bsv, \bsw \in \C^N$ with $\| \bsv \| = \| \bsw \| = 1$ and sufficiently small $\delta>0$, the following estimate holds uniformly on $z \in \Gamma^{\epsilon}$:
	\begin{equation}\label{eq:iso local law}
		\left|\langle \bsv, (H-zI)^{-1} \bsw \rangle - m_{sc}(z) \langle \bsv, \bsw \rangle \right| = O_\prec (N^{-\frac{1}{2}}).
	\end{equation}
\end{lemma}
\begin{proof}
	This follows from Theorem 2.15 of \cite{BEKYY14}. See Lemma 7.7 of \cite{CL19} for  details.
\end{proof}
From the isotropic local law, Lemma \ref{lem:iso local law}, we find that
\begin{equation}
	\langle \b v^{(i)}, S(z) \b v^{(m)} \rangle = \delta_{im}m_{sc}(z) + O_\prec(N^{-\frac{1}{2}}).
\end{equation}
Moreover, from the rigidity of the eigenvalues, for some constant $C$, we have
\begin{equation}
	 \langle \b v^{(m)}, R(\theta, z) \b v^{(m)} \rangle \leq \norm{R(\theta,z)}\leq C.
\end{equation}
We then have from \eqref{eq:resolvent_expansion} that
\begin{equation}
m_{sc} = \langle \b v^{(m)}, R(\theta, z) \b v^{(m)} \rangle + \theta\sqrt{\lambda} \langle \b v^{(m)}, R(\theta, z) \b v^{(m)} \rangle m_{sc} + O_\prec(N^{-1/2}).
\end{equation}
Therefore, we conclude that 
\begin{equation}
	 \langle \b v^{(m)}, R(\theta, z) \b v^{(m)} \rangle = \frac{m_{sc}(z)}{1 + \theta\sqrt{\lambda} m_{sc}(z)} + O_\prec(N^{-\frac{1}{2}}).
\end{equation}
Note that $|m_{sc}| \leq 1$ and $\lambda < 1$, hence $|1+\sqrt{\lambda} m_{sc}| > c > 0$ for some ($N$-independent) constant $c$.

Consider the boundary of the $\epsilon$-neighborhood of $z$, $\partial B_{\epsilon}(z) = \{ w \in \C : |w-z| = \epsilon \}$. If we choose $\epsilon$ as in the assumption of Lemma \ref{lem:iso local law}, $\partial B_{\epsilon}(z)$ does not intersect $[-2, 2]$. Applying Cauchy's integral formula, we get
\beq \begin{split} \label{eq:z_deriavtiave2}
\frac{\partial}{\partial z} \langle \b v^{(m)}, R(\theta, z) \b v^{(m)} \rangle & =\frac{1}{2\pi \ii} \oint_{\partial B_{\epsilon}(z)} \frac{\langle \b v^{(m)}, R(\theta, z) \b v^{(m)} \rangle}{(w-z)^2} \dd w \\
	& = \frac{m_{sc}'(z)}{(1 + \theta\sqrt{\lambda} m_{sc}(z))^2} + O_\prec(N^{-\frac{1}{2}}).
\end{split} \eeq
Plugging the estimate into the right-hand side of \eqref{eq:theta_derivative}, we get the claim \eqref{eq:fundamental}.

\section{Proof of Lemma \ref{lem:E e Tr G}}\label{sec:lem E e Tr G}
With Definitions in Appendix \ref{sec:prelim}, recall that our goal is to show that 
	\beq \begin{split} \label{eq:clt goal}
		&\E \qb{ e(\t)\cdot(\tr G(z)-\E\tr G(z)) \di z } \\ 
		&= -\phi(\t)\frac{\t\sqrt{N}}{\pi q}\xi^{(4)} m_{sc}(z)m_{sc}'(z)\oint_{\Gamma}f(z')m_{sc}(z')m_{sc}'(z') \di z'  + o(\frac{\sqrt{N}}{q})
	\end{split} \eeq
with overwhelming probability, where $e(\t)$ is defined by
\beq \label{eq:e(lambda)}
	e(\t) := \exp\hb{-\frac{\t q}{2\pi\sqrt{N}} \int_{\Gamma_2} f(z)(\tr G(z)-\E\tr G(z)) \di z    }.
\eeq
Then the characteristic function $\phi$ satisfies 
	$$
	\phi'(\t)=-\t \phi(\t)V(f) + o_\prec(1),
	$$
	where $V(f)$ is given by 
	\begin{equation}
		V(f) = 2\xi^{(4)}\tau_2(f)^2,
	\end{equation}
and which equals $2\tau_2(f)^2$ for $H$ given by Definition \ref{def:Csbm H}.

Throughout this section, we say that a random variable $Z$ is \emph{negligible} if $Z=o_\prec(\sqrt{N}q^{-1})$.
\begin{proof}
	To prove the proposition, we apply Lemma \ref{lemma:Stein} for $l=4$ and get
	\beq \begin{split}\label{eq:cumulant expansion 1}
		&z\E \qb{ e(\t)(\tr G(z) -\E \tr G(z))  } \\
		&= \E \qb{ e(\t) \sum_{i\neq j} (H_{ij}G_{ji} -\E [H_{ij}G_{ji}] ) } + \E \qb{ e(\t) \sum_{i} (H_{ii}G_{ii} -\E [H_{ii}G_{ii}] ) } \\
		& = I_1 + I_2 + I_3 + I_4 + R_5 + I_d,
	\end{split} \eeq
	where 
	\begin{align}
		I_d &= \E \qb{ e(\t) \sum_{i} (H_{ii}G_{ii} -\E [H_{ii}G_{ii}] ) }, \\
		I_1 &= \sum_{i\neq j} \kappa_{ij}^{(2)} \pB{\E\qb{\partial_{ij} e(\t) \cdot G_{ij}} + \E \qb{ \p{1-\E} \pb{\partial_{ij}G_{ij}  } \cdot e(\t)   } }, \\
		I_2 &= \sum_{i\neq j} \frac{\kappa_{ij}^{(3)}}{2!} \pB{\E\qb{\partial_{ij}^2 e(\t) \cdot G_{ij}} +2\E\qb{\partial_{ij} e(\t) \cdot \partial_{ij}G_{ij}}+ \E \qb{ \p{1-\E} \pb{\partial_{ij}^2G_{ij}  } \cdot e(\t)   } }, 
	\end{align}
	\begin{align}
		I_3 &= \sum_{i\neq j} \frac{\kappa_{ij}^{(4)}}{3!} \pB{\E\qb{\partial_{ij}^3 e(\t) \cdot G_{ij}} +3\E\qb{\partial_{ij}^2 e(\t) \cdot \partial_{ij}G_{ij}}+ 3\E\qb{\partial_{ij} e(\t) \cdot \partial_{ij}^2G_{ij}}+  \notag \\
			& \qquad \qquad \qquad \qquad \qquad \qquad \qquad \qquad \qquad \qquad \qquad  \E \qb{ \p{1-\E} \pb{\partial_{ij}^3G_{ij}  } \cdot e(\t)   } } ,\\
		I_4 &= \sum_{i\neq j} \frac{\kappa_{ij}^{(5)}}{4!} \pB{\E\qb{\partial_{ij}^4 e(\t) \cdot G_{ij}} +4\E\qb{\partial_{ij}^3 e(\t) \cdot \partial_{ij}G_{ij}}+ 6\E\qb{\partial_{ij}^2 e(\t) \cdot \partial_{ij}^2 G_{ij}}\notag \\
			& \qquad \qquad \qquad \qquad\qquad\qquad+	
			4\E\qb{\partial_{ij} e(\t) \cdot \partial_{ij}^3G_{ij}}+  \E \qb{ \p{1-\E} \pb{\partial_{ij}^4G_{ij}  } \cdot e(\t)   } } ,
	\end{align}
	and $R_5$ is the error term given by Lemma \ref{lemma:Stein}.

	\begin{lemma}\label{lem:main lemma}
	For sufficiently large $N$,	We have 
		\begin{align}
			I_d &= O(1) \label{eq:diag 0}\\
			I_1 &= \frac{m_{sc}'}{m_{sc}}(m_{sc}^2-1)\E \qb{e(\t) \cdot \p{1-\E} \tr G } \notag \\
			&\qquad  + m_{sc}^2 m_{sc}' \frac{\t\sqrt{N} }{\pi q }\xi^{(4)}\E\qb{e(\t) }\int_{\Gamma_2} f(z) m_{sc}(z) m_{sc}'(z)\di z   + o(\frac{\sqrt{N}}{q}),\\
			I_2 & = o(\frac{\sqrt{N}}{q}),\\
			I_3 & =  m_{sc}^2 \frac{\t\sqrt{N} }{\pi q }\xi^{(4)} \E\qb{e(\t)}\int_{\Gamma_2} f(z) m_{sc}(z) m_{sc}'(z)\di z  + o(\frac{\sqrt{N}}{q}),\\
			I_4 & = o(\frac{\sqrt{N}}{q}),\\
			R_5 & = o(\frac{\sqrt{N}}{q}).
		\end{align}
	with overwhelming probability.
	\end{lemma}
	
	Combining Lemma \ref{lem:main lemma} and \eqref{eq:cumulant expansion 1} together, with overwhelming probability, we obtain
	\begin{align}
		&z\E \qb{ e(\t)(\tr G(z) -\E \tr G(z))  } \notag \\
		&	= 	\frac{m_{sc}'}{m_{sc}}(m_{sc}^2-1)\E \qb{e(\t) \cdot \p{1-\E} \tr G } \notag\\
		&\qquad + m_{sc}^2(m_{sc}'+1) \frac{\t\sqrt{N} }{\pi q }\xi^{(4)} \E\qb{e(\t)}\int_{\Gamma_2} f(z) m_{sc}(z) m_{sc}'(z)\di z  + o(\frac{\sqrt{N}}{q})  \notag \\
		&= 	-m_{sc}\E \qb{e(\t) \cdot \p{1-\E} \tr G } +  m_{sc}' \frac{\t\sqrt{N} }{\pi q }\xi^{(4)} \E\qb{e(\t)}\int_{\Gamma_2} f(z) m_{sc}(z) m_{sc}'(z)\di z  +o(\frac{\sqrt{N}}{q}).\notag
	\end{align}
	
	Rearrange the equation above and divide both side by $(z+m_{sc})$, we obtain \eqref{eq:clt goal}.
\end{proof}

\subsection{Proof of Lemma \ref{lem:main lemma}}
For the estimates in this section, we use the following power counting argument frequently.
\begin{lemma}{(Lemma 6.5 of \cite{LS18})}\label{lemma:powercounting} For any ${i}$ and ${k}$,
	\begin{equation}
		\frac{1}{N}\sum_{{j}=1}^{N}\absa{G_{ij}(z)G_{jk}(z)} \prec \frac{\im m(z)}{N\eta}, \quad \frac{1}{N}\sum_{{j}=1}^{N}\absa{G_{ij}(z)} \prec \pa{\frac{\im m(z)}{N\eta}}^{1/2}, \quad (z\in\C^+).
	\end{equation}
	Moreover, For $z \in \Gamma$, we have 
	\begin{equation}
		\frac{1}{N}\sum_{{j}=1}^{N}\absa{G_{ij}(z)G_{jk}(z)} \prec \frac{1}{N}, \qquad \frac{1}{N}\sum_{{j}=1}^{N}\absa{G_{ij}(z)} \prec {\frac{1}{N^{1/2}}}, \qquad (z\in\Gamma).
	\end{equation}
\end{lemma}
\subsubsection{estimate on $e(\t)$, $R_5$ and $I_d$}
\begin{lemma}\label{lemma:partial e} Let $e(\t)$ defined in \eqref{eq:e(lambda)}. Then,
	\begin{align}
		\partial_{ij}e(\t) &= O_\prec \pB{\frac{\log N}{\sqrt{N}}},\\
		\partial_{ij}^2 e(\t) &=  -\frac{\t q}{\pi \sqrt{N}}e(\t) \int_{\Gamma_2} f(z) (G_{ii} (G')_{jj} + G_{jj} (G')_{ii} )\di z + O_\prec \pB{\frac{\log N}{\sqrt{N}}}, \\
		\partial_{ij}^k e(\t) &=  O_\prec(\frac{q}{\sqrt{N}}), \qquad (k \geq 3).
	\end{align}
\end{lemma}
\begin{proof}
	Note that $G_{ij}(z)$ ism  analytic in $z \in \C \setminus \R$ and $G^2_{ij} = \frac{d}{dz}G_{ij}(z)$. Using the Cauchy integral formula and the local law, we get
	\begin{equation}\label{eq:estimate G^2}
		\absa{\frac{d}{dz}G_{ij}(z) - \delta_{ij}m'_{sc}(z) } \prec \frac{1}{q \im z}.
	\end{equation}
	Then with \eqref{eq:estimate G^2}, we obtain
	\begin{align}\label{eq:estimate d e}
		\partial_{ij}e(\t)&=-\frac{\t q}{2\pi \sqrt{N}}e(\t)\int_{\Gamma_2} f(z)\partial_{ij}\sum_{k}G_{kk}(z) \di z   \notag \\
		&=\frac{\t q}{\pi \sqrt{N}}e(\t)\int_{\Gamma_2} f(z)(G^2)_{ij} \di z  \notag \\
		&= O_\prec \pB{\frac{\log N}{\sqrt{N}}}.
	\end{align}
	Taking derivative of \eqref{eq:estimate d e} again, we get
	\beq\begin{split}\label{eq:estimate d^2 e}
		\partial_{ij}^2 e(\t)&= \pa{\frac{\t q}{\pi \sqrt{N}}\int_{\Gamma_2} f(z)(G^2)_{ij} \di z}^2 e(\t) + \frac{\t q}{\pi \sqrt{N}}e(\t)\int_{\Gamma_2} f(z) \sum_{k}\partial_{ij}(G_{ik}G_{kj}) \di z \\
		&= -\frac{\t q}{\pi \sqrt{N}}e(\t) \int_{\Gamma_2} f(z) (G_{ii} (G')_{jj} + G_{jj} (G')_{ii} )\di z + O_\prec \pB{\frac{\log N}{\sqrt{N}}}.
	\end{split} \eeq
	In general, repeatedly taking derivative with local law, we can show that for $k \geq 3$,
	\begin{align}\label{eq:estimate d^k e}
		\partial_{ij}^k e(\t) =  O_\prec(\frac{q}{\sqrt{N}}), \qquad (k \geq 3).
	\end{align}
\end{proof}
With above bounds in Lemma \ref{lemma:partial e}, the error term $R_5$ can be estimated as
\beq \begin{split}
	\abs{R_5} &\leq C N^2 \E |H_{ij}|^6 \pB{  \norm{\partial_{ij}^5 (e(\t)G_{ij})}_\infty +\norm{e(\t) \partial_{ij}^5 G_{ij}}_\infty   } \\
	&=O_\prec\pB{ N^2 \frac{1}{Nq^4}} =o_\prec\pB{ \frac{\sqrt{N}}{q} },
\end{split} \eeq
which is negligible.
Similar as Lemma \ref{lemma:partial e}, for all $k$, we can show that
\beq\label{eq:partial e ii}
\partial_{ii}^{k} e(\t) = O_\prec(1).
\eeq
Using Lemma \ref{lemma:Stein} with $\ell=1$, we can expand $I_d$ as follows
\beq \begin{split}
	&\E \qb{ e(\t) \sum_{i} (H_{ii}G_{ii} -\E [H_{ii}G_{ii}] ) } \\
	&= \sum_{i} \kappa_{ii}^{(2)} \pB{\E\qb{\partial_{ii} e(\t) \cdot G_{ii}} + \E \qb{ \p{1-\E} \pb{\partial_{ii}G_{ii}  } \cdot e(\t)   } } + O_\prec\pb{\frac{1}{q}}.\label{eq:cumulant diagonal expansion}
\end{split} \eeq
With \eqref{eq:partial e ii} and local law, we can easily conclude that \eqref{eq:cumulant diagonal expansion} is $O(1)$ with overwhelming probability. 
\begin{remark}
	In the rest of the section, for simplicity of notation, we sometimes omit the inequality sign for indices below the summation.
\end{remark}

\subsubsection{estimate on $I_2$}\label{sec:I_2}
In this section, we will show that $I_2$ is negligible, which is defined by
$$I_2 = \sum_{i\neq j} \frac{\kappa_{ij}^{(3)}}{2!} \pB{\E\qb{\partial_{ij}^2 e(\t) \cdot G_{ij}} +2\E\qb{\partial_{ij} e(\t) \cdot \partial_{ij}G_{ij}}+ \E \qb{ \p{1-\E} \pb{\partial_{ij}^2G_{ij}  } \cdot e(\t)   } }.$$
Using Lemma \ref{lemma:powercounting} and local law, one can show that the terms that cannot be clearly ignored are 
\begin{equation}
	I_{2,1}:=	\sum_{i\neq j}\kappa_{ij}^{(3)}\E[G_{ii}G_{jj}\partial_{ij}e(\t)]
\end{equation}
and
\begin{equation}
	I_{2,2}:=\sum_{i\neq j}\kappa_{ij}^{(3)}\E \qb{e(\t) \cdot\p{1-\E} \pb{G_{ii}G_{jj}G_{ij}  }     }.
\end{equation}

We first estimate $I_{2,1}$,
\begin{align}\label{eq:I_{2.1}}
	I_{2,1} &=\sum_{i\neq j}\kappa_{ij}^{(3)}\E[G_{ii}G_{jj}\partial_{ij}e(\t)]  \notag \\
	&=m_{sc}^2 \frac{\t}{\pi}\frac{q}{\sqrt{N}}\sum_{i,j,k} \kappa_{ij}^{(3)}\E\qb{ e(\t)\int_\Gamma f(z)G_{ik}(z)G_{kj}(z) \di z   } + o_\prec\pB{\frac{\sqrt{N}}{q}} \notag \\
	&= m_{sc}^2 \frac{\t}{\pi}\frac{q}{\sqrt{N}} \int_\Gamma f(z) \sum_{i,j,k} \kappa_{ij}^{(3)} \E \qb{ e(\t) G_{ik}G_{kj}}\di z +o_\prec\pB{\frac{\sqrt{N}}{q}}.
\end{align}

To estimate  $\sum_{i,j,k} \kappa_{ij}^{(3)}\E \qb{ e(\t) G_{ik}G_{kj}}$ with the error of $o(N^{-1}q^{-2})$, we again multiply $z$ and expand it using Lemma \ref{lemma:Stein} until $l=10$.
\begin{align}\label{eq:I_2,1}
	z \sum_{i,j,k} \kappa_{ij}^{(3)}\E \qb{ e(\t) G_{ik}&G_{kj}} =  \sum_{i,j,k} \kappa_{ij}^{(3)}\E \qb{ e(\t) \sum_{l} H_{il} G_{lk}G_{kj}} \notag \\
	=& \sum_{i,j,k,l} \kappa_{ij}^{(3)} \kappa_{il}^{(2)} \E \qb{ \partial_{il} ( e(\t)   G_{lk}G_{kj})} +\sum_{i,j,k,l} \kappa_{ij}^{(3)} \frac{\kappa_{il}^{(3)}}{2} \E \qb{ \partial_{il}^2 (e(\t)   G_{lk}G_{kj})} \notag \\
	&+ \sum_{i,j,k,l} \kappa_{ij}^{(3)} \frac{\kappa_{il}^{(4)}}{6} \E \qb{ \partial_{il}^3 (e(\t)   G_{lk}G_{kj})} + \cdots + R_{10}.
\end{align}
For each $s\geq 3$, $\kappa_{il}^{(s+1)} \leq \frac{C}{Nq^2}$ and $\partial_{il}^s (e(\t)   G_{lk}G_{kj})$ contains at least two off-diagonal entries. Hence by Lemma \ref{lemma:powercounting}, that terms appear in \eqref{eq:I_{2.1}} are negligible. Moreover, since $q \gg N^{1/6}$, $R_{10}\leq CN^4 \frac{1}{Nq} \frac{1}{Nq^9} = C \frac{N^2}{q^{10}} = o_\prec(Nq^{-2})$ and thus negligible in \eqref{eq:I_{2.1}}.

Using Lemma \ref{lemma:powercounting} and  Lemma \ref{lemma:partial e}  with local law, we obtain the main order term of the first term of \eqref{eq:I_2,1},
\begin{align}
	\sum_{i,j,k,l} \kappa_{ij}^{(3)} \kappa_{il}^{(2)} \E \qb{ \partial_{il} ( e(\t)   G_{lk}G_{kj})} &= - \sum_{i,j,k,l} \kappa_{ij}^{(3)}\kappa_{il}^{(2)} \E\qb{e(\t) G_{ik}G_{kj}G_{ll}  } + O_\prec(\frac{\sqrt{N}}{q}) \notag \\
	&=-m_{sc} \sum_{i,j,k} \kappa_{ij}^{(3)} \E\qb{e(\t) G_{ik}G_{kj}} + O_\prec(\frac{N}{q^3}+\frac{\sqrt{N}}{q}
	).
\end{align}
Similarly, we get the estimate for the second term of \eqref{eq:I_2,1} as
\begin{align}
	\sum_{i,j,k,l} \kappa_{ij}^{(3)} \frac{\kappa_{il}^{(3)}}{2} \E \qb{ \partial_{il}^2 (e(\t)   G_{lk}G_{kj})}&= \sum_{i,j,k,l} \kappa_{ij}^{(3)} \kappa_{il}^{(3)}\E \qb{ e(\t)G_{ii}G_{ll}G_{lk}G_{kj} } + O_\prec(\frac{N}{q^3}) \notag \\  	
	&= m_{sc}^2\sum_{i,j,l} \kappa_{ij}^{(3)} \kappa_{il}^{(3)}\E \qb{ e(\t)\sum_{k}G_{lk}G_{kj} } + O_\prec(\frac{N}{q^3}) \notag \\ 
	&= m_{sc}^2\sum_{i,j,l} \kappa_{ij}^{(3)} \kappa_{il}^{(3)}\E \qb{ e(\t)(G^2)_{lj} } + O_\prec(\frac{N}{q^3}) \notag \\
	&= O_\prec(\frac{N}{q^3}).
\end{align}
To sum up, 
\begin{align}
	(z+m_{sc}) \sum_{i,j,k} \kappa_{ij}^{(3)}\E \qb{ e(\t) G_{ik}G_{kj}} = O_\prec(\frac{N}{q^3}+\frac{\sqrt{N}}{q}).
\end{align}
Since $|z+m_{sc}|>c$ for some constant $c$, we can divide both side by $(z+m_{sc})$ and conclude that 
\begin{align}
	I_{2,1} &= m_{sc}^2 \frac{\t}{\pi}\frac{q}{\sqrt{N}} \int_\Gamma f(z) \sum_{i,j,k} \kappa_{ij}^{(3)} \E \qb{ e(\t) G_{ik}G_{kj}}\di z +o\pB{\frac{\sqrt{N}}{q}} \notag \\
	&= O_\prec(\frac{\sqrt{N}}{q^2}+1 ) + o_\prec\pB{\frac{\sqrt{N}}{q}},
\end{align}
and thus negligible.

Now we move on to $I_{2,2}$. Multiplying $z$, we get
\begin{align}\label{eq:I_2,2 a}
	zI_{2,2}&= z\sum_{i, j}\kappa_{ij}^{(3)}\E \qb{e(\t) \cdot\p{1-\E} \pb{G_{ii}G_{jj}G_{ij}  }     } \notag \\
	&=  m_{sc}^2	\sum_{i, j}\kappa_{ij}^{(3)}\E \qb{e(\t) \cdot\p{1-\E} \pb{zG_{ij}  }    } + o_\prec(\frac{\sqrt{N}}{q}) . 
\end{align}
Similar as above, it can be easily checked that the only non negligible term of $zI_{2,2}$ is
\begin{equation}
	-m_{sc}^2	\sum_{i, j,k}\kappa_{ij}^{(3)}\kappa_{ik}^{(2)}\E \qb{e(\t) \cdot\p{1-\E} \pb{G_{kk}G_{ij}  }    },
\end{equation}
which can be estimated by 
\begin{equation}\label{eq:I_2,2 b}
	-m_{sc}^2	\sum_{i, j,k}\kappa_{ij}^{(3)}\kappa_{ik}^{(2)}\E \qb{e(\t) \cdot\p{1-\E} \pb{G_{kk}G_{ij}  }    } =  -m_{sc}^3	\sum_{i, j}\kappa_{ij}^{(3)}\E \qb{e(\t) \cdot\p{1-\E} \pb{G_{ij}  }    } +  o_\prec(\frac{\sqrt{N}}{q}).
\end{equation}
Combining \eqref{eq:I_2,2 a} and \eqref{eq:I_2,2 b}, we conclude 
\begin{align}
	(z+m_{sc})I_{2,2} = o(\frac{\sqrt{N}}{q}),
\end{align}
thus $I_{2,2}$ is negligible and so is $I_2$.

\subsubsection{estimate on $I_3$}\label{sec:I_3}
Recall that $I_3$ is defined by
\begin{align}
		I_3 &= \sum_{i\neq j} \frac{\kappa_{ij}^{(4)}}{3!} \pB{\E\qb{\partial_{ij}^3 e(\t) \cdot G_{ij}} +3\E\qb{\partial_{ij}^2 e(\t) \cdot \partial_{ij}G_{ij}}+ 3\E\qb{\partial_{ij} e(\t) \cdot \partial_{ij}^2G_{ij}}  \notag \\
		& \qquad \qquad \qquad \qquad \qquad \qquad \qquad \qquad  \qquad \qquad  +\E \qb{ \p{1-\E} \pb{\partial_{ij}^3G_{ij}  } \cdot e(\t)   } }.
\end{align}
Using Lemma \ref{lemma:powercounting} and local law, the only term that is not clearly negligible is 
$$ I_{3,1}:= \sum_{i\neq j}\kappa_{ij}^{(4)} \E \qb{  e(\t) \cdot  \q{ G_{ii}^2 G_{jj}^2 -\E  (G_{ii}^2 G_{jj}^2 )}   }, $$
which comes from $\E \qb{ \p{1-\E} \pb{\partial_{ij}^3G_{ij}  } \cdot e(\t)   }$ and
$$-\sum_{i,j}\frac{\kappa_{ij}^{(4)}}{2}  m_{sc}^2\E\qb{\partial_{ij}^2 e(\t) },  $$
which comes from $\sum_{i\neq j} {\kappa_{ij}^{(4)}}\E\qb{\partial_{ij}^2 e(\t) \cdot \partial_{ij}G_{ij}}$.
Here, we claim that $I_{3,1}$ is negligible.
After multiplying $z$, we have
\begin{align}\label{eq:I3,1}
	z I_{3,1} &= \sum_{i, j}\kappa_{ij}^{(4)} \E \qb{  e(\t) \cdot  \q{ zG_{ii} G_{ii} G_{jj}^2 -\E  (zG_{ii} G_{ii} G_{jj}^2 )}   } \notag \\
	&= \sum_{i,j}\kappa_{ij}^{(4)} \E \qB{ \sum_{k} H_{ik} e(\t) G_{ik}G_{ii}G_{jj}^2 -e(\t)\E \q{\sum_{k} H_{ik}  G_{ik}G_{ii}G_{jj}^2}   } \notag \\
	&\qquad -\sum_{i, j}\kappa_{ij}^{(4)} \E \qb{  e(\t) \cdot  \q{  G_{ii} G_{jj}^2 -\E  ( G_{ii} G_{jj}^2 )}   } \notag \\
	&= \sum_{i,j,k}\kappa_{ij}^{(4)}  \kappa_{ik}^{(2)} \pB{\E\qb{\partial_{ik} e(\t) \cdot G_{ik}G_{ii}G_{jj}^2 } + \E \qb{ \p{1-\E} \pb{\partial_{ik}G_{ik}G_{ii}G_{jj}^2  } \cdot e(\t)   } }, \notag \\
	&\qquad + \sum_{i,j,k}\kappa_{ij}^{(4)}  \frac{\kappa_{ik}^{(3)}}{2!} \pB{\E\qb{\partial_{ik}^2 e(\t) \cdot G_{ik}G_{ii}G_{jj}^2}  +2\E\qb{\partial_{ik} e(\t) \cdot \partial_{ik}G_{ik}G_{ii}G_{jj}^2}} \notag\\
	&\qquad+ \sum_{i,j,k}\kappa_{ij}^{(4)}  \frac{\kappa_{ik}^{(3)}}{2!} \pB{\E \qb{ \p{1-\E} \pb{\partial_{ik}^2 G_{ik}G_{ii}G_{jj}^2  } \cdot e(\t)   } }  \notag\\
	&\qquad-\sum_{i, j}\kappa_{ij}^{(4)} \E \qb{  e(\t) \cdot  \q{  G_{ii} G_{jj}^2 -\E  ( G_{ii} G_{jj}^2 )}   }  + o_\prec(\frac{\sqrt{N}}{q}) \notag \\
	&=-m_{sc}\sum_{i,j}\kappa_{ij}^{(4)}   \pB{ \E \qb{   \p{1-\E} \pb{  G_{ii}^2G_{jj}^2  } \cdot e(\t)   } } \notag \\ 
	&\qquad-\sum_{i, j}\kappa_{ij}^{(4)} \E \qb{  e(\t) \cdot  \q{  G_{ii} G_{jj}^2 -\E  ( G_{ii} G_{jj}^2 )}   } + o_\prec(\frac{\sqrt{N}}{q}).
\end{align}

However, in the same way, we can estimate the last term by

\begin{align}
	&\sum_{i, j}\kappa_{ij}^{(4)} \E \qb{  e(\t) \cdot  \q{  (1-\E)  ( G_{ii} G_{jj}^2 )}   } \notag \\
	&= -m_{sc}\sum_{i, j}\kappa_{ij}^{(4)} \E \qb{  e(\t) \cdot  \q{   (1-\E)  ( G_{ii} G_{jj}^2 )}   }-\sum_{i, j}\kappa_{ij}^{(4)} \E \qb{  e(\t) \cdot  \q{   (1-\E)  ( G_{jj}^2 )}   } + o_\prec(\frac{\sqrt{N}}{q}) \notag \\
	&= -m_{sc}\sum_{i, j}\kappa_{ij}^{(4)} \E \qb{  e(\t) \cdot  \q{   (1-\E)  ( G_{ii} G_{jj}^2 )}   } -\sum_{j}\sum_{i}\kappa_{ij}^{(4)} \E \qb{  e(\t) \cdot  \q{   (1-\E)  ( G_{jj}^2 )}   } + o_\prec(\frac{\sqrt{N}}{q}) \notag \\
	&= -m_{sc}\sum_{i, j}\kappa_{ij}^{(4)} \E \qb{  e(\t) \cdot  \q{   (1-\E)  ( G_{ii} G_{jj}^2 )}   } -2\frac{\xi^{(4)}}{q^2}m_{sc} \E \qb{  e(\t) \cdot  \q{   (1-\E)  ( \tr G )}   } + o_\prec(\frac{\sqrt{N}}{q}) \notag \\
	&=  -m_{sc}\sum_{i, j}\kappa_{ij}^{(4)} \E \qb{  e(\t) \cdot  \q{   (1-\E)  ( G_{ii} G_{jj}^2 )}   } + O_\prec(\frac{N}{q^4}).
\end{align}
Since $|z+m_{sc}| > c$, the the last line of \eqref{eq:I3,1} is negligible, and therefore $I_{3,1}$ is also negligible.

To sum up,
\begin{equation}
	I_{3} = -\sum_{i,j}\frac{\kappa_{ij}^{(4)}}{2}  m_{sc}^2\E\qb{\partial_{ij}^2 e(\t) } + o_\prec(\frac{\sqrt{N}}{q}).
\end{equation}

\subsubsection{estimate on $I_4$ and $I_1$}
By simple moment counting, it can be easily shown that all terms in $I_4$ are negligible. We move on to $I_1$ defined by 
\begin{equation}\label{eq:I_1}
	I_1 = \sum_{i\neq j} \kappa_{ij}^{(2)} \pB{\E\qb{\partial_{ij} e(\t) \cdot G_{ij}} + \E \qb{ \p{1-\E} \pb{\partial_{ij}G_{ij}  } \cdot e(\t)   } }.
\end{equation}
Using Lemma \ref{lemma:powercounting} and local law, we can ignore the first term of the right hand side of \eqref{eq:I_1} and conclude that the only non negligible term of $I_1$ is 

\begin{equation}
	I_{1,1}=-\sum_{i\neq j} \kappa_{ij}^{(2)}\E \qb{ \p{1-\E} \pb{G_{ii}G_{jj}  } \cdot e(\t)   }.
\end{equation}

Multiplying $z$ to $I_{1,1}$ and using Lemma \ref{lemma:Stein} with sufficiently large $l$,  we have

\begin{align}
	zI_{1,1} & = \sum_{i,j,k} \kappa_{ij}^{(2)} \kappa_{ik}^{(2)} \E \qb{ e(\t)(1-\E) ( G_{ii} G_{jj} G_{kk})   } + \sum_{i,j,k} \kappa_{ij}^{(2)} \kappa_{ik}^{(4)} \E \qb{ e(\t)(1-\E) ( G_{ii}^2 G_{jj} G_{kk}^2)   } \notag \\
	& \qquad + \sum_{i,j,k} \kappa_{ij}^{(2)} \frac{\kappa_{ik}^{(4)}}{2} \E\qb{\partial_{ij}^2 e(\t) \cdot G_{ii}G_{jj}G_{kk}  } + \E \qb{e(\t)(\tr G -\E \tr G) } + o_\prec(\frac{\sqrt{N}}{q}) \notag \\
	& = \wt I_1 + \wt I_2 + \wt I_3 + \E \qb{e(\t)(\tr G -\E \tr G) } + o_\prec(\frac{\sqrt{N}}{q}),  
\end{align}
after adding up all the easily negligible terms.

We first estimate $\wt I_1$. With local law, $\sum_{j}\kappa_{ij}^{(2)}G_{jj}$ can be estimated by $m_{sc}$, therefore
\begin{align}
	\wt I_{1} &= \sum_{i}  \E \qb{ e(\t)(1-\E) ( G_{ii} \sum_{j}\kappa_{ij}^{(2)}G_{jj} \sum_{k}\kappa_{ik}^{(2)} G_{kk})   } \notag \\
	&=- m_{sc}^2 \E \qb{e(\t)(\tr G -\E \tr G) } + 2m_{sc} \sum_{i,j}\kappa_{ij}^{(2)}\E \qb{ e(\t)(1-\E) ( G_{ii} G_{jj} )   } + o_\prec(\frac{\sqrt{N}}{q}).
\end{align}

For $\wt I_2$, after summing up all the terms for $j$, we get 
\begin{align}
	\wt I_2 & = m_{sc}\sum_{i,k} \kappa_{ik}^{(4)} \E \qb{ e(\t)(1-\E) ( G_{ii}^2  G_{kk}^2)   } + o_\prec(\frac{\sqrt{N}}{q}).
\end{align}
Recall that $I_{3,1}= \sum_{i\neq j}\kappa_{ij}^{(4)} \E \qb{  e(\t) \cdot  \q{ G_{ii}^2 G_{jj}^2 -\E  (G_{ii}^2 G_{jj}^2 )}   }= O_\prec(\frac{{N}}{q^4})$. Hence, $\wt I_2$ is also  $O_\prec(\frac{{N}}{q^4})$, thus negligible.

Finally, $\wt I_3$ can be computed by 
\begin{align}
	\wt I_3 = m_{sc}^3 \sum_{i,j} \frac{\kappa_{ij}^{(4)}}{2} \E\qb{\partial_{ij}^2 e(\t)} + o(\frac{\sqrt{N}}{q}),
\end{align}
since $G_{ii}$ can be estimated by $m_{sc}$ with local law.

To sum up, $z I_{1,1}$ is computed by 
\begin{align}
	-\sum_{i\neq j} \kappa_{ij}^{(2)}&\E \qb{ \p{1-\E} \pb{G_{ii}G_{jj}  } \cdot e(\t)   }\notag \\
	=&-m_{sc}^2 \E \qb{e(\t)(\tr G -\E \tr G) } + 2m_{sc} \sum_{i,j}\kappa_{ij}^{(2)}\E \qb{ e(\t)(1-\E) ( G_{ii} G_{jj} )   } \notag \\
	& \quad+m_{sc}^3 \sum_{i,j} \frac{\kappa_{ij}^{(4)}}{2} \E\qb{\partial_{ij}^2 e(\t)} +\E \qb{e(\t)(\tr G -\E \tr G) } + o_\prec(\frac{\sqrt{N}}{q}).
\end{align}

After rearranging the terms containing $G_{ii}G_{jj}$ to the left hand side, and using the identity that $z+2m_{sc} = -\frac{m_{sc}'}{m_{sc}}$, we finally get 
\begin{align}
	I_{1,1} = \frac{m_{sc}'}{m_{sc}}(m_{sc}^2-1)\E \qb{e(\t) \cdot \p{1-\E} \tr G } - m_{sc}^2 m_{sc}' \sum_{i\neq j} \frac{\kappa_{ij}^{(4)}}{2} \E\qb{\partial_{ij}^2 e(\t)}  + o_\prec(\frac{\sqrt{N}}{q}).
\end{align}

\section{Proof of Proposition \ref{prop:mean of LSS}}\label{sec:proof of mean}

Throughout this section, we say that a random variable $Z$ is \emph{negligible} if $|Z|=o(N^{-1/2}q^{-1})$ with overwhelming probability. For the convenience of notation, we define the parameter $\Phi:=N^{-1/2-\delta}q^{-1}$ for sufficiently small $\delta$.
To prove Proposition \ref{prop:mean of LSS}, we need the following Lemma.
\begin{lemma}\label{lem:main lemma of mean LSS}
	Suppose that $H$ satisfies conditions in Definition \ref{def:Dsbm M} with $N^{-2/3}\ll p_a \ll 1$. Then,
	\begin{align*}
		\E \qb{(1+zm+m^2+\frac{\xi^{(4)}}{q^{2}}m^4)(z+m+\zeta m) } = O_\prec(\Phi)
	\end{align*}
	where $\xi^{(4)}$ and $\zeta$ are defined in \eqref{eq:xi,zeta}.
\end{lemma}
When we prove Lemma \ref{lem:main lemma of mean LSS}, we have 
	\begin{align}
		&\E \qb{(1+zm+m^2)(z+m+\zeta m) } = - \E \qb{\frac{\xi^{(4)}}{q^{2}}m^4(z+m+\zeta m) } + O_\prec(\Phi), \\
		&\E \qb{(1+zm_{sc}+m_{sc}^2)(z+m+\zeta m) } = 0.
	\end{align}
	
	Using local law to estimate $|m-m_{sc}| \prec q^{-2}$ and subtracting left hand side of these two equations, we obtain 
	\begin{align}
		&\E \qb{ (z(m-m_{sc}) + m^2 -m_{sc}^2)(z+m(1+\zeta))     } \notag\\
		&= \E \qb{((z+2m_{sc})(m-m_{sc})+(m-m_{sc})^2) (z+m(1+\zeta))} \notag \\
		&= \E \qb{((z+2m_{sc})(m-m_{sc})+(m-m_{sc})^2) (z+m_{sc}(1+\zeta))} \notag \\
		& \qquad+  \E \qb{((z+2m_{sc})(m-m_{sc})+(m-m_{sc})^2) ((m-m_{sc})(1+\zeta))} \notag \\
		& = \E \qb{((z+2m_{sc})(m-m_{sc}) (z+m_{sc}(1+\zeta))} + O_\prec(\Phi).
	\end{align}
	Similarly subtracting right hand side of these two equations, we get
	\begin{align}
		- \E \qb{\frac{\xi^{(4)}}{q^{2}}m^4(z+m+\zeta m)} = - {\frac{\xi^{(4)}}{q^{2}}m_{sc}^4(z+(1+\zeta)m_{sc})}+ O_\prec(\Phi).
	\end{align} 
	Thus we conclude that 
	\begin{align}
		\E \qb{((z+2m_{sc})(m-m_{sc}) (z+m_{sc}(1+\zeta))} =  -{\frac{\xi^{(4)}}{q^{2}}m_{sc}^4(z+(1+\zeta)m_{sc})}+ O_\prec(\Phi),
	\end{align}
	and dividing the deterministic part, with overwhelming probability, we have
	\begin{align}\label{eq:expect mean 1}
		\E \qb{q\sqrt{N} (m-m_{sc})} &= -\frac{\sqrt{N}}{q}\frac{1}{z+2m_{sc}}\xi^{(4)} m_{sc}^4 + o(1) \notag \\
		& = \frac{\sqrt{N}}{q}\xi^{(4)} m_{sc}^3 m_{sc}' +  o(1).
	\end{align}
	By \eqref{eq:expect mean 1}, we finally obtain
	\begin{align}
		\E\qB{ \frac{q}{\sqrt{N}}\pb{ L_H(f) - N\langle\mu_{sc},f\rangle }} &= -\frac{1}{2\pi\ii} \oint_\Gamma f(z) \E \qb{q\sqrt{N}(m(z)-m_{sc}(z))} \di z \notag \\
		&= -\frac{1}{2\pi\ii} \frac{\sqrt{N}}{q} \oint_\Gamma f(z) \xi^{(4)} m_{sc}^3 m_{sc}'  \di z + o(1) \notag\\
		& = \frac{\sqrt{N}}{q}\xi^{(4)}\tau_4(f) + o(1).
	\end{align}
Indeed, since $q^2=Np_a$ and $\xi^{(4)}=1 +o(\frac{q}{\sqrt{N}})$ for $H$ given in \eqref{def:Csbm H}, Proposition \ref{prop:mean of LSS} follows.
\subsection{Proof of Lemma \ref{lem:main lemma of mean LSS}}
To prove Lemma \ref{lem:main lemma of mean LSS}, we return to Lemma \ref{lemma:Stein}  which reads
\begin{align}
	\E \qB{&(z+m+ 2\zeta m)(1+zm)} \notag \\
	=& \frac{1}{N}\sum_{i\neq k}\sum_{r=1}^{l}\frac{\kappa_{ik}^{(r+1)}}{r!}\E \qB{ (\partial_{ik})^r \pB{G_{ik}(z+m+ 2\zeta m) }}  + \E \qB{ \Omega_l \pB{(z+m+ 2\zeta m)(1+zm)}  },
\end{align}
where $\partial_{ij} = \partial/(\partial H_{ik})$ as before. Remark that similar as proof of \eqref{eq:diag 0}, we can ignore when $i=k$. We leave details to the reader. 

Abbreviate 
\begin{align}
	I\equiv I(z,m) &= (z+m+ 2\zeta m)(1+zm)= Q (1+zm).
\end{align}
where $Q:=z+m+ 2\zeta m$.
Then we can rewrite the cumulant expansion as 
\begin{align}\label{eq:cumulantexp1}
	&\E I = \sum_{r=1}^{l} \sum_{s=0}^{r} w_{I_{r,s}} \E I_{r,s} + \E \Omega_l(I),
\end{align}
where we set
\begin{align}
	I_{r,s}&=\frac{1}{N} \sum_{i\neq k} \kappa_{ik}^{(r+1)}{ \pB{\partial_{ik}^{r-s} G_{ik}} \pB{\partial_{ik}^{s}{Q}}},\\
	w_{I_{r,s}}& = \frac{1}{(r-s)!s!}.
\end{align}

We can prove Lemma \ref{lem:main lemma of mean LSS} directly from the following result.
\begin{lemma}\label{lem:truncated}
	Choose $\ell \geq 10$. Then we have,
	\begin{equation}\label{eq:estimate I_1}
		\begin{split}
			&w_{I_{1,0}}\E \qb{I_{1,0}}=-\E\qB{ (z+m+2\zeta m)(1-\zeta)m^2 +2\zeta^2m^3 - \zeta m - \zeta \frac{\xi^{(4)}}{q^{2}} m^5} + O(\Phi),\\
			&w_{I_{2,0}}\E \qb{I_{2,0}}= O(\Phi),\\
			&w_{I_{3,0}}\E \qb{I_{3,0}}= -\E \qB{q^{-2}\xi^{(4)}Qm^4 }+O(\Phi),\\
			&w_{I_{r,0}}\E \qb{I_{r,0}}=O(\Phi), \qquad (4\leq r\leq \ell), \\
			&w_{I_{r,s}}\absa{\E \qb{I_{r,s}}} = O(\Phi), \qquad (1\leq s\leq r\leq \ell),
		\end{split}
	\end{equation}
	uniformly in $z \in \Gamma_2$, for $N$ sufficiently large. Moreover, we have, for the remainder term,,
	\begin{equation}\label{eq:estimate I_2}
		\E \Omega_l(I) =O(\Phi),
	\end{equation}
	uniformly in $z \in \Gamma_2$, for $N$ sufficiently large.
\end{lemma}
\emph{Proof of Lemma \ref{lem:main lemma of mean LSS}.}	Recall that $Q:=(z+m+2\zeta m)$.  From Lemma \ref{lem:truncated},
	\begin{align}
		&\E \qb{(zm+1)(z+m+2\zeta m)} \notag \\
		&= -\E\qB{Q(1-\zeta)m^2 +2\zeta^2m^3 - \zeta m - \zeta \frac{\xi^{(4)}}{q^{2}} m^5}  -\E \qB{\frac{\xi^{(4)}}{q^{2}}Qm^4 } + O(\Phi) \notag\\
		&= \E \qB{-Q\pb{m^2 + \frac{\xi^{(4)}}{q^{2}}m^4}} + \E \qB{\zeta m \pb{1+zm+m^2+\frac{\xi^{(4)}}{q^{2}}m^4}} + O(\Phi).\notag
	\end{align}
	Collecting terms which contain $Q$, it can be shown that
	\begin{align*}
		\E \qb{(1+zm+m^2+\frac{\xi^{(4)}}{q^{2}}m^4)(z+m+2\zeta m) } = \E \qb{ \zeta m (1+zm+m^2 + \frac{\xi^{(4)}}{q^{2}}m^4)} + O(\Phi).
	\end{align*}
	Hence we conclude that 
	\begin{align*}
		\E \qb{(1+zm+m^2+\frac{\xi^{(4)}}{q^{2}}m^4)(z+m+\zeta m) } = O(\Phi).
	\end{align*}
	This proves Lemma \ref{lem:main lemma of mean LSS}.
\endproof
We now choose an initial (small) $\epsilon >0$. We use the factor $N^\epsilon$ and allow $\epsilon$ to increase by a tiny amount from line to line. We often drop $z$ from the notation; it always understood that $z \in \Gamma_2$ and all estimates uniform on $\Gamma_2$ and also for sufficiently large $N$. The proof of Lemma \ref{lem:truncated} is done in remaining Subsections \ref{sec:subsec1}-\ref{sec:subsec5} where $\E I_{r,s}$ and the error terms controlled.
Obviously, the error term $\E \Omega_l(I)$ is negligible by simple power counting, we omit the details.

\subsubsection{Estimate on {$I_{2,0}$}}\label{sec:subsec1}
Recall the definition of $I_{r,s}$. We have
\[
I_{2,0}:= \frac{1}{N} \sum_{i\neq k} \kappa_{ij}^{(3)}{ \pB{\partial_{ij}^{2} G_{ij}} Q}
\] 
We note that $I_{2,0}$ contains terms with one or three off-diagonal Green function entries $G_{ij}$. We denote them by $\E I_{2,0}^{(1)}$ and $\E I_{2,0}^{(3)}$, respectively.

\begin{lemma}\label{lemma:EI_{2,0}^1}
	For any small $\epsilon >0$ and for all $z \in \mathcal{E}$, we have
	\begin{equation}\label{eq:EI_{2,0}^1}
		|\E I_{2,0}^{(1)}| \leq \frac{N^\epsilon}{q^2\sqrt{N}}  +\Phi, \qquad 	|\E I_{2,0}^{(3)}| \leq \Phi.
	\end{equation}
	for $N$ sufficiently large. In particular, $\E I_{2,0}$ is negligible.
\end{lemma}

\begin{proof}
	Since $I_{2,0}^{(3)}$ has three off diagonal terms, we can bound one $G_{ij}$ by $q^{-1}$ and extract one factor of $\frac{1}{N}$ from others. Thus  we have
	\begin{equation}
		\absb{\E I_{2,0}^{(3)}}=\absBB{N\E \qa{\sum_{i\neq j} \frac{\kappa_{ij}^{(3)}}{N^2}G_{ij}^3Q } } \leq \frac{N^\epsilon}{Nq^2} \leq \Phi.
	\end{equation}

	Fix a small $\epsilon >0$. From the definition of $I_{2,0}^{(1)}$ and local law, we have
	\begin{equation}\label{eq:def I_{2,0}}
		\E I_{2,0}^{(1)}=N\E \qa{\sum_{i\neq j} \frac{\kappa_{ij}^{(3)}}{N^2}G_{ij}G_{ii}G_{jj}Q    } = N\E \qa{\sum_{i\neq j} \frac{\kappa_{ij}^{(3)}}{N^2}G_{ij}m^2 Q    } + \Phi 
	\end{equation}
	Using the resolvent formula we expand in the index $j$ to get 
	\begin{equation}\label{eq:zEI_{2,0}^{(1)}}
		z\E I_{2,0}^{(1)} = N\E \qa{\sum_{i\neq j \neq k} \frac{\kappa_{ij}^{(3)}}{N^2}H_{jk}G_{ki}m^2 Q   }.
	\end{equation}
	Similar as Section \ref{sec:I_2} and \ref{sec:I_3}, applying the cumulant expansion to the right side of \eqref{eq:zEI_{2,0}^{(1)}}, we can show that the leading term of \eqref{eq:zEI_{2,0}^{(1)}} is $-\E[mI_{2,0}^{(1)}]$. Then changing $m(z)$ by the deterministic quantity $m_{sc} (z)$ and showing all the other terms in the cumulant expansion are negligible. Then we will get 
	\begin{equation}
		|z+m_{sc}(z)||\E I_{2,0}^{(1)}|\leq  \frac{N^\epsilon}{q^2\sqrt{N}}  +O(\Phi)  = O(\Phi),
	\end{equation}
	for sufficiently large $N$. Since $|z+m_{sc}| > c $ for some constant $c>0$ uniformly on $\Gamma_2$, the lemma follows directly.                                                                                                                                                                                                                                                                                                                                                                                                                                                                                                                                                                                                                                                                                                                                                                                                                                                                                                                                                                                                                                                                                                                                                                                                                                                                                                                                                                                                                                                                                                                                                                                                                                                                                                                                                                                                                                                                                                                                                                                                                                                                                                                                                                                                                                                                                                                                                                                                                                                                                                                                                                                                                                                                                                                                                                                                                                                                                                                                                                                                                                                                                                                                                                                         
	
	For simplicity we abbreviate $\hat{I}\equiv I_{2,0}^{(1)}$. Using Lemma \ref{lemma:Stein}, for arbitrary $l' \in \N$ we have the cumulant expansion
	\begin{equation}\label{eq:zEhat I}
		z\E \hat{I} = \sum_{r'=1}^{l'} \sum_{s'=0}^{r'} w_{\hat{I}_{r',s'}} \E \hat{I}_{r',s'} + \E \Omega_{l'}(\hat{I})
	\end{equation}
	with
	\begin{equation}
		\hat{I}_{r',s'}=\frac{1}{N} \sum_{i\neq j\neq k} \kappa_{ij}^{(3)} \kappa_{jk}^{(r'+1)}{ \pB{\partial_{jk}^{r'-s'} G_{ki}} \pB{\partial_{jk}^{s'}{m^2 Q}}}
	\end{equation}
	with $ w_{\hat{I}_{r',s'}}= \frac{1}{(r'-s')!s'!}$. It can be checked that the error term  $\E\Omega_{l'}(\hat{I})$ is negligible for large $l'$.
	
	\begin{remark}\label{rmk:powercounting}
		Consider the terms $\hat{I}_{r',s'}$ with $1\leq s'\leq r'$.
		We claim that it is enough to show that
		\begin{equation}\label{eq:wt I}
			\wt	{I}_{r',0}:=\frac{1}{N} \sum_{i\neq j\neq k} \kappa_{ij}^{(3)} \kappa_{jk}^{(r'+1)}{ \pB{\partial_{jk}^{r'} G_{ki}} } m^2Q,
		\end{equation}
		are negligible for $r' \geq 2$ since $\partial_{jk}^{r'} G_{ki}$ contains at least one off-diagonal entry.
		Consider when $\partial_{jk}$ acts on $Qm^2$. Note that $Q$ is third order polynomial in $m$.
		
		Using $|Q'|\prec 1$, $|Q''|\prec 1 $, $|Q'''|\prec 1 $ and Lemma \ref{lemma:powercounting} we have 
		\begin{align}\label{eq:partial Q}
			\abs{\partial_{jk} (m^2 Q)} &= \absB{\pB{\frac{1}{N}\sum_{u=1}^{N}G_{uj}G_{ku}} (m^2Q)'} \prec \frac{1 }{N}, 
		\end{align}
		where the summation index $u$ is generated from $\partial_{ik}Q$. More generally, it can be easily shown that $\partial_{jk}^{s'}(m^2Q)$ contains at least two off-diagonal Green function entries for $n\geq 1$. 
		Thus we conclude that if $\partial_{jk}$ acts $s' \geq 1 $ times on $m^2Q$ then with $\abs{G} \prec 1$, we have 
		
		\begin{equation}
			\hat{I}_{r',s'}=\frac{1}{N} \sum_{i\neq j\neq k} \kappa_{ij}^{(3)} \kappa_{jk}^{(r'+1)}{ \pB{\partial_{jk}^{r'-s'} G_{ki}} \pB{\partial_{jk}^{s'}m^2Q}} \prec  \frac{1}{q^rN},
		\end{equation}
		which is negligible.
	\end{remark}

	We now estimate $\hat{I}_{r',0}$. For $r'=1$, we compute
	\begin{align}\label{eq:hatI1 split}
		\E \hat{I}_{1,0} &= -\E \qB{\frac{1}{N} \sum_{i\neq j\neq k} \kappa_{ij}^{(3)} \kappa_{jk}^{(2)}{ G_{ji}G_{kk}m^2 {Q}}  }  -\E \qB{\frac{1}{N} \sum_{i\neq j\neq k} \kappa_{ij}^{(3)} \kappa_{jk}^{(2)}{ G_{jk}G_{ki}m^2 {Q}}  } \notag \\ 
		&=: \E \hat{I}_{1,0}^{(1)} + \E \hat{I}_{1,0}^{(2)} 
	\end{align}
	where we organize the terms according to the off-diagonal Green function entries. By Lemma \ref{lemma:powercounting},
	\begin{equation}\label{eq:hatI^2,3}
		|\E\hat{I}_{1,0}^{(2)}| \leq \frac{N^\epsilon}{Nq}\leq \Phi.
	\end{equation}
	We rewrite $\hat{I}_{1,0}^{(1)}$ with $m_{sc}$ as 
	\begin{align}\label{eq:hatI^1}
		\E \hat{I}_{1,0}^{(1)} &= -\E \qB{\frac{1}{N} \sum_{i\neq j\neq k} \kappa_{ij}^{(3)} \kappa_{jk}^{(2)}{ G_{ji}G_{kk}m^2 {Q}}  } \notag \\
		&= -\E \qB{\frac{1}{N} \sum_{i\neq j} \kappa_{ij}^{(3)} { G_{ji} m^3 Q}  } +O(\Phi) \notag \\
		&= -\E \qB{\frac{1}{N} \sum_{i\neq j} \kappa_{ij}^{(3)} { G_{ji} m_{sc} m^2 Q}  } -\E \qB{\frac{1}{N} \sum_{i\neq j} \kappa_{ij}^{(3)} { (m-m_{sc})G_{ji}m^2 Q}  }  +O(\Phi).
	\end{align}
	By local law, for sufficiently large $N$, the second term in \eqref{eq:hatI^1} bounded as 
	\begin{align}\label{eq:hatI^1 2term}
		\absbb{\E \qB{\frac{1}{N} \sum_{i\neq j} \kappa_{ij}^{(3)} { (m-m_{sc})G_{ji}m^2 Q}  } } \leq \frac{N^{\epsilon}}{q}\E \qB{\frac{1}{N^2} \sum_{i\neq j} |m-m_{sc}||G_{ij}||Q|} \leq \frac{N^\epsilon}{\sqrt{N}q^3} = O(\Phi).
	\end{align}
	Thus, we get that
	\begin{equation}\label{eq:EhatI_{1,0}}
		\E \hat{I}_{1,0} = -m_{sc} \E \qB{\frac{1}{N} \sum_{i\neq j} \kappa_{ij}^{(3)} { G_{ji}m^2 Q}  } + O(\Phi) = -m_{sc}\E  I_{2,0}^{(1)} + O(\Phi).
	\end{equation}
	We remark that in the expansion of $\E \hat{I} = \E I_{2,0}^{(1)} $ the only term with one off-diagonal entry is $\E\hat{I}_{2,0}^{(1)}$. All the other terms contain at least two off-diagonal entries, thus, negligible.

	To sum up, we find that for sufficiently large $N$,
	\begin{equation}
		|z+m_{sc}||\E I_{2,0}^{(1)}| = O(\Phi).
	\end{equation}
	Since $|z+m_{sc} |> c$, we obtain $|\E I_{2,0}^{(1)}| = O(\Phi)$. This concludes the proof of \eqref{eq:EI_{2,0}^1}.	
\end{proof}	
Summarizing, we showed that 
\begin{equation}
	|EI_{2,0}|\leq \Phi,
\end{equation}
for $N$ sufficiently large and the second estimate in \eqref{eq:estimate I_1} is proved.

\subsubsection{Estimate on \texorpdfstring{$I_{3,0}$}{I3,0}} Note that $I_{3,0}$ contains terms with zero, two or four off-diagonal Green function entries. We split accordingly
\[
w_{I_{3,0}}I_{3,0} = w_{I_{3,0}^{(0)}}I_{3,0}^{(0)} + w_{I_{3,0}^{(2)}}I_{3,0}^{(2)} +w_{I_{3,0}^{(4)}}I_{3,0}^{(4)}. 
\]
When there are two off-diagonal entries, from Lemma \ref{lemma:powercounting}, we obtain
\begin{equation}
	|\E I_{3,0}^{(2)} | \leq \absbb{ N\max_{i,j}\kappa_{ij}^{(4)} \E \qB{\frac{1}{N^2} \sum_{i\neq j} G_{ii}G_{jj}(G_{ij})^2 Q }   } \leq \frac{N^\epsilon}{Nq^2}\E|Q| \leq \Phi,
\end{equation}
for sufficiently large $N$ and similar argument holds for $\E I_{3,0}^{(4)}$. Thus the only non-negligible term is $I_{3,0}^{(0)}$.

\begin{align}
	w_{I_{3,0}^{(0)}}\E I_{3,0}^{(0)} &= -\frac{1}{N}\E \qB{ \sum_{i\neq j} \kappa_{ij}^{(4)}G_{ii}^2 G_{jj}^2Q} \nonumber \\
	&= -\frac{1}{N}\E \qB{ \sum_{i, j} \kappa_{d}^{(4)}G_{ii}^2 G_{jj}^2Q} -\frac{1}{N}\E \qB{ \sum_{i}\sum_{j\sim i} (\kappa_{s}^{(4)}-\kappa_{d}^{(4)})G_{ii}^2 G_{jj}^2Q} 	\nonumber \\ 
	&=-\frac{1}{N}\kappa_{d}^{(4)}\E \qB{\sum_{i, j}  G_{ii}^2 G_{jj}^2 Q}-\frac{1}{N}(\kappa_{s}^{(4)}-\kappa_{d}^{(4)})\E \qB{ \sum_{i}\sum_{j\sim i} G_{ii}^2 G_{jj}^2Q} 	\nonumber\\
	&=-\frac{1}{N}\kappa_{d}^{(4)}\E \qB{\sum_{i, j}  G_{ii}^2 G_{jj}^2Q}-\frac{1}{N}(\kappa_{s}^{(4)}-\kappa_{d}^{(4)})\E \qB{ \sum_{i}\sum_{j\sim i} G_{ii}^2 G_{jj}^2Q} .
\end{align}

We have
\begin{align}
G_{ii}^2 & G_{jj}^2 = (G_{ii}^2-m^2) (G_{jj}^2-m^2) + m^2 G_{ii}^2 + m^2 G_{jj}^2 -m^4 \notag \\
	&= O(q^{-2}) + m^2 \pB{(G_{ii}-m)^2 + 2G_{ii}m -m^2   } + m^2 \pB{(G_{jj}-m)^2 + 2G_{jj}m -m^2   } -m^4 \notag \\
	&=O(q^{-2}) + 2m^3(G_{ii}+G_{jj}) -3m^4,
\end{align}
where $|G_{ii}-m|\prec q^{-1}$ by local semicircle law.
Therefore, for the first term, we can conclude that 
\begin{align}
	\frac{1}{N}\kappa_{d}^{(4)}\E \qB{\sum_{i, j}  G_{ii}^2 G_{jj}^2Q} &= 	\frac{1}{N}\kappa_{d}^{(4)}\E \qB{\sum_{i, j}  \pa{2m^3(G_{ii}+G_{jj}) -3m^4}Q} +O(\frac{N^\epsilon}{q^4}) \notag \\
	&=N\kappa_{d}^{(4)}\E \qB{m^4 Q } + O(\frac{N^\epsilon}{q^4}).
\end{align}

Similarly we can estimate the second term by 	
\begin{align}
	\frac{1}{N}(&\kappa_{s}^{(4)}-\kappa_{d}^{(4)})\E \qB{ \sum_{i}\sum_{j\sim i} G_{ii}^2 G_{jj}^2Q} \notag\\
	&= 	\frac{1}{N}(\kappa_{s}^{(4)}-\kappa_{d}^{(4)}) \E \qB{ \sum_{i}\sum_{j\sim i}\pa{ 2m^3(G_{ii}+G_{jj}) -3m^4 }Q} + O(\frac{N^\epsilon}{q^4}) \notag \\
	&=\frac{N}{K}(\kappa_{s}^{(4)}-\kappa_{d}^{(4)})\E \qB{m^4Q}+O(\Phi).
\end{align}

Therefore we obtain
\begin{align}
	w_{I_{3,0}^{(0)}}\E I_{3,0}^{(0)} &= -\frac{1}{N}\E \qB{ \sum_{i\neq j} \kappa_{ij}^{(4)}G_{ii}^2 G_{jj}^2Q} \nonumber \\
	&=-N\kappa_{d}^{(4)}\E \qB{m^4Q}- \frac{N}{K}(\kappa_{s}^{(4)}-\kappa_{d}^{(4)})\E \qB{m^4Q}+O(\Phi) \nonumber \\
	&= -\E \qB{q^{-2}\xi^{(4)}Qm^4 }+O(\Phi).
\end{align}

\subsubsection{Estimate on \texorpdfstring{$I_{r,0}$ for $r\geq 4$}{Ir,0}} For $r \geq 5$ we use the bound $|G_{ii}| \prec 1 $ to obtain
\begin{align}
	|\E I_{r,0}| \leq \absB{N \E \qB{\frac{1}{N^2} \kappa_{ij}^{(r+1)} \sum_{i\neq j}  (\partial_{ij}^r G_{ij})Q } }
	\leq \frac{N^\epsilon}{q^4} \E\qB{\frac{1}{N^2}\sum_{i\neq j}1} = O(\Phi),
\end{align}
for sufficiently large $N$. For $r=4$, $\partial_{ij}^r G_{ij}$ contains at least one off-diagonal term. Hence 
\begin{align}
	\absB{N \E \qB{\frac{1}{N^2}  \sum_{i\neq j} \kappa_{ij}^{(5)} (\partial_{ij}^r G_{ij})Q } } &\leq \frac{N^\epsilon}{q^3}  \E \qB{\frac{1}{N^2} \sum_{i\neq j} |G_{ij}||Q|}   \leq \frac{C N^\epsilon}{\sqrt{N}q^3} = O(\Phi),
\end{align}
for $N$ sufficiently large. Thus we can conclude that all $I_{r,0}, r\geq 4$ are negligible. 

\subsubsection{Estimate on \texorpdfstring{$I_{r,s}$ for $r\geq 2, s\geq 1$}{Ir,s}} Similar to Remark \ref{rmk:powercounting}, if $\partial_{jk}$ act on $Q$ then it can be easily shown that those terms are negligible. We leave details for the reader. 

\subsubsection{Estimate on \texorpdfstring{$I_{1,0}$}{I1,0}}\label{sec:subsec5}
Finally we only need to estimate $\E I_{1,0}$. We have 
\begin{align}
	\E I_{1,0} =& \frac{1}{N} \sum_{i\neq j} \kappa_{ij}^{(2)}\E \qB{ \pB{\partial_{ij} G_{ij}} Q}\nonumber \\
	=&- \frac{1}{N} \E \qB{ \sum_{i,j} \kappa_{ij}^{(2)} G_{ii}G_{jj}Q}+\frac{1}{N} \E \qB{ \sum_{i} \kappa_{ii}^{(2)} G_{ii}^2Q}-\frac{1}{N} \E \qB{ \sum_{i\neq j} \kappa_{ij}^{(2)} G_{ij}^2Q}\nonumber\\
	=:& -\E I_{1,0}^{(2)} + \E I_{1,0}^{(1)} - \E I_{1,0}^{(0)}
\end{align}

The second and third term can be bounded by	$\Phi$ since
\begin{equation}\label{eq:EI_{1,0}^1}
	|\E I_{1,0}^{(1)}|= \absBB{\frac{1}{N} \E \qB{ \sum_{i} \kappa_{s}^{(2)} G_{ii}^2Q}} = O(\Phi),
\end{equation}
\begin{equation}\label{eq:EI_{1,0}^0}
	|\E I_{1,0}^{(0)}|=\absBB{\frac{1}{N} \E \qB{ \sum_{i\neq j} \kappa_{ij}^{(2)} G_{ij}^2Q}} \leq \frac{C}{N} \E {\abs{Q} }= O(\Phi).
\end{equation}

Hence we only need to estimate 
\begin{align}\label{eq:EI_{1,0}^2}
	&\E I_{1,0}^{(2)}=\frac{1}{N} \E \qB{ \sum_{i,j} \kappa_{ij}^{(2)} G_{ii}G_{jj}Q} \nonumber\\
	&=\frac{1}{N}\E \qB{ \sum_{i,j} \kappa_d^{(2)}G_{ii}G_{jj}Q}+\frac{1}{N}\E\qB{\sum_{i}\sum_{j\sim i} (\kappa_s^{(2)}-\kappa_d^{(2)})G_{ii}G_{jj}Q} \nonumber \\
	&= \E \qB{ N\kappa_d^{(2)}m^2 Q } + \frac{1}{N}\E\qB{\sum_{i}\sum_{j\sim i} (\kappa_s^{(2)}-\kappa_d^{(2)})G_{ii}G_{jj}Q} \nonumber \\
	&= \E \qB{(1-\zeta )m^2 Q } + \frac{\zeta K}{N^2}\E\qB{\sum_{i}\sum_{j\sim i} G_{ii}G_{jj}Q} \notag \\
	&= \E \qB{(1-\zeta )m^2 Q } +\frac{\zeta K}{N^2}\E\qB{\sum_{i}\sum_{j\sim i} zG_{ii}G_{jj}} + \frac{\zeta K}{N^2}\E\qB{\sum_{i}\sum_{j\sim i} G_{ii}G_{jj}\p{m+2\zeta m}}.
\end{align}

Using Lemma \ref{lemma:Stein}, we expand the second term of \eqref{eq:EI_{1,0}^2} as
\begin{align}\label{eq:expand J_{r,s}}
	\frac{\zeta K}{N^2}\E\qB{\sum_{i}\sum_{j\sim i} zG_{ii}G_{jj}} 	&= \frac{\zeta K}{N^2}\E\qB{\sum_{i}\sum_{j\sim i} \pB{\sum_{k}H_{ik}G_{ki} -1 } G_{jj}}\nonumber \\ 
	&= \frac{\zeta K}{N^2}\E\qB{\sum_{i}\sum_{j\sim i}\sum_{k\neq i}H_{ik}G_{ki} G_{jj}} - \zeta \E\qb{m}\nonumber \\
	&=  \sum_{r=1}^{l} \sum_{s=0}^{r} w_{J_{r,s}} \E J_{r,s} - \zeta \E\qb{m} + O(\frac{N^\epsilon}{q^l}),
\end{align}
where 
\begin{align}
	w_{J_{r}}= \frac{1}{r!}, \quad J_{r}=\frac{\zeta K}{N^2}\sum_{i}\sum_{j\sim i}\sum_{k\neq i}\kappa_{ik}^{(r+1)}\E \qB{ \partial_{ik}^{r}\pB{ G_{ik}G_{jj}} }.
\end{align}

Similar as estimating $I_{r,s}$, it can be shown that all terms of $J_{r}$ and the error term are negligible except $J_{1}$ and $J_{3}$ by using Lemma \ref{lemma:powercounting}. We omit the details.
\subsubsection{Estimate on \texorpdfstring{$J_{1}$}{J1}}

\beq \begin{split}
	\E J_{1}&=\frac{\zeta K}{N^2}\sum_{i}\sum_{j\sim i}\sum_{k\neq i}\kappa_{ik}^{(2)}\E \qB{ \pB{\partial_{ik} G_{ik}G_{jj}} } \\
	&=-\frac{\zeta K}{N^2}\sum_{i}\sum_{j\sim i}\sum_{k}\kappa_{d}^{(2)}\E \qB{ {G_{ii} G_{jj}G_{kk}+G_{ik}^2G_{kk}} } \\
	&-\frac{\zeta K}{N^2}\sum_{i}\sum_{j\sim i}\sum_{k\sim i}(\kappa_{s}^{(2)}-\kappa_{d}^{(2)})\E \qB{ {G_{ii} G_{jj}G_{kk}+G_{ik}^2G_{kk}} } \\
	& -2\frac{\zeta K}{N^2}\sum_{i}\sum_{j\sim i}\sum_{k\neq i}\kappa_{ik}^{(2)}\E \qB{ {G_{ij} G_{jk}G_{ki}}}.
\end{split} \eeq
Then we can show that 
\begin{align}\label{eq:EJ_{1,0}}
	\E J_{1}=&-\frac{\zeta K}{N^2}\sum_{i}\sum_{j\sim i}\sum_{k}\kappa_{d}^{(2)}\E \qB{ {G_{ii} G_{jj}G_{kk}} } \nonumber\\
	&-\frac{\zeta K}{N^2}\sum_{i}\sum_{j\sim i}\sum_{k\sim i}(\kappa_{s}^{(2)}-\kappa_{d}^{(2)})\E \qB{ {G_{ii} G_{jj}G_{kk}} } +O(\Phi),
\end{align}
since other terms are all negligible, similar as proving \eqref{eq:EI_{1,0}^0} and \eqref{eq:EI_{1,0}^1}.
The first term can be computed by
\begin{align}\label{eq:J11}
	-\frac{\zeta K}{N^2}\sum_{i}\sum_{j\sim i}\sum_{k\neq i}\kappa_{d}^{(2)}\E \qB{ {G_{ii} G_{jj}G_{kk}} }&=	-\frac{\zeta K}{N}\sum_{i}\sum_{j\sim i}\kappa_{d}^{(2)}\E \qB{ mG_{ii} G_{jj}}\notag \\
	&=-\frac{\zeta (1-\zeta )K}{N^2}\sum_{i}\sum_{j\sim i}\E \qB{ mG_{ii} G_{jj} }.
\end{align}
For the second term, we have
\begin{align}\label{eq:J12}
	-\frac{\zeta K}{N^2}\sum_{i}\sum_{j\sim i}\sum_{k\sim i}(\kappa_{s}^{(2)}-\kappa_{d}^{(2)})\E &\qB{ {G_{ii} G_{jj}G_{kk}} } = -\frac{\zeta K}{N^2}\sum_{i}\sum_{j\sim i}\sum_{k\sim i}(\kappa_{s}^{(2)}-\kappa_{d}^{(2)})\E \qB{ {G_{ii} G_{jj}G_{kk}} } \notag \\
	&\qquad= -\frac{\zeta^2 K^2}{N^3}\sum_{i}\sum_{j\sim i}\sum_{k\sim i}\E \qB{ {G_{ii} G_{jj}G_{kk}} } \notag \\
	&\qquad= -\frac{\zeta^2 K^2}{N^3}\sum_{i\sim j\sim k}\E \qB{ m^3 - 3m^2 G_{ii} +3mG_{ii}G_{jj} }  \notag \\
	&\qquad= -\frac{\zeta^2 K^2}{N^3}\E \qB{ \frac{N^3}{K^2}m^3 - \frac{N^3}{K^2}3m^3 + \frac{N}{K}\sum_{i}\sum_{j\sim i}3mG_{ii}G_{jj} } \notag\\
	&\qquad= -\frac{3\zeta^2 K}{N^2}\E \qB{ \sum_{i}\sum_{j\sim i}mG_{ii}G_{jj}   } + \zeta^2 \E \qb{2m^3}.
\end{align}

\subsubsection{Estimate on \texorpdfstring{$J_{3}$}{J3}}
Recall that 
\[ 
w_{J_{3}}\E J_{3}=\frac{\zeta K}{N^2}\sum_{i}\sum_{j\sim i}\sum_{k\neq i}\kappa_{ik}^{(4)}\E \qB{ \partial_{ik}^{3} \pB{G_{ik}G_{jj}} }.
\]
Note that the terms contained in $J_{3}$ with more than two off-diagonal Green function entries are negligible by using Lemma \ref{lemma:powercounting}.
The only non-negligible terms are $J_{3}^{(0)}$ and $J_{3}^{(1)}$ which contain no and one off-diagonal Green function entry respectively.
By simple calculation, we get
\begin{align}
	J_{3}^{(0)}= -\frac{\zeta K}{N^2} \sum_{i}\sum_{j\sim i}\sum_{k}\kappa_{ik}^{(4)}G_{ii}^2 G_{jj}G_{kk}^2,\\
	J_{3}^{(1)}= -\frac{\zeta K}{N^2} \sum_{i}\sum_{j\sim i}\sum_{k}\kappa_{ik}^{(4)}G_{ii} G_{jj}G_{kk}^2G_{ij},
\end{align}
To estimate 	$J_{3}^{(0)}$ we expand 
\begin{align}
	(G_{ii}^2-m^2)(G_{jj}-m)(G_{kk}^2-m^2)=&-m^5+G_{jj}m^4 +(G_{ii}^2+G_{kk}^2)m^3\nonumber \\
	&-G_{jj}(G_{ii}^2+G_{kk}^2)m^2 - G_{ii}^2G_{kk}^2m + G_{ii}^2G_{jj}G_{kk}^2,
\end{align}
and by simple calculation, we obtain
\begin{align}\label{eq:G_ii^2G_jjG_kk^2}
	G_{ii}^2 G_{jj}G_{kk}^2 = -4m^5 +m^4 (G_{jj} + 2G_{ii}+2G_{kk})+O(N^\epsilon q^{-2}).
\end{align}

Hence we get 
\begin{align}
	\E J_{3}^{(0)}&=\frac{\zeta K}{N^2} \sum_{i}\sum_{j\sim i}\sum_{k}\kappa_{ik}^{(4)}\E \qb{G_{ii}^2 G_{jj}G_{kk}^2} \nonumber\\
	&= -\frac{\zeta K}{N^2} \sum_{i}\sum_{j\sim i}\sum_{k}\kappa_{d}^{(4)}\E \qb{G_{ii}^2 G_{jj}G_{kk}^2}-\frac{\zeta K}{N^2} \sum_{i}\sum_{j\sim i}\sum_{k\sim i}(\kappa_{s}^{(4)}-\kappa_{d}^{(4)})\E \qb{G_{ii}^2 G_{jj}G_{kk}^2} \nonumber\\
	&=-\zeta N\kappa_{d}^{(4)}\E \qb{m^5}-\frac{\zeta K}{N^2}(\kappa_{s}^{(4)}-\kappa_{d}^{(4)})\E \qb{ \frac{N^3}{K^2}m^5} + O(N^\epsilon q^{-4})\nonumber\\
	&= -\zeta \pB{\frac{N}{K} (\kappa_s^{(4)} - \kappa_d^{(4)} ) + N \kappa_d^{(4)}}\E\qb{m^5}+ O(\Phi) \nonumber\\
	&= -\zeta q^{-2}\xi^{(4)}\E\qb{m^5}+ O(\Phi).
\end{align}

Now we show that $\E J_{3,0}^{(1)}$ is also negligible. Using $|G_{ii}|, |G_{jj}|, |G_{kk}| \prec 1 $ and Lemma \ref{lemma:powercounting}, we get 
\begin{equation}
	|\E J_{3}^{(1)}| \leq \frac{N^\epsilon}{\sqrt{N}q^2}= O(\Phi).
\end{equation}
To sum up, we conclude that 
\begin{equation}\label{eq:J3}
	w_{J_{3}}\E J_{3} = -\zeta q^{-2}\xi^{(4)}\E\qb{m^5}+ O(\Phi).
\end{equation}

Now we estimate $\E I_{1,0}$. By \eqref{eq:EI_{1,0}^2}, \eqref{eq:J11}, \eqref{eq:J12}, \eqref{eq:expand J_{r,s}} and \eqref{eq:J3}
\begin{align}
	&\E I_{1,0}=\frac{1}{N} \E \qB{ \sum_{i,j} \kappa_{ij}^{(2)} G_{ii}G_{jj}Q} + O(\Phi) \nonumber\\
	&= -\E \qB{(1-\zeta )m^2 Q } -\frac{\zeta K}{N^2}\E\qB{\sum_{i}\sum_{j\sim i} zG_{ii}G_{jj}} - \frac{\zeta K}{N^2}\E\qB{\sum_{i}\sum_{j\sim i} G_{ii}G_{jj}\p{m+2\zeta m}} +O(\Phi) \notag\\
	&= -\frac{K\zeta(1+2\zeta)}{N^2}\E\qB{ \sum_{i}\sum_{j\sim i}   mG_{ii}G_{jj}} + \E \qb{ 2\zeta ^2 m^3 - \zeta m - \zeta q^{-2} \xi^{(4)}m^5 } \notag \\
	& \qquad- \E \qB{(1-\zeta )m^2 Q }  - \frac{\zeta K}{N^2}\E\qB{\sum_{i}\sum_{j\sim i} G_{ii}G_{jj}\p{m+2\zeta m}} + O(\Phi)  \notag\\
	& =  -\E\qB{ (z+m+2\zeta m)(1-\zeta)m^2 +2\zeta^2m^3 - \zeta m - \zeta \frac{\xi^{(4)}}{q^{2}} m^5} + O(\Phi),
\end{align}
and this concluded the proof of Lemma \ref{lem:truncated}.

\section{Proof of Lemma \ref{lemma:local law cgSBM}} \label{sec:proof of local law}
With Definitions in Appendix \ref{sec:prelim}, recall that our goal is to show that 
\begin{align}
    \left\lvert s(z)-\frac{-1}{z+m_{sc}} \right \rvert = \abs{s(z)-m_{sc}(z)} \prec \frac{\sqrt{N}}{q^4}+\frac{1}{q}
\end{align}
where $s(z)$ is defined as
\begin{align}
    s(z) = \frac{1}{N}\sum_{i,j}\G{ij}(z)
\end{align}
and $G(z) = (H-zI)^{-1}$ is the green function of matrix $H$ satisfying the Definition \ref{def:cgSBM} with $\phi>1/8$. To prove this Lemma, we prove a Lemma that gives better upper bounds for sums of entries of $G$.

\begin{lemma}[Improved Bound for $T_k$]\label{lem:T_k}
    For any $\phi>0$, 
    \begin{align}
        T_k \prec \frac{1}{q^2}
    \end{align}
    where $T_k$ is defined as
    \begin{align}
        T_k(z) = \frac{1}{\sqrt{N}} \sum_{j} \G{kj}
    \end{align}
\end{lemma}

This Lemma will be proved in Appendix \ref{subsec:Tk}. To prove the Lemma \ref{lemma:local law cgSBM}, which is the local law of cgSBM in sparse regime, the goal is to find some $\Psi(z) = o(1)$ such that
\begin{align}
    \expect{\abs{P(s)}^{2D}} \leq \abs{\Psi(z)}^{2D}
\end{align}
where $P(s)$ is defined as
\begin{align}
    P(s) = 1+(z+m_{sc}(z)-\kappa_3 m_{sc}(z)^2 )s(z)
\end{align}
with $\kappa_{3} = \sum_j \kappa_{ij}^{(3)}$.

By rewriting only one $P(s)$ in $\expect{\abs{P(s)}^{2D}}$,
\begin{align}
    & \expect{\abs{P(s)}^{2D}} \nonumber\\
    =\ & \expect{P(s)^{D-1}\overline{P(s)}^D \cdot \left\{1+(z+m_{sc}-\kappa_3 m_{sc}^2 )s(z)\right\} } \nonumber\\
    =\ & \expect{P(s)^{D-1}\overline{P(s)}^D} + \expect{P(s)^{D-1}\overline{P(s)}^D \cdot \frac{1}{N}\sum_{i,j}(HG-I)_{ij}} \nonumber\\
    & + m_{sc}\expect{P(s)^{D-1}\overline{P(s)}^D \cdot s(z)}
    - \kappa_3 m_{sc}^2 \expect{P(s)^{D-1}\overline{P(s)}^D \cdot s(z)} \nonumber\\
    =\ & \frac{1}{N}\sum_{i,j,k}\expect{P(s)^{D-1}\overline{P(s)}^D \cdot H_{ik}\G{kj}} \nonumber\\
    & + m_{sc}\expect{P(s)^{D-1}\overline{P(s)}^D \cdot s(z)}
    - \kappa_3 m_{sc}^2 \expect{P(s)^{D-1}\overline{P(s)}^D \cdot s(z)} \nonumber\\
    =\ & C_1 + C_2 + \cdots + C_t + R_t \nonumber\\
    & + m_{sc}\expect{P(s)^{D-1}\overline{P(s)}^D \cdot s(z)}
    - \kappa_3 m_{sc}^2 \expect{P(s)^{D-1}\overline{P(s)}^D \cdot s(z)} \label{eq:Ps_cumul_sum}
\end{align}
where the terms $C_1, C_2, \cdots, C_t$ and $R_t$ is generated by Lemma \ref{lemma:Stein} (cumulant expansion) and
\begin{align}
    C_r = \sum_{i,j,k} \frac{\kappa_{ik}^{(r+1)}}{r!}\expect{\partial_{ik}^{r}\left(P(s)^{D-1}\overline{P(s)}^D \cdot \frac{1}{N}\G{kj}\right)}
\end{align}
for $1 \leq r \leq t$ and
\begin{align}
    R_t = \sum_{i,j,k}\expect{\Omega_t\left(P(s)^{D-1}\overline{P(s)}^D \cdot \frac{1}{N}\G{kj}H_{ik} \right)}
\end{align}
where
\begin{align}
    \expect{\Omega_t\left(P(s)^{D-1}\overline{P(s)}^D \cdot \frac{1}{N}\G{kj}H_{ik} \right)} \leq 2C_t' \cdot \expect{\abs{H_{ij}}^{t+2}} \cdot \| \partial_{ij}^{t+1}(P(s)^{D-1}\overline{P(s)}^D \cdot \frac{1}{N}\G{kj}) \|_\infty
\end{align}
Then, we have a lemma for $C_1,\cdots,C_t$ and $R_t$ which will be proved in Appendix \ref{subsec:lemmaPs}.
\begin{lemma} \label{lemmaPs}
For sufficiently large $N$ and $r,t \geq 3$, we have
\begin{align}
    &C_1+m_{sc}\expect{P(s)^{D-1}\overline{P(s)}^D \cdot s(z)} = O\left(\frac{\sqrt{N}}{q^6}\right)\cdot\expect{\abs{P(s)}^{2D-1}} \\
    &C_2 - \kappa_3 m_{sc}^2 \expect{P(s)^{D-1}\overline{P(s)}^D \cdot s(z)} =O\left(\frac{\sqrt{N}}{q^4}\right)\expect{\abs{P(s)}^{2D-1}} \\
    &C_r = O\left(\frac{\sqrt{N}}{q^{r+1}}\right)\expect{\abs{P(s)}^{2D-1}} \\
    &R_t = O\left(\frac{\sqrt{N}}{q^{t+2}}\right) \cdot \| \abs{P(s)}^{2D-1} \|_\infty
\end{align}
\end{lemma}
Take the integer $t$ larger than $8D-2$. Continuing (\ref{eq:Ps_cumul_sum}), by Lemma \ref{lemmaPs},
\begin{align}
    \expect{\abs{P(s)}^{2D}} =\ & O\left(\frac{\sqrt{N}}{q^6}\right) \cdot \expect{\abs{P(s)}^{2D-1}} + O\left(\frac{\sqrt{N}}{q^4}\right) \cdot \expect{\abs{P(s)}^{2D-1}} \nonumber\\
    & + O\left(\frac{\sqrt{N}}{q^4}\right) \cdot \expect{\abs{P(s)}^{2D-1}} + O\left(\frac{\sqrt{N}}{q^5}\right) \cdot \expect{\abs{P(s)}^{2D-1}} + \cdots \nonumber\\
    & + O\left(\frac{1}{q^{t+2}}\right) \cdot \| \abs{P(s)}^{2D-1} \|_\infty \nonumber\\
    =\ & O\left(\frac{\sqrt{N}}{q^4}\right) \cdot \expect{\abs{P(s)}^{2D-1}} + O\left(\frac{\sqrt{N}}{q^{t+2}}\right) \cdot \| \abs{P(s)}^{2D-1} \|_\infty
\end{align}
Then since the infinite norm is constant, by Young's inequality and Jensen's inequality,
\begin{align}
    \expect{\abs{P(s)}^{2D}} & \leq \frac{1}{2D} \cdot O\left(\frac{\sqrt{N}}{q^4}\right)^{2D} + \frac{2D-1}{2D} \cdot \expect{\abs{P(s)}^{2D-1}}^{\frac{2D}{2D-1}} + O\left(\frac{1}{q^{t+2}}\right) \nonumber\\
    & \leq \frac{1}{2D} \cdot O\left(\frac{\sqrt{N}}{q^4}\right)^{2D} + \frac{2D-1}{2D} \cdot \expect{\abs{P(s)}^{2D}} + O\left(\frac{1}{q^{t+2}}\right) \nonumber\\
    & = \frac{1}{2D} \cdot O\left(\frac{\sqrt{N}}{q^4}\right)^{2D} + \frac{2D-1}{2D} \cdot \expect{\abs{P(s)}^{2D}}
\end{align}
since $t>8D-2$ and
\begin{align}
    \expect{\abs{P(s)}^{2D}} \leq O\left(\frac{\sqrt{N}}{q^4}\right)^{2D}
\end{align}
which implies
\begin{align}
    P(s) = 1+ (z+m_{sc}-\kappa_3 m_{sc}^2)s(z) \prec \frac{\sqrt{N}}{q^4}
\end{align}
so that
\begin{align}
    \abs{s(z)-\frac{-1}{z+m_{sc}-\kappa_3 m_{sc}^2}} \prec \frac{\sqrt{N}}{q^4}
\end{align}
Also,
\begin{align}
    \abs{\frac{-1}{z+m_{sc}-\kappa_3 m_{sc}^2}-\frac{-1}{z+m_{sc}}} = \abs{\frac{-1}{z+m_{sc}}}\cdot\abs{\left(1-\frac{\kappa_3 m_{sc}^2}{z+m_{sc}}\right)^{-1}-1} \prec \frac{1}{q}
\end{align}
leads to
\begin{align}
    \abs{s(z)-\frac{-1}{z+m_{sc}}} = \abs{s(z)-m_{sc}(z)} \prec \frac{\sqrt{N}}{q^4}+\frac{1}{q}
\end{align}
which finalizes the proof of Lemma \ref{lemma:local law cgSBM}.

\subsection{Proof of Lemma \ref{lem:T_k}} \label{subsec:Tk}
For any $t \geq 3$ and sufficiently large integer $D$,
\beq \begin{split}
    z\expect{\abs{T_k(z)}^{2D}}
    &= \expect{T_k(z)^{D-1}\overline{T_k(z)}^D \cdot \frac{1}{\sqrt{N}}\sum_{j}z\G{kj}} \\
    &= \expect{T_k(z)^{D-1}\overline{T_k(z)}^D \cdot \frac{1}{\sqrt{N}}\sum_{j}(GH-I)_{kj}} \\
    &= \expect{T_k(z)^{D-1}\overline{T_k(z)}^D \cdot \frac{1}{\sqrt{N}}\sum_{i,j}\G{ki}H_{ij}}
    - \expect{T_k(z)^{D-1}\overline{T_k(z)}^D \cdot \frac{1}{\sqrt{N}}} \label{T_k_first}\\
    &\approx \sum_{i,j}\expect{T_k(z)^{D-1}\overline{T_k(z)}^D \cdot \frac{1}{\sqrt{N}}\G{ki}H_{ij}}
\end{split} \eeq
where the second term in (\ref{T_k_first}) is $O(1/\sqrt{N})$ so that negligible.  By Lemma \ref{lemma:Stein},
\begin{align}
    z\expect{\abs{T_k(z)}^{2D}} = A_1 + A_2 + \cdots + A_t + R_t
\end{align}
where
\begin{align}
    A_r = \sum_{i,j} \frac{\kappa_{ij}^{(r+1)}}{r!}\expect{\partial_{ij}^{r}(T_k(z)^{D-1}\overline{T_k(z)}^D \cdot \frac{1}{\sqrt{N}}\G{ki})}
\end{align}
for $1 \leq r \leq t$ and
\begin{align}
    R_t = \sum_{i,j}\expect{\Omega_t\left(T_k(z)^{D-1}\overline{T_k(z)}^D \cdot \frac{1}{\sqrt{N}}\G{ki}H_{ij} \right)}
\end{align}
where
\begin{align}
    \expect{\Omega_t\left(T_k(z)^{D-1}\overline{T_k(z)}^D \cdot \frac{1}{\sqrt{N}}\G{ki}H_{ij} \right)} \leq 2C_t \cdot \expect{\abs{H_{ij}}^{t+2}} \cdot \| \partial_{ij}^{t+1}(T_k(z)^{D-1}\overline{T_k(z)}^D \cdot \frac{1}{\sqrt{N}}\G{ki}) \|_\infty
\end{align}
In this proof, we use Lemma \ref{lemma:powercounting} for power counting argument frequently. For expanding the derivatives in $A_r$, we should divide the cases for the indices $i=j$ by the matrix differentiation formula (\ref{eq:matrixderivative}). However, dividing the $i=j$ cases makes the bounds of equations smaller so it is okay to consider the derivatives with respect to $H_{ij}$ for $i=j$ as same as $i \neq j$.

\subsubsection{Estimates for $A_1$}
The goal of this subsection is to prove 
\begin{align}
    A_1 = -m_{sc} \expect{\abs{T_k(z)}^{2D}}
\end{align}
where $A_1$ is defined as
\begin{align}
    A_1 \equiv \sum_{i,j} \kappa_{ij}^{(2)} \expect{\partial_{ij}\left( T_k(z)^{D-1} \overline{T_k(z)}^D \cdot \frac{1}{\sqrt{N}} \G{ki} \right)}
\end{align}
Among the terms in $A_1$, the only term that is not clearly negligible is
\begin{align}
    A_1 \approx \frac{1}{\sqrt{N}} \sum_{i,j}\kappa_{ij}^{(2)} \expect{T_k(z)^{D-1} \overline{T_k(z)}^D(-\G{kj}\G{ii})}
\end{align}
Since local law implies $\abs{\sum_i \kappa_{ij}^{(2)} \G{ii} - m_{sc}}\prec 1/q^2$, by taking $i$-sum first,
\begin{align}
    A_1 &\approx -m_{sc} \cdot \frac{1}{\sqrt{N}} \sum_{j}\expect{T_k(z)^{D-1} \overline{T_k(z)}^D\G{kj}} \nonumber\\
    &= -m_{sc} \cdot \expect{T_k(z)^{D-1} \overline{T_k(z)}^D \frac{1}{\sqrt{N}} \sum_{j}\G{kj}} \nonumber\\
    &= -m_{sc} \expect{\abs{T_k(z)}^{2D}}
\end{align}

\subsubsection{Estimates for $A_2$}
The goal of this subsection is to prove
\begin{align}
    A_2 = \kappa_3 \cdot m_{sc}^2 \cdot  \expect{\abs{T_k(z)}^{2D}} + O(1/q^2) \expect{\abs{T_k(z)}^{2D-1}}
\end{align}
where $A_2$ is defined as
\begin{align}
    A_2 \equiv \sum_{i,j} \frac{\kappa_{ij}^{(3)}}{2} \expect{\partial_{ij}^2 \left( T_k(z)^{D-1} \overline{T_k(z)}^D \cdot \frac{1}{\sqrt{N}} \G{ki} \right)}
\end{align}
For this subsection and some next subsections, we will introduce some additional definition.
\begin{definition}
For integer $r \geq 2$ and integer $1 \leq p \leq 6$, define $A_{r,p}$ as following:
\begin{align}
    &A_{r,1} = \frac{1}{r!\sqrt{N}} \sum_{i,j} \kappa_{ij}^{(r+1)} \expect{\partial_{ij}^{r-2}\{\partial_{ij}^2(T_k(z)^{D-1}) \overline{T_k(z)}^D\G{ki}\}} \\
    &A_{r,2} = \frac{1}{r!\sqrt{N}} \sum_{i,j} \kappa_{ij}^{(r+1)} \expect{\partial_{ij}^{r-2}\{T_k(z)^{D-1} \partial_{ij}^2(\overline{T_k(z)}^D)\G{ki}\}} \\
    &A_{r,3} = \frac{1}{r!\sqrt{N}} \sum_{i,j} \kappa_{ij}^{(r+1)} \expect{\partial_{ij}^{r-2}\{T_k(z)^{D-1} \overline{T_k(z)}^D\partial_{ij}^2(\G{ki})\}} \\
    &A_{r,4} = \frac{1}{r!\sqrt{N}} \sum_{i,j} \kappa_{ij}^{(r+1)} \expect{\partial_{ij}^{r-2}\{2\partial_{ij}(T_k(z)^{D-1}) \partial_{ij}(\overline{T_k(z)}^D)\G{ki}\}} \\
    &A_{r,5} = \frac{1}{r!\sqrt{N}} \sum_{i,j} \kappa_{ij}^{(r+1)} \expect{\partial_{ij}^{r-2}\{2\partial_{ij}(T_k(z)^{D-1}) \overline{T_k(z)}^D\partial_{ij}(\G{ki})\}} \\
    &A_{r,6} = \frac{1}{r!\sqrt{N}} \sum_{i,j} \kappa_{ij}^{(r+1)} \expect{\partial_{ij}^{r-2}\{2T_k(z)^{D-1} \partial_{ij}(\overline{T_k(z)}^D)\partial_{ij}(\G{ki})\}} 
\end{align}
It is clear that $A_r = A_{r,1}+A_{r,2}+A_{r,3}+A_{r,4}+A_{r,5}+A_{r,6}$.
\end{definition}
By simple calculation and counting, it is clear that $A_{2,3}$ is the only non-negligible term among the six components of $A_2$. By calculating the derivative in 
\begin{align}
    A_{2,3} = \frac{1}{2\sqrt{N}} \sum_{i,j} \kappa_{ij}^{(3)} \expect{T_k(z)^{D-1} \overline{T_k(z)}^D\partial_{ij}^2(\G{ki})},
\end{align}
the only non-negligible term is 
\begin{align}
    A_{2,3} &\approx \frac{1}{\sqrt{N}} \sum_{i,j} \kappa_{ij}^{(3)} \expect{T_k(z)^{D-1} \overline{T_k(z)}^D\G{ki}\G{jj}\G{ii}} \nonumber\\
    &= \frac{1}{\sqrt{N}} \sum_{i,j} \kappa_{ij}^{(3)} \expect{T_k(z)^{D-1} \overline{T_k(z)}^D\G{ki}(\G{jj}-m_{sc})\G{ii}} \label{A231_1}\\
    &\quad + \frac{1}{\sqrt{N}} \sum_{i,j} \kappa_{ij}^{(3)} \expect{T_k(z)^{D-1} \overline{T_k(z)}^D\G{ki}m(\G{ii}-m_{sc})} \label{A231_2}\\
    &\quad + \frac{1}{\sqrt{N}} \sum_{i,j} \kappa_{ij}^{(3)} \expect{T_k(z)^{D-1} \overline{T_k(z)}^D\G{ki}m_{sc}^2} \label{A232}
\end{align}
Define $A_{2,3,1}$ as the sum of (\ref{A231_1}) and (\ref{A231_2}), and $A_{2,3,2}$ as (\ref{A232}) so that $A_{2,3} = A_{2,3,1} + A_{2,3,2}$. Then,
\begin{align}
    \abs{A_{2,3,1}} &\leq \frac{2}{\sqrt{N}}\cdot N^2 \cdot \expect{\abs{T_k(z)}^{2D-1}} \cdot \frac{1}{Nq} \cdot \frac{1}{\sqrt{N}} \cdot \frac{1}{q} \cdot C \nonumber\\
    &= 2C \cdot \frac{1}{q^2} \expect{\abs{T_k(z)}^{2D-1}}
\end{align}
and by taking $j$-sum first,
\begin{align}
    A_{2,3,2} &= \kappa_3 \cdot m^2 \cdot \expect{T_k(z)^{D-1} \overline{T_k(z)}^D \cdot \frac{1}{\sqrt{N}} \sum_{i}\G{ki}} \nonumber\\
    &= \kappa_3 \cdot m^2 \cdot \expect{\abs{T_k(z)}^{2D}}
\end{align}
Therefore,
\begin{align}
    A_2 \approx A_{2,3} = A_{2,3,1} + A_{2,3,2} = O(1/q^2) \expect{\abs{T_k(z)}^{2D-1}} + \kappa_3 \cdot m^2 \cdot  \expect{\abs{T_k(z)}^{2D}}
\end{align}

\subsubsection{Estimates for $A_r \ (r\geq3)$}
The goal of this subsection is to prove
\begin{align}
    A_r = O(1/q^{r-1}) \expect{\abs{T_k(z)}^{2D-1}}
\end{align}
for all integer $r \geq 3$ where $A_r$ is defined as
\begin{align}
    A_r = \sum_{i,j} \frac{\kappa_{ij}^{(r+1)}}{r!}\expect{\partial_{ij}^{r}\left(T_k(z)^{D-1}\overline{T_k(z)}^D \cdot \frac{1}{\sqrt{N}}\G{ki}\right)}
\end{align}
We introduce new definitions and lemmas for proof.
\begin{definition}[Effective Non-diagonal Entries] \label{effentry}
The number of \textbf{effective non-diagonal entries} of some product of the entries of $G = (H-zI)^{-1}$ is defined as the number of non-diagonal entries of $G$ except $\G{ij},\G{ji}$ while the power larger than $2$ is counted as $2$. For example, the number of effective non-diagonal entries in $\G{ki}^3\G{kj}\G{ij}^2\G{ii}\G{jj}$ is $2+1=3$.
\end{definition}
\begin{definition}[Effectively negligible] \label{effneg}
The term with constant, sigma, kappa, expectation, $T_k$, and some entries of $G=(H-zi)^{-1}$ is \textbf{effectively negligible} if
\begin{align}
    n_c+n_s-1-\frac{1}{2}n_{end}<0
\end{align}
where $n_c$ is the degree of $N$ of the constant, $n_s$ is the number of indices of sigma and $n_{end}$ is the number of effective non-diagonal entries of the term.  By Lemma \ref{lemma:powercounting}, if a term is effectively negligible, then the term is negligible even without $\G{ij}, \G{ji}$ hence negligible. 
\end{definition}
For example, negligible term
\begin{align}
    \frac{1}{\sqrt{N}}\sum_{i,j}\kappa_{ij}^{(4)}\expect{T_k(z)^{D-1}\overline{T_k(z)}^D \cdot \G{ki}\G{ij}}
\end{align}
is not effectively negligible since $n_c=-1/2, n_s=2$, and $n_{end}=1$. Using this definition, we can make a lemma
\begin{lemma} \label{lemma_partial_rev}
Let $F(H_{ij})$ be a multiplication of the entries of $G=(H-zi)^{-1}$ with $n_{end}$ effectively negligible entries. Also, suppose that the powers of indices except $k,i,j$ and the powers of $\G{ki},\G{kj},\G{ik},\G{jk}$ are at most $1$. Then, among $\partial_{ij} F(H_{ij})$, there is no terms whose number of effectively negligible entries less than $n_{end}$.
\end{lemma}
\begin{proof}
Since $\partial_{ij}\G{pq} = -\G{pi}\G{jq} - \G{pj}\G{iq}$, without loss of generality, we will consider $\partial_{ij}$ replaces $\G{pq}$ into $\G{pi}\G{jq}$. Suppose that there exists a term in $\partial_{ij} F(H_{ij})$ such that have effectively negligible entries less than $n_{end}$. Then, $\G{pq}$ should be effectively negligible but $\G{pi}$ and $\G{jq}$ which are newly added should be not effectively negligible. Then, $\G{pi}$ and $\G{jq}$ should satisfy one of following conditions. The first one is to be diagonal entry, $\G{ij}$, or $\G{ji}$ and the second condition is to be the entry originally placed so that the power of such entry became larger than $2$ and not in the first condition. Now, there are three possible cases.
\begin{enumerate}
    \item If both $\G{pi}$ and $\G{jq}$ satisfy the first condition, then $p$ and $q$ should be either $i$ and $j$ which makes $\G{pq}$ be not effectively negligible entry.
    \item If one of $\G{pi}$ and $\G{jq}$ satisfies the first condition and the another satisfies the second condition, without loss of generality, suppose that $\G{pi}$ and $\G{jq}$ satisfy the first and second condition respectively. Then since $p=i$ or $p=j$, $q$ should be $k$ or index except $i,j,k$ so that $\G{pq}$ be effectively negligible. However, if $q=k$, then $\G{jq} = \G{jk}$ should have power larger than $2$ which is impossible by assumption. Also, if $q$ is a index except $i,j,k$, namely $\alpha$, then $\G{jq} = \G{j\alpha}$ should have power larger than $2$ which is impossible by assumption again.
    \item If both $\G{pi}$ and $\G{jq}$ satisfy the second condition, since they should not satisfy the first condition, $p$ and $q$ cannot be $i,j$. By the same reason as second case, $p$ and $q$ cannot be the indices except $i,j,k$. Therefore, $p=k$ and $q=k$ which leads to contradiction since $\G{pq}$ is not diagonal entry.
\end{enumerate}
Therefore, there is no such terms with effectively negligible entries less than $n_{end}$.
\end{proof}

\begin{lemma} \label{lemma_eff_rev}
Let $F(H_{ij})$ be a sum of some multiplications of the entries of $G=(H-zI)^{-1}$. Also, suppose that the powers of indices except $k,i,j$ and the powers of $\G{ki},\G{kj},\G{ik},\G{jk}$ are at most $1$. Then if $\sum_{i,j}\kappa_{ij}^{(r)}\expect{T_k(z)^\alpha\overline{T_k(z)}^\beta F(H_{ij}) }$ is effectively negligible, then
\begin{align}
    \sum_{i,j}\kappa_{ij}^{(r+1)}\expect{\partial_{ij}\left\{T_k(z)^\alpha\overline{T_k(z)}^\beta F(H_{ij})\right\} }
\end{align}
is also effectively negligible where $\alpha, \beta \geq 2$ and $n\geq3$
\end{lemma}
\begin{proof}
Since $\sum_{i,j}\kappa_{ij}^{(r)}\expect{T_k(z)^\alpha\overline{T_k(z)}^\beta F(H_{ij}) }$ is effectively negligible, $n_c+n_s-1-\frac{1}{2}n_{end}<0$ where $n_c$ is the degree of $N$ of the constant in $F$, $n_s$ is the number of indices of sigma and $n_{end}$ is the number of effective non-diagonal entries of $G$ in $F$. Then,
\begin{align}
    \sum_{i,j}\kappa_{ij}^{(r+1)}\expect{\partial_{ij}\left\{T_k(z)^\alpha\overline{T_k(z)}^\beta F(H_{ij})\right\} } = L_1 + L_2 + L_3
\end{align}
where
\begin{align}
    &L_1 = \alpha \sum_{i,j}\kappa_{ij}^{(r+1)}\expect{T_k(z)^{\alpha-1}\overline{T_k(z)}^\beta F(H_{ij})\cdot \partial_{ij}T_k(z)} \\
    &L_2 = \beta \sum_{i,j}\kappa_{ij}^{(r+1)}\expect{T_k(z)^{\alpha}\overline{T_k(z)}^{\beta-1} F(H_{ij})\cdot \partial_{ij}\overline{T_k(z)}} \\
    &L_3 = \sum_{i,j}\kappa_{ij}^{(r+1)}\expect{T_k(z)^\alpha\overline{T_k(z)}^\beta F'(H_{ij})}
\end{align}
Let $n_c', n_s'$, and $n_{end}'$ be the constant of any term of $L_1$. Then since
\begin{align}
    \abs{\partial_{ij}T_k(z)} = \abs{\frac{1}{\sqrt{N}}\sum_q(-\G{ki}\G{jq}-\G{kj}\G{iq}) }
\end{align}
has new effectively non-diagonal term $\G{jq}$ and $\G{iq}$, $n_c' = n_c -1/2$, $n_s' = n_s + 1$, and $n_{end}' \geq n_{end}+1$. Therefore,
\begin{align}
    n_c'+n_s'-1-\frac{1}{2}n_{end}' &\leq n_c -1/2 +n_s+1-1-\frac{1}{2}(n_{end}+1) \nonumber\\
    & = n_c+n_s-1-\frac{1}{2}n_{end} <0
\end{align}
which implies any term of $L_1$ is effectively negligible so that $L_1$ is effectively negligible. Similarly, $L_2$ is N-effectively negligible also. For $L_3$, consider any term of $L_3$ with constants $n_c'', n_s''$ and $n_{end}''$, By the Lemma \ref{lemma_partial_rev}, $n_{end}'' \geq n_{end}$ which implies
\begin{align}
    n_c'' + n_s'' -1 -\frac{1}{2}n_{end}'' \leq n_c + n_s -1 -\frac{1}{2}n_{end} <0
\end{align}
so that any terms of $L_3$ is effectively negligible which makes $L_3$ became effectively negligible. Since $L_1$, $L_2$, and $L_3$ are effectively negligible,
\begin{align}
    \sum_{i,j}\kappa_{ij}^{(r+1)}\expect{\partial_{ij}\left\{T_k(z)^\alpha\overline{T_k(z)}^\beta F(H_{ij})\right\} }
\end{align}
is effectively negligible also.
\end{proof}

We will prove for $A_{r,p}$ ($p\neq3$) first.
\begin{theorem} \label{Arpthm}
$A_{r,p}$ is effectively negligible for $r \geq 2$ and $p \neq 3$
\end{theorem}
\begin{proof}
Consider $A_{r,1}$ first. We will use induction on $r \geq 2$. First, it can be shown that $A_{2,1}$ is effectively negligible by calculation. Suppose that
\begin{align}
    A_{r',1} = \frac{1}{(r')!\sqrt{N}} \sum_{i,j} \kappa_{ij}^{(r'+1)} \expect{\partial_{ij}^{r'-2}\left\{\partial_{ij}^2(T_k(z)^{D-1}) \overline{T_k(z)}^D\G{ki}\right\}}
\end{align}
is effectively negligible and define $F(H_{ij})$ as the sum of the terms $T_k(z)^{\alpha} \overline{T_k(z)}^{\beta} F(H_{ij})$
\begin{align}
    \sum_{t\in T} T_k(z)^{\alpha_t} \overline{T_k(z)}^{\beta_t} F_t(H_{ij}) = \partial_{ij}^{r'-2}\left\{\partial_{ij}^2(T_k(z)^{D-1}) \overline{T_k(z)}^D\G{ki}\right\}
\end{align}
Since the indices except $i,j,k$ should be came from
\begin{align}
    \partial_{ij}T_k(z) = \frac{1}{\sqrt{N}}\sum_{\gamma}(-\G{ki}\G{j\gamma}-\G{kj}\G{i\gamma}),
\end{align}
the new index($\gamma$) appears only one time. Also, define the subsets $T_1$ and $T_2$ of $T$ by
\begin{align}
    T_1 = \{ t\in T \mid \text{indices of }\G{ki}, \G{kj}, \G{ik}, \G{jk} \text{ of }F_t(H_{ij})\text{ are less than 2} \}, \ T_2 = T - T_1
\end{align}
Then,
\begin{align}
    A_{r'+1,1} &= \frac{1}{(r'+1)!\sqrt{N}} \sum_{i,j} \kappa_{ij}^{(r'+2)} \expect{\partial_{ij}^{r'-1}\left\{\partial_{ij}^2(T_k(z)^{D-1}) \overline{T_k(z)}^D\G{ki}\right\}} \nonumber\\
    &=\frac{1}{(r'+1)!\sqrt{N}} \sum_{i,j} \kappa_{ij}^{(r'+2)} \expect{\partial_{ij}\left\{\sum_{t\in T} T_k(z)^{\alpha_t} \overline{T_k(z)}^{\beta_t} F_t(H_{ij})\right\}} \nonumber\\
    &=\sum_{t\in T_1}\frac{1}{(r'+1)!\sqrt{N}} \sum_{i,j} \kappa_{ij}^{(r'+2)} \expect{\partial_{ij}\left\{ T_k(z)^{\alpha_t} \overline{T_k(z)}^{\beta_t} F_t(H_{ij})\right\}} \label{Arp_proof_neg} \\
    &+\sum_{t\in T_2}\frac{1}{(r'+1)!\sqrt{N}} \sum_{i,j} \kappa_{ij}^{(r'+2)} \expect{\partial_{ij}\left\{ T_k(z)^{\alpha_t} \overline{T_k(z)}^{\beta_t} F_t(H_{ij})\right\}} \label{Arp_proof_neg_2}
\end{align}
and (\ref{Arp_proof_neg}) is effectively negligible by the Lemma \ref{lemma_eff_rev}. Also, the baddest case in (\ref{Arp_proof_neg_2}) is the case when the number of effectively negligible entries decreases, denoted by $n_{end}' = n_{end}-1$. However, in this case, since $F_2$ contains either $\G{ki}^2$ or $\G{kj}^2$, which can be made only from $\partial_{ij}T_k$, $n_c + n_s -1 -\frac{1}{2}n_{end} < -\frac{1}{2}$ which leads to $n_c' + n_s' -1 -\frac{1}{2}n_{end}'<0$ so that (\ref{Arp_proof_neg_2}) is also effectively negligible. Therefore, $A_{r,1}$ is N-effectively negligible for $r \geq 2$. By similar induction, the cases for $p=2,4,5,6$ also can be proved since $A_{2,p}$ is effectively negligible for $p\neq3$.
\end{proof}

Next, we have a lemma for the case $A_{r,3}$.
\begin{lemma} \label{Ar3lemma}
Among $A_{r,3}$, the non effectively negligible terms have form of
\begin{align}
    \frac{1}{\sqrt{N}}\sum_{i,j}\kappa_{ij}^{(r+1)}\expect{T_k(z)^{D-1}\overline{T_k(z)}^D F(H_{ij})}
\end{align}
for all $r \geq 2$ where $F(H_{ij})$ is a sum of constant multiple of product of entries of $G=(H-zI)^{-1}$ and $n_{end}$ of $F(H_{ij})$ satisfies $n_{end} \geq 1$.
\end{lemma}
\begin{proof}
We will use induction on $r$. First, when $r=2$, we can calculate $A_{2,3}$ as
\begin{align}
    A_{2,3} = \frac{1}{\sqrt{N}}\sum_{i,j}\kappa_{ij}^{(3)}\expect{T_k(z)^{D-1}\overline{T_k(z)}^D F(H_{ij})} \label{A23proof_rev}
\end{align}
where $F(H_{ij})$ can be written as
\begin{align}
    F(H_{ij}) = \G{ki}\G{ji}^2 + \G{kj}\G{ii}\G{ji} + \G{kj}\G{ii}\G{ij} + \G{ki}\G{ii}\G{jj}
\end{align}
Then, since $n_c=-1/2, n_s=2, n_{end}=1$, (\ref{A23proof_rev}) is not effectively negligible. Therefore, $F(H_{ij})$ is desired form and it implies that our claim holds for $r=2$. Suppose that the non effectively negligible terms of $A_{r',3}$ has form of 
\begin{align}
    \frac{1}{\sqrt{N}}\sum_{i,j}\kappa_{ij}^{(r'+1)}\expect{T_k(z)^{D-1}\overline{T_k(z)}^D F(H_{ij})}
\end{align}
where $F(H_{ij})$ is a sum of constant multiple of product of entries of $G=(H-zI)^{-1}$ and $n_{end}\geq1$. Then,
\begin{align}
    A_{r',3} = A_{r',3,eff} + \frac{1}{\sqrt{N}}\sum_{i,j}\kappa_{ij}^{(r'+1)}\expect{T_k(z)^{D-1}\overline{T_k(z)}^D F(H_{ij})}
\end{align}
where $A_{r',3,eff}$ is effectively negligible term. By the definition of $A_{r'+1,3}$,
\begin{align}
    A_{r'+1,3} &= \frac{1}{r'+1}A_{r',3,eff}' + \frac{1}{(r'+1)\sqrt{N}}\sum_{i,j}\kappa_{ij}^{(r'+2)}\expect{\partial_{ij}\left\{T_k(z)^{D-1}\overline{T_k(z)}^D F(H_{ij})\right\}} \nonumber\\
    &= \frac{1}{r'+1}A_{r',3,eff}' + L_1 + L_2 + L_3
\end{align}
with $A_{r',3,eff}'$ is defined by replacing $\kappa_{ij}^{(r'+1)}$ with $\kappa_{ij}^{(r'+2)}$ in $A_{r',3,eff}$ and put $\partial_{ij}$ inside the expectation of $A_{r',3,eff}$ which is effectively negligible by the proof of above lemmas and theorems. Also, $L_1$, $L_2$, and $L_3$ is defined by
\begin{align}
    &L_1 = \frac{D-1}{(r'+1)\sqrt{N}}\sum_{i,j}\kappa_{ij}^{(r'+2)}\expect{T_k(z)^{D-2}\overline{T_k(z)}^D F(H_{ij}) \cdot \partial_{ij}T_k(z) } \\
    &L_2 = \frac{D}{(r'+1)\sqrt{N}}\sum_{i,j}\kappa_{ij}^{(r'+2)}\expect{T_k(z)^{D-1}\overline{T_k(z)}^{D-1} F(H_{ij})\cdot \partial_{ij}\overline{T_k(z)} } \\
    &L_3 = \frac{1}{(r'+1)\sqrt{N}}\sum_{i,j}\kappa_{ij}^{(r'+2)}\expect{T_k(z)^{D-1}\overline{T_k(z)}^D F'(H_{ij})}
\end{align}
For $L_1$ and $L_2$, since
\begin{align}
    \partial_{ij}T_k(z) = \frac{1}{\sqrt{N}}\sum_q(-\G{ki}\G{jq}-\G{kj}\G{iq})
\end{align}
, $n_c'=-1,n_s'=3,n_{end}'=n_{end}+2\geq3$ implies $L_1$ and $L_2$ is effectively negligible. Therefore, non effectively negligible term of $A_{r'+1,3}$ is subsum of $L_3$ denoted by
\begin{align}
    \frac{1}{\sqrt{N}}\sum_{i,j}\kappa_{ij}^{(r'+2)}\expect{T_k(z)^{D-1}\overline{T_k(z)}^D \hat{F}(H_{ij})} 
\end{align}
while $\hat{F}(H_{ij})$ is the smallest subsum of $\frac{1}{r'+1}F'(H_{ij})$ that makes the equation
\begin{align}
    L_3 - \frac{1}{\sqrt{N}}\sum_{i,j}\kappa_{ij}^{(r'+2)}\expect{T_k(z)^{D-1}\overline{T_k(z)}^D \hat{F}(H_{ij})} 
\end{align}
is effectively negligible. Since $\hat{F}(H_{ij})$ is desired form, what remains to prove is that $n_{end}$ of $\hat{F}$ satisfies $n_{end}\geq 1$. This can be proved since $\hat{F}$ contains at least one of $\G{ki}$ or $\G{kj}$. Therefore, our claim holds for $r=r'+1$ also.
\end{proof}
Now we can prove the main theorem for $A_{r,3}$.
\begin{theorem} \label{Ar3thm}
$A_{r,3} = O(1/q^{r-1})\expect{\abs{T_k(z)}^{2D-1}}$ for all $r \geq 3$
\end{theorem}
\begin{proof}
By Lemma \ref{Ar3lemma}, we know that the non effectively negligible part of $A_{r,3}$ has the form of
\begin{align}
    \frac{1}{\sqrt{N}}\sum_{i,j}\kappa_{ij}^{(r+1)}\expect{T_k(z)^{D-1}\overline{T_k(z)}^D F(H_{ij})} \label{Ar3thmproof}
\end{align}
where $F(H_{ij})$ is a sum of constant multiple of product of entries of $G=(H-zI)^{-1}$ and $n_{end}\geq1$ for $r \geq 2$. Since $T_k(z)^{D-1}\overline{T_k(z)}^D F(H_{ij})$ is some part of $\partial_{ij}^{r}\{T_k(z)^{D-1}\overline{T_k(z)}^D\G{ki}\} $, we can know that $F(H_{ij})$ is some part of $\partial_{ij}^r \G{ki}$. By differentiating a product of entries of $G$, the length of terms increases by 1 and the index $i$ and the index $j$ appear one more tie. Therefore, $\partial_{ij}^r \G{ki}$(as so $F(H_{ij})$) is a sum of constant multiple of product of $r+1$ entries of $G$ with $2r+2$ indices ($k$ : $1$ time, $i$ : $r+1$ times, $j$ : $r$ times). Going back to (\ref{Ar3thmproof}), since the equation is not N-effectively negligible,
\begin{align}
    n_c+n_s-1-\frac{1}{2}n_{end} = -\frac{1}{2}+2-1-\frac{1}{2}n_{end} \geq 0
\end{align}
which implies $n_{end} \leq 1$ so that $n_{end} = 1$. Consider the indices of the terms in $F(H_{ij})$. Since there is only one $k$ index, $\G{ki}$ or $\G{kj}$ is the corresponding effectively non-diagonal entry($n_{end}=1$). Therefore, the largest possible term in $F(H_{ij})$ has the form
\begin{align}
    \G{ki}\times\text{(}r\text{ diagonal entries) or }\G{kj}\times\text{(}r\text{ diagonal entries)}
\end{align}
whose absolute value has order
\begin{align}
    &\abs{\frac{1}{\sqrt{N}}\sum_{i,j}\kappa_{ij}^{(r+1)}\expect{T_k(z)^{D-1}\overline{T_k(z)}^D \G{ki}\times\text{(}r\text{ diagonal entries)}}} \nonumber\\
    &\text{or } \abs{\frac{1}{\sqrt{N}}\sum_{i,j}\kappa_{ij}^{(r+1)}\expect{T_k(z)^{D-1}\overline{T_k(z)}^D \G{kj}\times\text{(}r\text{ diagonal entries)}}} \nonumber\\
    &\leq \frac{1}{\sqrt{N}}\cdot N^2 \frac{1}{Nq^{r-1}}\cdot \expect{\abs{T_k(z)}^{2D-1}} \cdot \frac{1}{\sqrt{N}} \cdot \nonumber\\
    &= \frac{1}{q^{r-1}}\expect{\abs{T_k(z)}^{2D-1}}\cdot C
\end{align}
where $C$ is a constant. Therefore, since N-effectively negligible term is negligible,
\begin{align}
    A_{r,3} &\approx \frac{1}{\sqrt{N}}\sum_{i,j}\kappa_{ij}^{(r+1)}\expect{T_k(z)^{D-1}\overline{T_k(z)}^D F(H_{ij})} \nonumber\\
    &= O(1/q^{r-1}) \expect{\abs{T_k(z)}^{2D-1}}
\end{align}
for all $r \geq 3$
\end{proof}
Finally, by Theorem \ref{Arpthm} and Theorem \ref{Ar3thm},
\begin{align}
    A_r &= (A_{r,1}+A_{r,2}+A_{r,4}+A_{r,5}+A_{r,6})+A_{r,3} \nonumber\\
    &= O(1/q^{r-1}) \expect{\abs{T_k(z)}^{2D-1}}
\end{align}

\subsubsection{Estimates for $R_t$}
The goal of the subsection is to prove
\begin{align}
    R_t = O(1/q^t) \cdot \| \abs{T_k(z)}^{2D-1} \|_\infty
\end{align}
where $R_t$ is defined as
\begin{align}
    R_t = \sum_{i,j}\expect{\Omega_t\left(T_k(z)^{D-1}\overline{T_k(z)}^D \cdot \frac{1}{\sqrt{N}}\G{ki}H_{ij} \right)}
\end{align}
where
\begin{align}
    \expect{\Omega_t\left(T_k(z)^{D-1}\overline{T_k(z)}^D \cdot \frac{1}{\sqrt{N}}\G{ki}H_{ij} \right)} \leq C_t \cdot \expect{\abs{H_{ij}}^{t+2}} \cdot \| \partial_{ij}^{t+1}(T_k(z)^{D-1}\overline{T_k(z)}^D \cdot \frac{1}{\sqrt{N}}\G{ki}) \|_\infty
\end{align}
for some constant $C_t$. Then,
\begin{align}
    \abs{R_t} &\leq C_t \sum_{i,j} \expect{\abs{H_{ij}}^{t+2}} \cdot \left \| \partial_{ij}^{t+1}(T_k(z)^{D-1}\overline{T_k(z)}^D \cdot \frac{1}{\sqrt{N}}\G{ki}) \right \|_\infty \nonumber\\
    &= C_t \sum_{i,j} \kappa_{ij}^{(t+2)} \cdot \left \| \partial_{ij}^{t+1}(T_k(z)^{D-1}\overline{T_k(z)}^D \cdot \frac{1}{\sqrt{N}}\G{ki}) \right \|_\infty \label{Rteq}
\end{align}
while (\ref{Rteq}) can be made by replacing the expectation symbol of $A_{r+1}$ with infinite norm. Also, the process of estimating the bound of $A_{r+1}$ contains no procedure using expectation. In other words, the expectation acts only as a linear operator in the process of estimating $A_{r+1}$. Therefore, we can calculate the bound of $R_t$ by same way as in $A_{r+1}$ and
\begin{align}
    R_t = O(1/q^t) \cdot \| \abs{T_k(z)}^{2D-1} \|_\infty
\end{align}
\emph{Proof for Lemma \ref{lem:T_k}.} By the subsections above, we can calculate $z\expect{\abs{T_k(z)}^{2D}}$. Take the integer $t$ larger than $4D$. Then,
\begin{align}
    z\expect{\abs{T_k(z)}^{2D}} &= A_1 + \cdots + A_t + R_t \nonumber\\
    &= -m_{sc} \expect{\abs{T_k(z)}^{2D}} + \kappa_3 m_{sc}^2 \expect{\abs{T_k(z)}^{2D}} + O(1/q^2) \expect{\abs{T_k(z)}^{2D-1}} \nonumber\\
    &\quad+ O(1/q^2) \expect{\abs{T_k(z)}^{2D-1}} + \cdots + O(1/q^{t-1})\expect{\abs{T_k(z)}^{2D-1}}\nonumber\\
    &\quad+ O(1/q^t) \cdot \| \abs{T_k(z)}^{2D-1} \|_\infty \nonumber\\
    &= -m_{sc} \expect{\abs{T_k(z)}^{2D}} + \kappa_3 m_{sc}^2 \expect{\abs{T_k(z)}^{2D}} + O(1/q^2) \expect{\abs{T_k(z)}^{2D-1}}\nonumber\\
    &\quad+ O(1/q^t) \cdot \| \abs{T_k(z)}^{2D-1} \|_\infty
\end{align}
Now,
\begin{align}
    (z+m_{sc}-\kappa_3 m^2)\expect{\abs{T_k(z)}^{2D}} = O(1/q^2) \expect{\abs{T_k(z)}^{2D-1}} + O(1/q^t) \cdot \| \abs{T_k(z)}^{2D-1} \|_\infty
\end{align}
Since $z+m_{sc} \gg \kappa_3 m^2$, $\abs{z+m_{sc}-\kappa_3 m^2}>c$ for some constant $c$. Also, the infinite norm is constant. Therefore, by dividing each side by $z+m_{sc}-\kappa_3 m^2$,
\begin{align}
    \expect{\abs{T_k(z)}^{2D}} = O(1/q^2) \expect{\abs{T_k(z)}^{2D-1}} + O(1/q^t)
\end{align}
Then by Young's inequality and Jensen's inequality,
\begin{align}
    \expect{\abs{T_k(z)}^{2D}} & \leq \frac{1}{2D} \cdot O(1/q^2)^{2D} + \frac{2D-1}{2D} \cdot \expect{\abs{T_k(z)}^{2D-1}}^{\frac{2D}{2D-1}} + O(1/q^t) \nonumber\\
    & \leq \frac{1}{2D} \cdot O(1/q^2)^{2D} + \frac{2D-1}{2D} \cdot \expect{\abs{T_k(z)}^{2D}} + O(1/q^t) \nonumber\\
    & = \frac{1}{2D} \cdot O(1/q^2)^{2D}+ \frac{2D-1}{2D} \cdot \expect{\abs{T_k(z)}^{2D}}
\end{align}
since $t>4D$ and
\begin{align}
    \expect{\abs{T_k(z)}^{2D}} \leq O(1/q^2)^{2D}
\end{align}
which implies
\begin{align}
    T_k \prec 1/q^2
\end{align}

\subsection{Proof of Lemma \ref{lemmaPs}} \label{subsec:lemmaPs}
In this subsection, we use Lemma \ref{lemma:powercounting} with Lemma \ref{lem:T_k} for power counting argument frequently. As in the proof of Lemma \ref{subsec:Tk} in subsection \ref{subsec:Tk}, we will consider the derivatives with respect to $H_{ik}$ when $i=k$ as same as $i\neq j$.

\subsubsection{Estimates for $C_1$}
In this subsection, we will estimate the order of
\begin{align}
    C_1+m_{sc}\expect{P(s)^{D-1}\overline{P(s)}^D \cdot s(z)}
\end{align}
By simple derivative calculation, we can know that $C_1$ = $C_{1,1}+C_{1,2}+C_{1,3}$ where
\begin{align}
    C_{1,1} &= \frac{D-1}{N} \sum_{i,j,k} \kappa_{ik}^{(2)} \expect{P(s)^{D-2}\overline{P(s)}^{D}\G{kj} \cdot \partial_{ik}P(s) } \nonumber\\
    &= \frac{D-1}{\sqrt{N}} \sum_{i,k} \kappa_{ik}^{(2)} \expect{P(s)^{D-2}\overline{P(s)}^{D} \cdot (z+m_{sc}-\kappa_3 m_{sc}^2) \cdot T_k (Q_{ik}+Q_{ki}) } \\
    C_{1,2} &= \frac{D}{N} \sum_{i,j,k} \kappa_{ik}^{(2)} \expect{P(s)^{D-1}\overline{P(s)}^{D-1}\G{kj} \cdot \overline{\partial_{ik}P(s)} } \nonumber\\
    &= \frac{D}{\sqrt{N}} \sum_{i,k} \kappa_{ik}^{(2)} \expect{P(s)^{D-2}\overline{P(s)}^{D} \cdot (\overline{z}+\overline{m_{sc}}-\overline{\kappa_3 m_{sc}^2}) \cdot T_k (\overline{Q_{ik}}+\overline{Q_{ki}}) } \\
    C_{1,3} &= \frac{1}{N} \sum_{i,j,k} \kappa_{ik}^{(2)} \expect{P(s)^{D-1}\overline{P(s)}^{D}(-\G{ki}\G{kj}-\G{kk}\G{ij}) } \nonumber\\
    &= \frac{1}{N} \sum_{i,j,k} \kappa_{ik}^{(2)} \expect{P(s)^{D-1}\overline{P(s)}^{D} \cdot (-\G{ki}\G{kj})} - m_{sc}\expect{P(s)^{D-1}\overline{P(s)}^D s(z)}
\end{align}
Then by moment counting,
\begin{align}
    &C_{1,1} = O\left(\frac{\sqrt{N}}{q^6}\right)\cdot \expect{\abs{P(s)}^{2D-1}},\ C_{1,2} = O\left(\frac{\sqrt{N}}{q^6}\right)\cdot \expect{\abs{P(s)}^{2D-1}},\\
    &C_{1,3} + m_{sc}\expect{P(s)^{D-1}\overline{P(s)}^D \cdot s(z)} = O\left(\frac{1}{q^2} \right)\cdot \expect{\abs{P(s)}^{2D-1}}
\end{align}
are negligible since $q \geq N^\phi > N^{1/8}$. Therefore,
\begin{align}
    C_1+m_{sc}\expect{P(s)^{D-1}\overline{P(s)}^D \cdot s(z)} = O\left(\frac{\sqrt{N}}{q^6}\right)\cdot\expect{\abs{P(s)}^{2D-1}}
\end{align}

\subsubsection{Estimates for $C_2$}
In this subsection, we will estimate the order of
\begin{align}
    C_2 - \kappa_3 m_{sc}^2 \expect{P(s)^{D-1}\overline{P(s)}^D \cdot s(z)} 
\end{align}
Before calculating the bound, we will define $C_{r,p} \ (r\geq2, p=1,2,3)$ which clearly satisfy $C_r = C_{r,1}+C_{r,2}+C_{r,3}$ as following:
\begin{align}
    C_{r,1} &= \frac{D-1}{N} \sum_{i,j,k} \frac{\kappa_{ik}^{(r+1)}}{r!} \expect{\partial_{ik}^{r-1}\left\{P(s)^{D-2}\overline{P(s)}^{D}\G{kj} \cdot \partial_{ik}P(s)\right\} } \\
    C_{r,2} &= \frac{D}{N} \sum_{i,j,k} \frac{\kappa_{ik}^{(r+1)}}{r!} \expect{\partial_{ik}^{r-1}\left\{P(s)^{D-1}\overline{P(s)}^{D-1}\G{kj} \cdot \overline{\partial_{ik}P(s)}\right\} } \\
    C_{r,3} &= \frac{1}{N} \sum_{i,j,k} \frac{\kappa_{ik}^{(r+1)}}{r!} \expect{\partial_{ik}^{r-1}\left\{P(s)^{D-1}\overline{P(s)}^{D}(-\G{ki}\G{kj}-\G{kk}\G{ij}) \right \} }
\end{align}
Then by derivative calculation, we can know that $C_{2,1}$ and $C_{2,2}$ have following bound
\begin{align}
    C_{2,1} = O\left(\frac{\sqrt{N}}{q^7}\right)\cdot\expect{\abs{P(s)}^{2D-1}},\ C_{2,2} = O\left(\frac{\sqrt{N}}{q^7}\right)\cdot\expect{\abs{P(s)}^{2D-1}}
\end{align}
For $C_{2,3}$, the only non-negligible term is
\begin{align}
    \frac{1}{N}\sum_{i,j,k}\frac{\kappa_{ik}^{(3)}}{2}\expect{P(s)^{D-1}\overline{P(s)}^D\G{kk}\G{ii}\G{kj}} = O\left(\frac{\sqrt{N}}{q^3}\right)\expect{\abs{P(s)}^{2D-1}}
\end{align}
which can be written as
\begin{align}
    &\frac{1}{N}\sum_{i,j,k}\frac{\kappa_{ik}^{(3)}}{2}\expect{P(s)^{D-1}\overline{P(s)}^D(\G{kk}-m_{sc})\G{ii}\G{kj}} \nonumber\\
    + &\frac{1}{N}\sum_{i,j,k}\frac{\kappa_{ik}^{(3)}}{2}\expect{P(s)^{D-1}\overline{P(s)}^D m_{sc}(\G{ii}-m_{sc})\G{kj}} \nonumber\\
    + &\frac{1}{N}\sum_{i,j,k}\frac{\kappa_{ik}^{(3)}}{2}\expect{P(s)^{D-1}\overline{P(s)}^D m_{sc}^2\G{kj}} \nonumber\\
    = &\  O\left(\frac{\sqrt{N}}{q^4}\right)\cdot\expect{\abs{P(s)}^{2D-1}} + \kappa_3 m_{sc}^2 \expect{P(s)^{D-1}\overline{P(s)}^D \cdot s(z)} 
\end{align}
Therefore,
\begin{align}
    C_{2,3} - \kappa_3 m_{sc}^2 \expect{P(s)^{D-1}\overline{P(s)}^D \cdot s(z)} =O\left(\frac{\sqrt{N}}{q^4}\right)\expect{\abs{P(s)}^{2D-1}}
\end{align}
which implies the desired term satisfies
\begin{align}
    C_2 - \kappa_3 m_{sc}^2 \expect{P(s)^{D-1}\overline{P(s)}^D \cdot s(z)} =O\left(\frac{\sqrt{N}}{q^4}\right)\expect{\abs{P(s)}^{2D-1}}
\end{align}

\subsubsection{Estimates for $C_r \ (r \geq 3)$}
In this subsection, we will change the two definitions \ref{effentry} and \ref{effneg} by replacing $\G{ij},\G{ji}$ into $\G{ik},\G{ki}$ since we differentiate the terms by $H_{ik}$ in this subsection. Then the number $n_c+n_s-1-\frac{1}{2}n_{end}$ in definition \ref{effneg} implies the order of $N$ in the term while $q$ is not counted.
\begin{theorem}
$C_{r,1}$ and $C_{r,2}$ are negligible for all integer $r \geq 2$.
\end{theorem}
\begin{proof}
By the previous subsection, we can know that $C_{2,1} = O \left( \sqrt{N}/q^7 \right)\cdot\expect{\abs{P(s)}^{2D-1}} $ is negligible. Now, suppose that $C_{r',1}$ is negligible for some integer $r' \geq 2$. Write $C_{r',1}$ as following:
\begin{align}
    C_{r',1} = \frac{1}{(r')!N}\sum_{i,j,k} \kappa_{ik}^{(r'+1)} \expect{\partial_{ik}^{r'-1} \left \{ P(s)^{D-2}\overline{P(s)}^{D}\G{kj} \cdot \partial_{ik}P(s) \right \}  }
\end{align}
and define $F_t(H_{ik})$ as
\begin{align}
    \sum_{t\in T} P(s)^{\alpha_t} \overline{P(s)}^{\beta_t} F_t(H_{ik}) = \partial_{ik}^{r'-1} \left \{ P(s)^{D-2}\overline{P(s)}^{D}\G{kj} \cdot \partial_{ik}P(s) \right \}
\end{align}
Then, $C_{r'+1,1}$ can be interpreted as
\begin{align}
    (r'+1)C_{r'+1,1} = L_1 + L_2 + L_3
\end{align}
where
\begin{align}
    L_1 &= \sum_{t\in T}\frac{\alpha_t}{(r')!N}\sum_{i,j,k} \kappa_{ik}^{(r'+2)} \expect{P(s)^{\alpha_t-1} \overline{P(s)}^{\beta_t} F_t(H_{ik}) \cdot \partial_{ik}P(s) } \\
    L_2 &= \sum_{t\in T}\frac{\beta_t}{(r')!N}\sum_{i,j,k} \kappa_{ik}^{(r'+2)} \expect{P(s)^{\alpha_t} \overline{P(s)}^{\beta_t-1} F_t(H_{ik}) \cdot \partial_{ik}\overline{P(s)} } \\
    L_3 &= \sum_{t\in T}\frac{1}{(r')!N}\sum_{i,j,k} \kappa_{ik}^{(r'+2)} \expect{P(s)^{\alpha_t} \overline{P(s)}^{\beta_t} F_t'(H_{ik}) }
\end{align}
Since $\partial_{ik}P(s)\prec 1/q^4$ and $\kappa_{ik}^{(r'+2)}$ has one more $q$ in the denominator than $\kappa_{ik}^{(r'+1)}$, the bound of $L_1$ has 5 more $q$ in the denominator than $C_{r',1}$ hence negligible. Similarly, $L_2$ is negligible also. For $L_3$, by the changed version of Lemma \ref{lemma_eff_rev}, the index of $N$ of order of $C_{r'+1}$ is smaller than that of $C_{r',1}$ (hence smaller than that of $C_{2,1}$ (equals to $1/2$)). Also, since $r' \geq 2$, $\kappa_{ik}^{(r'+2)}$ has more than two $q$ in the denominator. Plus, since there is only one entry of $G$ with index $j$ inside $L_3$, there exists at least one $T_\gamma$ ($\gamma$ is arbitrary) inside $F'(H_{ik})$ which means at least two $q$ appear again inside $F(H_{ik})$. Combining previous two facts, the order of $L_3$ is
\begin{align}
    L_3 = O\left( \frac{N^{p_1}}{q^{p_2}} \right),\quad p_1<\frac{1}{2},\ p_2\geq4
\end{align} 
which means $L_3$ is negligible since $\phi>1/8$. Therefore, since $L_1$, $L_2$, and $L_3$ are negligible, $C_{r'+1,1}$ is negligible also. By similar process, $C_{r'+1,2}$ is negligible.
\end{proof}

\begin{theorem}
$C_{r,3} = O(\sqrt{N}/q^{r+1})\expect{\abs{P(s)}^{2D-1}}$ for all integer $r \geq 3$.
\end{theorem}
\begin{proof}
Among the terms in
\begin{align}
    \frac{1}{N}\sum_{i,j,k} \frac{\kappa_{ik}^{(r+1)}}{r!} \expect{P(s)^{\alpha}\overline{P(s)}^{\beta}F(H_{ik})}
\end{align}
, if $F(H_{ik})$ has more than two non-diagonal entries, then such terms are negligible since the order of $N$ of their bounds are negative. Also, $F(H_{ik})$ have at least one non-diagonal entries containing the index $j$. Therefore, we can only consider the case that $F(H_{ik})$ contains the only non-diagonal entry (containing $j$) which leads the order of $N$ in their bounds are $1/2$. Consequently, if $\alpha<D-1$ or $\beta <D$, it means that at least one of derivative has effected to at least one $P(s)$ or $\overline{P(s)}$. Then since $\partial_{ik}P(s) \prec 1/q^4$, the order is smaller than $O(\sqrt{N}/q^4)$ which is negligible by the fact that $\phi>1/8$. In conclusion, the possible non-negligible term in $A_{r,3}$ is
\begin{align}
    \frac{1}{N}\sum_{i,j,k} \frac{\kappa_{ik}^{(r+1)}}{r!} \expect{P(s)^{D-1}\overline{P(s)}^{D}\partial_{ik}^r (\G{kj})}
\end{align}
Since $\partial_{ik}^r (\G{kj})$ contains $2r+2$ indices consisted of one $j$, r $i$, and r+1 $k$, the possible non-negligible terms are followings:
\begin{align}
    &\frac{1}{N}\sum_{i,j,k} \frac{\kappa_{ik}^{(r+1)}}{r!} \expect{P(s)^{D-1}\overline{P(s)}^{D}\G{ij}\times\text{(}r\text{ diagonal entries)}} \nonumber\\
    \text{or }&\frac{1}{N}\sum_{i,j,k} \frac{\kappa_{ik}^{(r+1)}}{r!} \expect{P(s)^{D-1}\overline{P(s)}^{D}\G{kj}\times\text{(}r\text{ diagonal entries)}}
\end{align}
which has order of
\begin{align}
    O\left(\frac{\sqrt{N}}{q^{r+1}}\right)\expect{\abs{P(s)}^{2D-1}}
\end{align}
\end{proof}

\subsubsection{Estimates for $R_t$}
The goal of the subsection is to prove
\begin{align}
    R_t = O(\sqrt{N}/q^t) \cdot \| \abs{P(s)}^{2D-1} \|_\infty
\end{align}
where $R_t$ is defined as
\begin{align}
    R_t = \sum_{i,j,k}\expect{\Omega_t\left(P(s)^{D-1}\overline{P(s)}^D \cdot \frac{1}{N}\G{kj}H_{ik} \right)}
\end{align}
where
\begin{align}
    \expect{\Omega_t\left(P(s)^{D-1}\overline{P(s)}^D \cdot \frac{1}{N}\G{kj}H_{ik} \right)} \leq C_t' \cdot \expect{\abs{H_{ij}}^{t+2}} \cdot \| \partial_{ij}^{t+1}(P(s)^{D-1}\overline{P(s)}^D \cdot \frac{1}{N}\G{kj}) \|_\infty
\end{align}
for some constant $C_t$. Then,
\begin{align}
    \abs{R_t} &\leq C_t' \sum_{i,j,k} \expect{\abs{H_{ij}}^{t+2}} \cdot \left \| \partial_{ij}^{t+1}(P(s)^{D-1}\overline{P(s)}^D \cdot \frac{1}{N}\G{kj}) \right \|_\infty \nonumber\\
    &= C_t' \sum_{i,j,k} \kappa_{ij}^{(t+2)} \cdot \left \| \partial_{ij}^{t+1}(P(s)^{D-1}\overline{P(s)}^D \cdot \frac{1}{N}\G{kj}) \right \|_\infty \label{Rteq2}
\end{align}
while the equation \ref{Rteq2} can be made by replacing the expectation symbol of $C_{t+1}$ with infinite norm. Also, the process of estimating the bound of $C_{t+1}$ contains no procedure using expectation. In other words, the expectation acts only as a linear operator in the process of estimating $C_{t+1}$. Therefore, we can calculate the bound of $R_t$ by same way as in $C_{t+1}$ and
\begin{align}
    R_t = O(\sqrt{N}/q^{t+2}) \cdot \| \abs{P(s)}^{2D-1} \|_\infty
\end{align}

\section{Proof of Lemma \ref{lemma:innerpro} and Lemma \ref{lemma:EM}} \label{sec:EM}

\subsection{Proof of Lemma \ref{lemma:innerpro}}
By the block structure of $M$, reordering the indices without loss of generality, there exists $T>0$ such that $N$ can be divided by $T$. For $i=1,\cdots,K$
    \begin{align}
        \mathbf{v}_i = (\mathbf{v}_{i1}\cdots\mathbf{v}_{i1} \vert \mathbf{v}_{i2}\cdots\mathbf{v}_{i2} \vert \cdots \vert
        \mathbf{v}_{iT}\cdots\mathbf{v}_{iT})^T
    \end{align}
    For the orthnormality of $\{\mathbf{v}_i, \ i=1,\cdots,K\}$, $\mathbf{v}_{ij} \ (i=1,\cdots,K, \ j,j'=1,\cdots,T)$ should satisfy
    \begin{align}
        \frac{N}{T}\sum_{j=1}^T \mathbf{v}_{ij}^2 =1, \ \frac{N}{T}\sum_{j=1}^T \mathbf{v}_{ij}\mathbf{v}_{i'j} = 0
    \end{align}
    if $i\neq i'$. Define $s_{\alpha\beta}$ which is partial sum calculated by $\alpha$-th block of $\mathbf{v}_i$ and $\beta$-th block of $\mathbf{v}_j$ inside $\left\langle \mathbf{v}_i,G(z)\mathbf{v}_j\right \rangle$ by
    \begin{align}
        s_{\alpha\beta} = \mathbf{v}_{i\alpha}\mathbf{v}_{j\beta}\cdot\sum_{u=1+\frac{N}{T}(\alpha-1)}^{\frac{N}{T}\alpha}\sum_{v=1+\frac{N}{T}(\beta-1)}^{\frac{N}{T}\beta}\G{uv}
    \end{align}
    for $\alpha,\beta = 1,\cdots,T$. If $\alpha \neq \beta$, there is no $\G{uv}$ inside $s_{\alpha\beta}$ such that $u=v$ which means there is no diagonal entries in $s_{\alpha\beta}$. Therefore, $s_{\alpha\beta}$ with $\alpha\neq\beta$ are negligible as $N\rightarrow\infty$ since they are consisted of non-diagonal entries of $G$. Otherwise, for $\alpha = 1,\cdots,T$,
    \begin{align}
        s_{\alpha\alpha} = \mathbf{v}_{i\alpha}\mathbf{v}_{j\alpha}\cdot\sum_{u,v=1+\frac{N}{T}(\alpha-1)}^{\frac{N}{T}\alpha}\G{uv}
    \end{align}
    has same behavior with
    \begin{align}
        \mathbf{v}_{i\alpha}\mathbf{v}_{j\alpha}\cdot\sum_{u,v=1}^{\frac{N}{T}}\G{uv}
        =\frac{N}{T}\mathbf{v}_{i\alpha}\mathbf{v}_{j\alpha}\cdot \frac{1}{N/T}\sum_{u,v=1}^{\frac{N}{T}}\G{uv}
        =\frac{N}{T}\mathbf{v}_{i\alpha}\mathbf{v}_{j\alpha}\cdot\left(m_{sc}+O\left(\frac{\sqrt{N/T}}{q^4}+\frac{1}{q^2}\right)\right)
    \end{align}
    while the last equality is by Lemma \ref{lemma:local law cgSBM}. Then by summing all $s_{\alpha\beta}$,
    \begin{align}
        \scalar{\mathbf{v}_i}{G(z)\mathbf{v}_j} = \sum_{\alpha=1}^T s_{\alpha\alpha} =& \frac{N}{T}\sum_{\alpha=1}^T \mathbf{v}_{i\alpha}\mathbf{v}_{j\alpha} \cdot \scalar{\mathbf{v}_i}{\mathbf{v}_j} \left(m_{sc}+O\left(\frac{\sqrt{N/T}}{q^4}+\frac{1}{q^2}\right)\right) \\
        =& \scalar{\mathbf{v}_i}{\mathbf{v}_j} \left(m_{sc}+O\left(\frac{\sqrt{N}}{q^4}+\frac{1}{q^2}\right)\right)
    \end{align}

\subsection{Proof of Lemma \ref{lemma:EM}}
Without loss of generality, re-order the indices $\mathbb{E}M$ so that $i_1<i_2$ for all $i_1 \in C_{j_1}$ and $i_2 \in C_{j_2}$ with $j_1<j_2$ so that $\mathbb{E}M$ becomes block matrix. By (\ref{eq:M}), the expectation matrix $\mathbb{E}M$ satisfies
    \begin{align}
        \mathbb{E}M_{ij} = \begin{cases}
            \frac{p_s-p_a}{\sigma} = \frac{1}{\sigma} \cdot \frac{K-1}{K} (p_s-p_d) & (i\sim j) \\
            \frac{p_d-p_a}{\sigma} = \frac{1}{\sigma} \cdot \frac{1}{K} (p_d-p_s) & (i\not\sim j) 
        \end{cases}
        = \begin{cases}
            (K-1)x & (i\sim j) \\
            -x & (i\not\sim j)
        \end{cases}
    \end{align}
    where $x$ is defined as $x=\frac{1}{\sigma K}(p_s-p_d)$. Then $\mathbb{E}M$ is a block matrix as following:
    \begin{align}
        \begin{pmatrix}
        \begin{matrix}
        (K-1)x & \cdots & (K-1)x \\
        \vdots & & \vdots \\
        (K-1)x & \cdots & (K-1)x
        \end{matrix}
        & \rvline & 
        \begin{matrix}
        -x & \cdots & -x \\
        \vdots & & \vdots \\
        -x & \cdots & -x
        \end{matrix}
        & \rvline & \cdots & \rvline &
        \begin{matrix}
        -x & \cdots & -x \\
        \vdots & & \vdots \\
        -x & \cdots & -x
        \end{matrix} \\
        \hline
        \begin{matrix}
        -x & \cdots & -x \\
        \vdots & & \vdots \\
        -x & \cdots & -x
        \end{matrix}
        & \rvline & 
        \begin{matrix}
        (K-1)x & \cdots & (K-1)x \\
        \vdots & & \vdots \\
        (K-1)x & \cdots & (K-1)x
        \end{matrix}
        & \rvline & \cdots & \rvline &
        \begin{matrix}
        -x & \cdots & -x \\
        \vdots & & \vdots \\
        -x & \cdots & -x
        \end{matrix}
        \\
        \hline
        \vdots & \rvline & \vdots & \rvline & & \rvline & \vdots \\
        \hline
        \begin{matrix}
        -x & \cdots & -x \\
        \vdots & & \vdots \\
        -x & \cdots & -x
        \end{matrix}
        & \rvline & 
        \begin{matrix}
        -x & \cdots & -x \\
        \vdots & & \vdots \\
        -x & \cdots & -x
        \end{matrix}
        & \rvline & \cdots & \rvline &
        \begin{matrix}
        (K-1)x & \cdots & (K-1)x \\
        \vdots & & \vdots \\
        (K-1)x & \cdots & (K-1)x
        \end{matrix}
        \end{pmatrix}
    \end{align}
    Since the rank of a matrix is the number of nonzero rows of its reduced row-echelon form, $rank(A) = rank(A')$ while the $K\times K$ matrix $A'$ is defined as
    \begin{align}
        A' = \begin{pmatrix}
            (K-1)x & -x & \cdots & -x \\
            -x & (K-1)x & \cdots & -x \\
            \vdots & \vdots & & \vdots \\
            -x & -x & \cdots & (K-1)x
        \end{pmatrix}
    \end{align}
    while $rank(A')=K-1$ since $null(A')=span(\{ (1,1,\cdots,1)^T \})$ which implies $nullity(A')=1$. Therefore, the rank of $\mathbb{E}M$ is $K-1$. Define $N$-dimensional vectors $\mathbf{w}_i \ (i=1,\cdots,K-1)$ by
    \begin{align}
        \mathbf{w}_i(j) = \begin{cases}
            1 & \text{if  } (i-1)K/N \leq j < iK/N \\
            -1 & \text{if  } iK/N \leq j < (i+1)K/N \\
            0 & \text{else}
        \end{cases}
    \end{align}
    Then $(A-(Nx)I_N)\mathbf{w}_i=\mathbf{0}$ for all $i=1,\cdots,K-1$ and $\{\mathbf{w}_i, \ i=1,\cdots,K-1\}$ is independent. Since $rank(A)=K-1$, the non-zero eigenvalues of $A$ is $Nx = \frac{N}{\sigma K}(p_s-p_d)$ with multiplicity $K-1$. Consider the orthonormal eigenvectors $\mathbf{v}_i \ (i=1,\cdots,K-1)$ of $\mathbb{E}M$ with respect to eigenvalue $\frac{N}{\sigma K}(p_s-p_d)$ generated by using the Gram-Schmidt process to $\{\mathbf{w}_i, \ i=1,\cdots,K-1\}$. Since $\mathbf{w}_i \ (i=1,\cdots,K-1)$ satisfy the block structure properties described in the lemma and the Gram-Schmidt process cannot affect to this property, $\mathbf{v}_i \ (i=1,\cdots,K-1)$ satisfies such property also.

\end{document}